\documentclass[11pt]{amsart}

\usepackage{amssymb}
\usepackage{amsmath}
  \usepackage{paralist}
  \usepackage{graphics} %% add this and next lines if pictures should be in esp format
  \usepackage{epsfig} %For pictures: screened artwork should be set up with an 85 or 100 line screen
\usepackage{graphicx}  \usepackage{epstopdf}%This is to transfer .eps figure to .pdf figure; please compile your paper using PDFLeTex or PDFTeXify.
\usepackage{enumitem}
\usepackage{comment}
\excludecomment{tohide}

 \usepackage[colorlinks=true]{hyperref}
\hypersetup{urlcolor=blue, citecolor=red}

\newtheorem{thmm}{Theorem}
\newtheorem*{TheoremD'}{Theorem D'}

\newtheorem*{TheoremE'}{Theorem E'}

\newtheorem*{CorollaryE1}{Corollary E.1}

\newtheorem{theorem}{Theorem}[section]
\newtheorem{corollary}[theorem]{Corollary}

\newtheorem*{main*}{Main Theorem}
\newtheorem{lemma}[theorem]{Lemma}
\newtheorem{proposition}[theorem]{Proposition}

\theoremstyle{definition}
\newtheorem{definition}[theorem]{Definition}
\newtheorem{remark}[theorem]{Remark}

\newcommand{\CCC}{\mathcal{C}}
\newcommand{\DDD}{\mathcal{D}}
\newcommand{\PPP}{\mathcal{P}}
\newcommand{\SSS}{\mathcal{S}}
\newcommand{\BBB}{\mathcal{B}}
\newcommand{\UUU}{\mathcal{U}}
\newcommand{\GGG}{\mathcal{G}}

\newcommand{\CC}{\mathbb{C}}

\def\P{{\mathcal P}}

\def\CC{{\mathcal C}}

\def\p{{\partial}}
\def\M{{\mathcal M}}
\newcommand{\F}{\mathcal{F}}
\def\CG{{\mathcal G}}
\newcommand{\UR}{\mathcal{UR}}
\newcommand{\URR}{\mathcal{URR}}
\newcommand{\R}{\mathcal{R}}

\def\pX{{\partial X}}
\def\RR{{\mathbb R}}
\def\NN{{\mathbb N}}

\def\ZZ{{\mathbb Z}}

\def\inj{{\text{inj}}}
\def\diam{\mathop{\hbox{{\rm diam}}}}

\def\diam{\mathop{\hbox{{\rm diam}}}}

\def\supp{\mathop{\hbox{{\rm supp}}}}

\def\loc{{\mathop{\hbox{\footnotesize  \rm loc}}}}
\def\supp{\mathop{\hbox{{\rm supp}}}}

\def\lv{{\underline v}}
\def\lc{{\underline c}}

\def\lB{{\underline B}}
\def\lS{{\underline S}}

\def\a{\alpha}
\def\b{\beta}
\def\c{\gamma}   \def\C{\Gamma}
\def\d{\delta}   %\def\C{\Gamma}
\def\l{\lambda}   \def\L{\Lambda}
 \def\e{\epsilon}
   
\def\t{\theta}
\def\lv{{\underline v}}
\def\ae{\text{-a.e.}\ }
\def\bP{\textbf{P}}
\def\bF{\textbf{F}}

\newcommand{\Per}{\mathrm{Per}}

\title[Running heading with forty characters or less]
      {Patterson-Sullivan construction of equilibrium states and weighted counting in nonpositive curvature}

\author[first-name1 last-name1 and first-name2 last-name2]{Weisheng Wu}

\subjclass{}
 \keywords{}

\address{School of Mathematical Sciences, Xiamen University, Xiamen, 361005, P.R. China}
\email{wuweisheng@xmu.edu.cn}

\begin{document}
\begin{abstract}
Consider the geodesic flow on a closed rank one manifold of nonpositive curvature. For certain H\"{o}lder continuous potential, there exists a unique equilibrium state by \cite{BCFT}. In this paper, we introduce the notions of core limit set, regular radial limit set and uniformly recurrent and regular vectors, and then construct a family of Patterson-Sullivan measures on the boundary at infinity in two separate settings. Then we give an explicit construction of the above unique equilibrium state using Patterson-Sullivan measures. This enables us to prove the Bernoulli property of the equilibrium states. Using the Patterson-Sullivan construction and mixing properties of equilibrium states, we count the number of free homotopy classes with weights in nonpositive curvature.  
\end{abstract}

\maketitle
\markboth{Patterson-Sullivan construction}
{W. Wu}
\renewcommand{\sectionmark}[1]{}

\tableofcontents

\section{Introduction}
The study of geodesic flow lies in the crossing field of dynamical systems and differential geometry. The geodesic flow on closed manifolds of negative (sectional) curvature everywhere is a prime example of uniformly hyperbolic dynamical systems. Hopf \cite{Ho0, Ho1}, Anosov and Sinai \cite{An, AnS} proved the ergodicity of the geodesic flow with respect to the Liouville measure. On the other hand, the ergodic theory and thermodynamical formalism have deep applications in rigidity and counting problems in geometry. For instance, Margulis in his thesis \cite{Mar2} gave an explicit construction of measure of maximal entropy (MME for short) and Bowen \cite{Bo2} proved the uniqueness of MME for the geodesic flow in negative curvature. Nowdays the unique MME is called Bowen-Margulis measure, which exhibits nice mixing properties and local product structure with respect to the stable/unstable manifolds. Margulis \cite{Mar2} then derived an asymptotic formula for the growth of the number of closed geodesics:
\begin{equation*}
 \lim_{t\to \infty}\#P(t)/\frac{e^{ht}}{ht}=1
\end{equation*}
where $P(t)$ is the set of closed geodesics with length no more than $t$, $h$ is the topological entropy of the geodesic flow. Later, Parry and Pollicott \cite{PP} proved this formula using symbolic coding for the geodesic flow (and more generally for Axiom A flows). This formula is also called prime geodesic theorem since it resembles the formula in the prime number theorem.     

The study of geodesic flow in nonpositive curvature is more delicate due to the existence of singular geodesics and flat strips.  The well known higher rank rigidity theorem \cite{Bal, BS1, BS2} says that a simply connected irreducible manifold of nonpositive curvature and rank greater than one must be isometric to a locally symmetric space of noncompact type. On the other hand, it is believed that the geodesic flow on a closed rank one manifold inherits most properties in negative curvature. Pesin established his theory of nonuniform hyperbolicity during the study of such geodesic flow  (cf. \cite{BP}).  Nevertheless, some problems in rank one case are much more difficult. For instance, the ergodicity of the geodesic flow with respect to the Liouville measure is widely open (cf. \cite{BP, Wu,WLW}).

It had been a long-standing problem to extending the uniqueness of MME and prime geodesic theorem from negative curvature to nonpositive curvature. The uniqueness of MME was conjectured by Burns and Katok \cite{BuKa} and finally proved by Knieper \cite{Kn}. Knieper \cite{Kn} built the MME via Patterson-Sullivan measures (cf. \cite{Pat, Sul}) on the boundary at infinity. Knieper \cite{Kn1} also obtained lower and upper bounds on the number of free homotopy classes. Katok launched a program to prove the Margulis type asymptotics for the number of free homotopy classes (cf. \cite{We}). Ricks \cite{Ri} made the breakthrough and proved the Margulis type asymptotic formula for the number of free homotopy classes, for compact CAT(0) spaces which include closed rank one manifolds of nonpositive curvature. Further, Climenhaga, Knieper and War \cite{CKW1, CKW2} established the uniqueness of MME and Margulis type asymptotic formula for certain closed manifolds without conjugate points. Wu \cite{Wu2} established the formula for closed rank one manifolds without focal points.

As a far reaching generalization of entropy, Ruelle \cite{Ru} and Walters \cite{Wa} introduced the notion of topological pressure to dynamical systems and established a variational principle for it. Pressure and equilibrium states constitute the main components of the thermodynamical formalism and play important roles in the study of the geodesic flow (cf. \cite{PP1}). In negative curvature, the uniqueness of equilibrium states are again proved by Bowen \cite{Bo3,Bo4}. In a monograph \cite{PPS}, Paulin, Pollicott and Schapira obtained the uniqueness of equilibrium states, weighted prime geodesic theorem and many other results for certain noncompact manifolds of negative curvature. These results are based on the Patterson-Sullivan construction for equilibrium states.

In nonpositive curvature case, the first major development is the following theorem by Burns, Climenhaga, Fisher and Thompson.
\begin{theorem}(\cite[Theorem A]{BCFT})\label{bcft}
Let $\CG=(g^t)_{t\in \RR}$ be the geodesic flow over a closed rank one manifold $M$ and let $F: SM\to \RR$ be a H\"{o}lder continuous potential. If $P(\text{Sing}, F)< P(F)$, then $F$ has
a unique equilibrium state.  This equilibrium state is hyperbolic, fully supported, and is the weak$^*$ limit of weighted regular closed geodesics.
\end{theorem}

Call and Thompson \cite{CT} proved the Kolmogorov property (and hence mixing property) of the above equilibrium states. In \cite[Subsection 7.3]{CT}, the authors discussed the power of the local product structure for equilibrium states. The local product structure of equilibrium states at the symbolic level is obtained by Araujo, Lima and Poletti \cite{ALP}. In fact, Lima and Poletti \cite{LP} gave a new proof of Theorem \ref{bcft} using symbolic dynamics. Recently, using Gibbs property of equilibrium states, Call, Constantine, Erchenko, Sawyer and Work \cite[Theorem D]{CCESW} obtained the local product structure at the dynamical level for certain equilibrium states.  As commented at the end of \cite{CT}, the geometric local product structure, that is the Patterson-Sullivan construction of equilibriums states, is not known. 

We resolve the above problem in this paper. More precisely, Patterson-Sullivan construction of equilibrium states is given in nonpositive curvature in two separate settings. Then Bernoulli property of equilibrium states is obtained. As a highly nontrivial application, we establish equidistribution of closed geodesics and asymptotic inequalities for the number of weighted closed geodesics. We believe that more results in \cite{PPS} could be extended to nonpositive curvature setting using our results in this paper.

\subsection{Statement of main results}
Let $\CG=(g^t)_{t\in \RR}$ be the geodesic flow over a closed rank one manifold $M$ and $F: SM\to \RR$ a H\"{o}lder continuous potential. If $P(\text{Sing}, F)< P(F)$, then $F$ has
a unique equilibrium state by Theorem \ref{bcft}. Let $\C$ be the fundamental group of $M$,  $X$ the universal cover of $M$ and $\pX$ the boundary at infinity of $X$.

To construct Patterson-Sullivan measures on $\pX$, we start with Poincar\'{e} series. From a dynamical point of view, since the unit sphere at a reference point is a submanifold in the unit tangent bundle, we are expecting that separated sets lying on this submanifold give the topological pressure of the geodesic flow. In negative curvature, this is guaranteed by bounded distortion (or so-called Bowen property)  of a H\"{o}lder potential coming from uniform hyperbolicity. Furthermore, the definition of Gibbs cocycle relies on the bounded distortion. To bypass this obstruction, we establish bounded distortion along ``uniformly regular'' orbits. To do so, we introduce the following three types of ``uniform regularity'':

\begin{itemize}
\item Core limit set $\L_c(\C)$: the set of limit points in $\pX$ accumulated by good orbit segments, i.e., the core part $\GGG$ in the $(\P,\GGG,\SSS)$-decomposition defined in \cite{BCFT}.
\item Regular radial limit set $\L_r(\C)$: the set of limit points in $\pX$ pointed by a geodesic ray starting at the reference point, and entering regular compact regions in $X$ infinitely many times.
\item The set $\URR$ of uniformly recurrent and regular vectors. 
\end{itemize}
The core limit set originates from the $(\P,\GGG,\SSS)$-decomposition in \cite{BCFT}. Pliss time is essentially used to obtain bounded distortion in this case. The regular radial limit set is motivated by the notion of strongly positive recurrence (SPR for short) studied for example in \cite{GST}. Here we consider geodesic rays entering \textit{regular} compact regions infinitely many times. The idea behind uniformly recurrent and regular vectors is from \cite{BBE}. It plays some roles of Pesin sets, with the main difference that we have exponential decay of distance under the geodesic flow of any two points in the \textit{global} stable manifolds, not just Pesin local stable manifolds. In some sense, these three notions characterize ``uniformly regularity'' in dynamical, geometric, and ergodic ways respectively.

The following proposition shows the relation between the critical exponent $\d_F$ of $\Gamma$ relative to $F$, and the topological pressure $P(F)$ of the geodesic flow on $M=X/\Gamma$. The use of uniformly recurrent and regular vectors is crucial in the proof. 
\begin{proposition}\label{pro9}
Let $M$ be a closed rank one manifold of nonpositive curvature and $X$ its universal cover. Assume that $P(\text{Sing},F)<P(F)$. Then $P(F)=\delta_F$.
\end{proposition}

In order to construct a family of Patterson-Sullivan measures on $\pX$, our strategy is to show that the complement of core limit set or regular radial limit set is null with respect to a measure constructed using reference point. At this step, we need impose one of the following two conditions:
\begin{itemize}
    \item Condition (A): There exists a constant $C>0$ such that 
    $$\sum_{\c\in \C, n-1< d(o, \c o)\le n}e^{\int_o^{\c o}F}\ge Ce^{n\d_F}, \quad \forall n\in \NN.$$ 
        \item Condition (B): $F$ is constant in a neighborhood of the singular set.
\end{itemize}

Condition (A) is fulfilled for $F\equiv 0$ in nonpositive curvature \cite[Theorem A]{Kn1} and for any H\"{o}lder potential $F$ in negative curvature by \cite[Corollary 9.10]{PPS}. It is similar to the condition of purely exponential volume growth in \cite{Zi} in the case that $F\equiv 0$ for noncompact manifolds. Corollary E.1 below shows that it is also a necessary condition to Patterson-Sullivan construction. 

Condition (B) is introduced and extensively studied in \cite{BCFT}. It implies the pressure gap  $P(F,\text{Sing})<P(F)$ by \cite[Theorem B]{BCFT}.  In \cite[Theorem D]{CCESW}, it is proved that equilibrium states under Condition (B) have local product structure at the dynamical level. 

We show that under Condition (A) the complement of core limit set has null measure, and under Condition (B) bounded distortion of the potential also holds with respect to the complement of regular radial limit set. This allows us to construct a family of measures on $\pX$, which forms a $\d_F$-dimensional Busemann density.
\begin{thmm}\label{ps}
Let $M$ be a closed rank one manifold of nonpositive curvature and $X$ its universal cover. Suppose that $F:SM\to \RR$ is a H\"{o}lder continuous potential satisfying $P(\text{Sing},F)<P(F)$, and either Condition (A) or Condition (B) holds. Then there exists a family of $\d_F$-dimensional Busemann density $\{\mu_{F,q}\}_{q\in X}$ on $\pX$, that is,
\begin{enumerate}
  \item $\mu_{F,\gamma q}(\gamma A)=\mu_{F,q}(A)$ for any $\c\in \Gamma$ and any Borel set $A\subset \pX$;
  \item $\frac{d\mu_{F,q}}{d\mu_{F,p}}(\xi)=e^{-C_{F-\d_F,\xi}(q,p)}$ for almost every $\xi\in \pX$.
\end{enumerate}    
\end{thmm}

Since the Busemann cocycle $C_{F-\d_F,\xi}(q,p)$ is defined almost everywhere, we can define a measure $\bar\mu_F$ on $\p^2X$, the set of pairs $(\xi,\eta)\in\pX\times \pX$ which can be connected by a geodesic $c_v$, as follows.
$$d \bar \mu_{F}(\xi,\eta) = e^{C_{F\circ \iota-\d_F,\xi}(o, \pi(v))+C_{F-\d_F,\eta}(o, \pi(v))}d\mu_{F\circ \iota,o}(\xi) d\mu_{F,o}(\eta).$$
Here $\iota: SM\to SM, v\mapsto -v$ is the flip map. It is not evident that the density above is bounded, so $\bar\mu_F$ might not be a Radon measure. Our strategy is to restrict $\bar\mu_F$ first to a subset with ``uniformly bounded'' density and then to a ``uniformly regular'' subset of $\p^2X$ to get Radon measures. Such a $\C$-invariant Radon measure induces a $g^t$-invariant probability measure $\mu_F$ on $SM$.
\begin{thmm}\label{es}
Let $M$ be a closed rank one manifold of nonpositive curvature, $X$ its universal cover and $F:SM\to \RR$ a H\"{o}lder continuous potential satisfying $P(\text{Sing},F)<P(F)$. Assume that there exists a family of $\d_F$-dimensional Busemann density $\{\mu_{F,q}\}_{q\in X}$ on $\pX$. Then $\mu_F$ is the unique equilibrium state for $F$.
\end{thmm}
The main difficulty in the proof of Theorem \ref{es} is the absence of classical shadow lemma for Busemann density $\{\mu_{F,q}\}_{q\in X}$. Nevertheless, we obtain one inequality in the shadow lemma over uniformly recurrent and regular vectors, which is enough for us to estimate the measure of Bowen balls from above. The proof of Theorem \ref{es} is completed by Katok entropy formula and the ergodicity of $\mu_F$.

As a first application of Patterson-Sullivan construction of equilibrium states, we prove that the equilibrium states are Bernoulli. The result has been proved recently using symbolic coding \cite[Corollary 1.3]{ALP}. We provide a more direct and geometric proof using Ornstein-Weiss argument \cite{OW}.
\begin{thmm}\label{bernoulli}
Let $M$ be a closed rank one manifold of nonpositive curvature, $X$ its universal cover and $F:SM\to \RR$ a H\"{o}lder continuous potential satisfying $P(\text{Sing},F)<P(F)$. Assume that there exists a family of $\d_F$-dimensional Busemann density $\{\mu_{F,q}\}_{q\in X}$ on $\pX$. Then the unique equilibrium state $\mu_F$ is Bernoulli.
\end{thmm}

As another application, we obtain the following equidistribution result extending the original one by Bowen \cite{Bo33}. Since one free-homotopic class may contain uncountably many closed geodesics, we will pick one geodesic from each class.
\begin{thmm}\label{equi0}
Let $M$ be a closed rank one manifold of nonpositive curvature and $F:SM\to \RR$ a H\"{o}lder continuous potential satisfying $P(\text{Sing},F)<P(F)$. Assume that there exists a family of $\d_F$-dimensional Busemann density $\{\mu_{F,q}\}_{q\in X}$ on $\pX$. Suppose that $\e\in (0, \inj (M)/2)$ is fixed where $\inj (M)$ is the injectivity radius of $M$. For $t > 0$, let $C(t)$ be any maximal set of pairwise non-free-homotopic
closed geodesics with lengths in $(t-\e, t]$, and define the measure
$$\nu_{F,t}:=\frac{1}{\sum_{c\in C(t)}e^{\Per_F(c)}}\sum_{c\in C(t)}e^{\Per_F(c)}\frac{Leb_c}{t}$$
where $Leb_c$ is the Lebesgue measure along the curve $\dot c$ in the unit
tangent bundle $SM$, and $\Per_F(c):=\int_0^{T}F(\dot c(s))ds$ with $T$ being the length of $c$.

Then the measures $\nu_{F,t}$ converge in the
weak$^*$ topology to the unique equilibrium state $\mu_F$ as $t\to \infty$.
\end{thmm}

Last but not the least, we obtain the following asymptotic inequalities on weighted closed geodesics.
\begin{thmm}\label{margulis}
Let $M$ be a closed rank one manifold of nonpositive curvature, and $F:SM\to \RR$ a H\"{o}lder continuous potential satisfying $P(\text{Sing},F)<P(F)$ and $P(F)>0$. Assume that there exists a family of $\d_F$-dimensional Busemann density $\{\mu_{F,q}\}_{q\in X}$ on $\pX$. Let $P(t)$ denote the set of free-homotopy classes containing a closed geodesic with length at most $t$. Then\footnote{The notation $f\lesssim g$ (resp. $f\gtrsim g$) means $\limsup_{t\to \infty}\frac{f(t)}{g(t)}\le 1$ (rsp. $\liminf_{t\to \infty}\frac{f(t)}{g(t)}\ge 1$).}
\begin{equation*}
\begin{aligned}
\sum_{c\in P(t)}e^{\Per_F(c)}&\lesssim C\frac{e^{\d_Ft}}{\d_Ft},\\
\sum_{c\in P(t)}e^{\Per_F(c)}&\gtrsim \frac{e^{\d_Ft}}{\d_Ft}
\end{aligned}
\end{equation*}
where $1\le C\le +\infty$ (see Remark \ref{CB1}).
\end{thmm}
We are expecting $C=1$ so that the Margulis type asymptotic formula holds. Nevertheless, our method provides an effective upper bound which illuminates that the rate of $\nu_{F,t}$ converging to $\mu_F$ on ``uniformly regular'' sets is relevant. By \cite[Proposition 6.4]{BCFT} and using pressure gap condition, there exists a constant $\b>0$ such that 
$$\frac{\b e^{\d_Ft}}{t}\le \sum_{c\in C(t)}e^{\Per_F(c)}\le \b^{-1}e^{\d_Ft}.$$
New bounds are obtained in Proposition \ref{sumup}. The above upper bound also indicates the possibility $C=+\infty$.

In the proof of Theorem \ref{margulis}, we encounter difficulties caused by ``genuine''  nonuniform hyperbolicity, not only from the dynamical structure, but also from distortion of the potential. We consider the rectangles formed by stable/unstable manifolds of ``uniformly regular'' points inside a geometric flow box.  Using local product structure of the rectangle and mixing properties of equilibrium states, we calculate the number of intersections of the rectangle and its images under the geodesic flow. We need that these estimates are uniform with respect to the size of flow box, which can be achieved if the rectangle has a large percentage of measure inside the flow box. As Besicovitch covering lemma or Lebesgue density theorem does not apply directly to flow boxes, instead we construct a partition of ``uniformly regular'' set by flow boxes. In this way, we obtain a sequence of boxes with size shrinking to zero, inside which the rectangle accounts for a large proportion in measure. 

Our weighted counting result also leads to the following consequence, so Condition (A) is a necessary condition for the existence of $\d_F$-dimensional Busemann density.
\begin{CorollaryE1}\label{necessary}
Let $M$ be a closed rank one manifold of nonpositive curvature and $X$ its universal cover. Suppose that $F:SM\to \RR$ is a H\"{o}lder continuous potential satisfying $P(\text{Sing},F)<P(F)$. Then there exists a family of $\d_F$-dimensional Busemann density $\{\mu_{F,q}\}_{q\in X}$ on $\pX$ if and only if Condition (A) holds, that is,
$$\sum_{\c\in \C, n-1< d(o, \c o)\le n}e^{\int_o^{\c o}F}\ge Ce^{n\d_F}, \quad \forall n\in \NN$$
for some constant $C>0$.
\end{CorollaryE1}

The paper is organized as follows. In Section 2, we present some preliminaries on equilibrium states and geodesic flow in nonpositive curvature, and then introduce three types of ``uniformly regular'' points. We establish bounded distortion of the potential and define Busemann cocycles. In Section 3, Patterson-Sullivan measures on the boundary at infinity are constructed under Condition (A) and Condition (B) separately. This proves Theorem \ref{ps}. Some properties such as a half shadow lemma are proved for Patterson-Sullivan measures. In Section 4, we construct a measure using Patterson-Sullivan measures and show that it coincides with the equilibrium state, proving Theorem \ref{es}. In Section 5, Theorem \ref{bernoulli}, that is, the Bernoulli property of equilibrium states is proved. In Section 6, we prove the equidistibution of closed geodesics Theorem \ref{equi0}, weighted counting Theorem \ref{margulis} and Corollary E.1. In the Appendix, we provide proofs of some technical lemmas on uniformly recurrent and regular vectors.

\section{Preliminaries}

\subsection{Equilibrium states in nonpositive curvature}
Suppose that $(M,g)$ is a $C^{\infty}$ closed Riemannian manifold,
where $g$ is a Riemannian metric of nonpositive curvature. Let $\pi: SM\to M$ be the unit tangent bundle over $M$. For each $v\in SM$,
we always denote by $c_{v}: \RR\to M$ the unique geodesic on $M$ satisfying the initial conditions $c_v(0)=\pi(v)$ and $\dot c_v(0)=v$. The geodesic flow $\CG=(g^{t})_{t\in\mathbb{R}}$ is defined as:
\[
g^{t}: SM \rightarrow SM, \qquad v \mapsto \dot c_{v}(t),\ \ \ \ \forall\ t\in \RR .
\]

\subsubsection{Pressure and equilibrium states}
Let $d$ denote the distance function on $M$ induced by Riemannian metric. The Knieper metric $d_K$ on $SM$ is defined by 
$d_K(v,w):=\max\{d(c_v(t),c_w(t)): 0\le t\le 1\}$. Then we define a dynamical metric on $SM$ as
 $$d_t(v,w):=\max_{0\le s\le t}d_K(g^sv, g^sw), \quad \forall v,w\in SM.$$
Given $\e>0$ and $t>0$, the $t$-th \emph{Bowen ball} centered at $v\in SM$ is defined as
 $$B_t(v,\e):=\{w\in SM: d_t(v,w)<\e\}.$$
Let $\mathcal{C}\subset SM\times \mathbb{R}^+$ and $\CC_t:=\{v\in SM: (v,t)\in \CC\}$. $E\subset \mathcal{C}_t$ is \emph{$(t,\e)$-separated}, if for any $(v,t), (w,t)\in E$ with $v\neq w$, one has $d_t(v,w)>\e$. 

A potential $F\in C(SM, \mathbb{R})$ is a continuous real function on $SM$. Given a scale $\e>0$ and time $t>0$, the representative information of $F$ over a $t$-th Bowen ball centered at $v$ is given by
$$
\F_{\e}(v,t):=\sup_{w\in B_t(v,\e)}\int_0^t F(g^sw)ds.
$$
We also write $\F(v,t):=\F_0(v,t)$. The notation with parameter $\e$ follows from \cite{CT16} and will be used in Subsection 3.2.

\begin{definition}
Given $\mathcal{C}\subset SM\times \mathbb{R}^+$ and $t,\delta,\e>0$, the (separated)\textit{ partition function} of $F$ is defined by
$$
\Lambda_t(\mathcal{C},F,\delta,\e):=\sup\Big\{\sum_{(v,t)\in E}e^{\F_{\e}(v,t)}: E\subset \mathcal{C}_t \text{ is } (t,\d)\text{-separated}\Big\}.
$$
Then the \textit{pressure} of $F$ over $\mathcal{C}$ at scale $(\delta,\e)$ is defined by
$$
P(\mathcal{C},F,\delta,\e)=\limsup_{t\to \infty}\frac{1}{t}\log \Lambda_t(\mathcal{C},F,\delta,\e).
$$
When $\e=0$, we simply write $\Lambda_t(\mathcal{C},F,\delta):=\Lambda_t(\mathcal{C},F,\delta, 0)$ and $P(\mathcal{C},F,\delta):=P(\mathcal{C},F,\delta,0)$. Then the \textit{pressure} of $F$ over $\mathcal{C}$ is defined as 
$$P(\mathcal{C},F)=\lim_{\delta\to 0}P(\mathcal{C}, F,\delta).$$ 
Denote $P(Z, F):=P(Z\times \RR^+, F)$ for any $Z\subset SM$. Write $P(F,\delta):=P(SM\times \mathbb{R}^+,F,\delta)$, and the \textit{topological pressure} of $F$ is defined as  
$$P(F):=\lim_{\delta\to 0}P(F,\delta).$$
\end{definition}

\begin{theorem}(Variational principle for Pressure, \cite[Theorem 4.3.7]{FH})
Let $F \in C(SM,\RR)$. Then
\begin{equation*}
\begin{aligned}
P(F)&=\sup\Big\{h_{\mu}(g^1)+\int F d{\mu}: \mu\in \mathcal{M}_{\CG}(SM)\Big\}\\
&=\sup\Big\{h_{\mu}(g^1)+\int F d{\mu}: \mu\in \mathcal{M}_{\CG}^e(SM)\Big\}
\end{aligned}
\end{equation*}
where $\mathcal{M}_{\CG}(SM)$ (resp. $\mathcal{M}_{\CG}^e(SM))$ denotes the set of all probability measures on $SM$ invariant (resp. ergodic) under the geodesic flow, and $h_{\mu}(g^1)$ is the metric entropy of $\mu$ under the geodesic flow.
\end{theorem}
A measure $\mu\in \M_{\CG}(SM)$ is called an \emph{equilibrium state} of $F$, if
\begin{equation*}
  h_{\mu}(g^1)+\int F d\mu=P(F).
\end{equation*}
An equilibrium state of $F\equiv 0$ is called a \emph{measure of maximal entropy} (MME for short) of the geodesic flow.

\subsubsection{Rank one manifolds}
A vector field $J(t)$ along a geodesic $c:\RR\to M$ is called a \emph{Jacobi field} if it satisfies the \emph{Jacobi equation} 
$$J''+R(J, \dot c)\dot c=0$$
where $R$ is the Riemannian curvature tensor and\ $'$\ denotes the covariant derivative along $c$.

A Jacobi field $J(t)$ is called \emph{parallel} if $J'(t)=0$ for all $t\in \RR$. $J(t)$ is called \emph{stable} (resp. \emph{unstable}) if there exists $C>0$ such that $\|J(t)\|\le C$ for all $t\ge 0$ (resp. $t\le 0$). The notion of \emph{rank} is defined as follows.
\begin{definition}
For each $v \in SM$, we define \text{rank}($v$) to be the dimension of the vector space of parallel Jacobi fields along the geodesic $c_{v}$, and 
\[\text{rank}(M):=\min\{\text{rank}(v): v \in SM\}.\] 
For a geodesic $c$ we define $\text{rank}(c):=\text{rank}(\dot c(t)), \forall\ t\in \mathbb{R}.$
\end{definition}
Let $M$ be a closed rank one manifold of nonpositive curvature. Then $SM$ splits into two disjoint subsets invariant under the geodesic flow: the regular set $\text{Reg}:= \{v\in SM: \text{rank}(v)=1\}$, and the singular set $\text{Sing}:= SM \setminus \text{Reg}$. 

Let $p: X\to M$ be the universal cover of $M$ and $\C \simeq \pi_1(M)$ the group of deck transformations on $X$. So each $\c\in \C$ acts isometrically on $X$. Since $M=X/\Gamma$ is compact, each $\c\in \Gamma$ is \emph{axial} (cf. \cite[Lemma 2.1]{CS}), that is, there exists a geodesic $c$ and $t_{0}>0$ such that $\c(c(t)) = c(t+t_{0})$ for every $t\in \mathbb{R}$. Correspondingly $c$ is called an \emph{axis} of $\c$ and we denote $|\c|:=t_0$ where $t_0$ is minimal with the above property.  

We call two geodesics $c_{1}$ and $c_{2}$ on $X$ \emph{positively asymptotic or asymptotic} if there exists $C > 0$ such that $d(c_{1}(t),c_{2}(t)) \leq C, ~~\forall~ t \geq 0.$ Note that $d(c_{1}(t),c_{2}(t))$ is convex in $t$ due to nonpositive curvature. Asymptoticity is an equivalence relation between geodesics on $X$. The class of geodesics that are asymptotic to a given geodesic $c_v/c_{-v}$ is denoted by $c_v(+\infty)$/$c_v(-\infty)$ or $v^+/v^-$ respectively. We call them \emph{points at infinity}. Obviously, $c_{v}(-\infty)=c_{-v}(+\infty)$. We call the set $\partial X$ of all points at infinity the \emph{boundary at infinity}.
Denote $\overline X:=X\cup \pX$. If $n=\dim X$, $\overline{X}$ is homeomorphic to the closed unit ball in $\mathbb{R}^n$, and $\pX$ is homeomorphic to the unit sphere $\mathbb{S}^{n-1}$ under the cone topology, see \cite{EO}.

For any $p\in X$ and $\xi\in \overline{X}$, there exists a unique geodesic connecting $p$ and $\xi$, denoted by $c_{p,\xi}$, parametrized with $c_{p,\xi}(0)=p$. For $p,q\in X, \xi\in \pX$, we write $[p,q]$ for the geodesic segment from $p$ to $q$, and $[p,\xi)$ the geodesic ray from $p$ pointing to $\xi$. If $\xi,\eta\in \pX$, there may be more than one geodesics connecting $\xi$ and $\eta$, which form  a \emph{flat strip}, i.e., an isometric embedding of a strip $E\times \RR$ in Euclidean space into $X$.

For each pair of points $(p,q)\in X \times X$ and each $\xi \in \pX$, the \emph{Busemann function based at $\xi$ and normalized by $p$} is
$$b_{\xi}(q,p):=\lim_{t\rightarrow +\infty}\big(d(q,c_{p,\xi}(t))-t\big).$$
The limit exists since the function $t \mapsto d(q,c_{p,\xi}(t))-t$ is bounded from above by $d(p,q)$, and decreasing in $t$. If $v\in S_pX$ points at $\xi\in \pX$, we also write $b_v(q):=b_{\xi}(q,p).$

Let $\F$ be a fundamental domain with respect to $\C$ and $D=\diam \F$.  One can lift a potential $F:SM\to \RR$  to a potential on $SX$, which still denoted by $F$. If there is no confusion, we do not distinguish the notations of objects on $M$ and its lift to $X$. For $p,q\in SX$, we also write
$$\int_p^qF:=\int_0^{d(p,q)}F(g^sv)ds$$
where $v:=\dot c_{p,q}(0)$.

\subsection{Orbit decomposition}
For basic notions on the geometry of geodesic flows, we refer to \cite[Section 2.4]{BCFT}. 
\subsubsection{$\l$ function}
There exist $g^t$-invariant subbundles $E^s$ and $E^u$ of $TSM$, which are integrable into $g^t$-invariant foliations $W^s$ and $W^u$ respectively. For $v\in SM$, we call $W^{s/u}(v)$ the \textit{(global) stable/unstable manifolds} of the geodesic flow through $v$. We denote by $W^{s/u}(v,\delta)$ the ball of radius $\delta>0$ centered at $v$ with respect to the intrinsic metrics $d^{s/u}$ on $W^{s/u}(v)$.  
The weak stable/unstable foliations $W^{cs/cu}$ are defined in a similar way. The lifted foliations into $SX$ are denoted by the same notation.

For $v\in SX$, let $H^{s/u}(v)=\pi W^{s/u}(v)$ be the stable/unstable horospheres in $X$ associated to $v$. We also write $H^{s/u}(v)$ as $H^{s/u}(\pi(v),v^+)$. In fact, the stable horospheres are the level sets of Busemann functions. Recall that $\UUU_v^s: T_{\pi v}H^s\to  T_{\pi v}H^s$ is the symmetric linear operator associated to the stable horosphere $H^s$, and similarly for $\UUU_v^u$. Let $\lambda^u(v)$ be the minimal eigenvalue of $\UUU_v^u$ and let $\lambda^s(v)=\lambda^u(-v)$. Then we define $\lambda(v)=\min \{\lambda^u(v), \lambda^s(v)\}.$ $\l: SX\to \RR$ is a continuous function, which descends to a continuous function on $SM$.

By compactness of $SM$, given $\eta>0$, there exists $\delta=\delta(\eta)>0$ small enough such that for any $v,w\in SX$, 
\begin{equation}\label{smalldelta}
d_K(v,w)< \delta e^\Lambda \Rightarrow |\lambda(v)-\lambda(w)|\le \eta/3.    
\end{equation}
Here $\Lambda$ is the maximal eigenvalue of $\UUU^u(v)$ taken over all $v\in SM$. 

The following lemma will be useful to treat Condition (B).
\begin{lemma}(\cite[Proposition 3.4]{BCFT})\label{nhdsing}
For any $\rho>0$, there are $\eta>0$ and $T>0$ such that if $\l^s(g^tv)\le \eta$ for all $t\in [-T,T]$, then $d_K(v,\text{Sing})<\rho$. 
\end{lemma}

\subsubsection{$(\mathcal{P},\mathcal{G},\mathcal{S})$-decomposition}
We recall the $(\mathcal{P},\mathcal{G},\mathcal{S})$-decomposition defined in \cite{BCFT} for geodesic flows on rank one manifolds.
\begin{definition} \label{def:decomp}
A decomposition $(\mathcal{P},\mathcal{G},\mathcal{S})$ for $\mathcal{D}\subset SM\times \mathbb{R}^+$ consists of three collections $\mathcal{P},\mathcal{G},\mathcal{S} \subset SM\times \mathbb{R}^+$ and functions $p,h,s:\mathcal{D}\to \mathbb{R}^+\cup\{0\}$ such that for every $(v,t)\in \mathcal{D}$, writing $p(v,t), h(v,t), s(v,t)$ as $p,h,s$ respectively, we have $t=p+h+s$, and
$$(v,p)\in \mathcal{P}, \qquad (g^p(v),h)\in \mathcal{G}, \qquad (g^{p+h}(v),s)\in \mathcal{S}.$$
Given a decomposition $(\mathcal{P},\mathcal{G},\mathcal{S})$ and constant $L\geq 0$, we denote by $\mathcal{G}^L$ the collection of $(v,t)\in \mathcal{D}$ satisfying $\max\{p(v,t),s(v,t)\}\leq L$. %Meanwhile, $X\times \{0\}$ is always assumed to belong to all three collections $(\mathcal{P},\mathcal{G},\mathcal{S})$ by default.
\end{definition}
Let $\eta>0$. Define 
\begin{equation*}
    \begin{aligned}
\GGG(\eta)&:=\Big\{(v,t): \int_0^\tau \lambda(g^sv)ds \ge \eta\tau,\ \int_0^\tau \lambda(g^{-s}g^tv)ds \ge \eta\tau,\ \forall \tau \in [0,t]\Big\},\\
\BBB(\eta)&:=\Big\{(v,t): \int_0^t \lambda(g^sv)ds < \eta t\Big\}.
%\text{Reg}(\eta)&:=\{v\in SM: \lambda(v)\ge \eta\},\\
%\CCC(\eta)&:=\{(v,t)\in SM\times \RR^+: v\in \text{Reg}(\eta), \ g^tv\in \text{Reg}(\eta)\}.
    \end{aligned}
\end{equation*}
\begin{definition}(\cite[p. 1221]{BCFT})\label{decom}
Given $(v,t)\in SM\times\RR^+$, take $p=p(v,t)$ to be the largest time such that $(v,p)\in \BBB(\eta)$, and $s=s(v,t)$ be the largest time in $[0, t-p]$ such that $(g^{t-s}v, s)\in \BBB(\eta)$. Then it follows that $(g^pv, h)\in \GGG(\eta)$ where $h=t-p-s$. Thus the triple $(\BBB(\eta), \GGG(\eta), \BBB(\eta))$ equipped with the functions $(p,h,s)$ determines a decomposition for $SM\times \RR^+$.
\end{definition}

As verified in \cite{BCFT}, the above $(\P, \GGG,\SSS)$-decomposition satisfies the conditions of the following criterion for uniqueness of equilibrium states. Then Theorem \ref{bcft} follows.
\begin{theorem}(\cite[Theorem A]{CT16})
 Let $(Y, (f^t)_{t\in \RR})$ be a continuous flow on a compact metric space $Y$, and $F: Y\to \RR$ a continuous potential function. Suppose that $P^{\perp}_{\exp}(F) <P(F)$ and that $Y\times \RR^+$ admits 
a decomposition $(\PPP, \GGG, \SSS)$ with the following properties:
\begin{enumerate}
    \item $\GGG$ has the weak specification property;
    \item $F$ has the Bowen property on $\GGG$; 
    \item  $P([\PPP]\cup [\SSS], F)<P(F).$
\end{enumerate}
  Then $F$ has a unique equilibrium state.   
\end{theorem}
The following lemma is crucial. For $v\in SM$, denote 
$$\tilde{\l}^{s/u}(v):=\max\{\l^{s/u}(v)-\frac{\eta}{3}, 0\}.$$
\begin{lemma}(\cite[Lemma 3.10]{BCFT})\label{stablegeo}
Given $\eta>0$, let $\delta=\delta(\eta)$ be as in \eqref{smalldelta}, $v\in SM$ and $w,w'\in W^s(v,\delta)$. Then for any $t\ge 0$ we have
$$d^s(g^tw, g^tw')\le d^s(w,w')e^{-\int_0^t\tilde\l^s(g^\tau v)d\tau}.$$
Similarly, for every $w,w'\in W^u(v,\delta)$ and $t\ge 0$, we have
$$d^u(g^{-t}w, g^{-t}w')\le d^u(w,w')e^{-\int_0^t\tilde\l^u(g^{-\tau} v)d\tau}.$$
\end{lemma}

\subsection{Core limit set}
Given $p,q\in X$, we also write $(\dot c_{p,q}(0), d(p,q))\in SX\times \RR^+$ as the geodesic segment $[p,q]$ for simplicity. 
\begin{definition}\label{corelimit}
Let $\CCC$ denote the set of geodesic segments $[o,\c o]$ from $o$ to $\c o$ where $\c\in \C$. Given $L>0$, let $\L_c^L(\C)$ be the set of $\xi\in \p X$ such that there exists $\c_n\in \C$ such that $[o,\c_no]\in \GGG^L$ and $\xi=\lim_{n\to \infty}\c_no.$
Then we define the \emph{core limit set} as
$$\L_c(\C):=\bigcup_{L>0}\L_c^L(\C).$$
\end{definition}

\begin{lemma}\label{core}
Suppose that $v=\dot c_{o,\xi}(0)$ for some $\xi\in \L_c^L(\C)$. Then there exists $T_0\in [0,L]$, such that $\int_{T_0}^t\lambda(g^sv)ds \ge \eta(t-T_0)$ for every $t\ge L$. 
%Moreover, there exist a neighborhood $U$ of $v$ in $W^s(v)$ such that for every $w\in U$, every $Y\in J^s(w)$ and all $t > T_0$,
%\begin{equation}\label{e:deca}
%\|Y(t)\|< Ce^{-\eta t/2}\|Y(0)\|
%\end{equation}
%where $C$ is a constant depending on $M$ and $\eta$.
\end{lemma}
\begin{proof}
Since $\xi\in \L_c^L(\C)$, there exist $\c_n\in \C$ such that $[o,\c o]\in \GGG^L$ and $\xi=\lim_{n\to \infty}\c_n o$. Fix $t\ge L$. %First we consider $Y\in J^s(w)$. By \cite[Lemma 2.11]{BCFT}, 
%\begin{equation*}
%\|Y(t)\|<e^{-\int_0^t\lambda(g^sv)ds}\|Y(0)\|.
%\end{equation*}
Note that $\dot c_{o, \c_no}(0)\to v$ as $n\to \infty$. Since $t$ is fixed, we see that $\lambda(\dot c_{o, \c_no}(s))\to \lambda(g^sv)$ uniformly for $0\le s\le t$. 

On the other hand, since $[o,\c o]\in \GGG^L$, there exist $t_n\in [0, L]$ such that
$$\int_{t_n}^t\lambda(\dot c_{o, \c_no}(s))ds\ge \eta(t-t_n)$$ 
for $n$ large enough. By passing to a subsequence, we may assume that $\lim_{n\to \infty}t_n=T_0\in [0,L]$. Setting $n\to \infty$ in the above, we have 
$$\int_{T_0}^t\lambda(g^sv)ds \ge \eta(t-T_0).$$
The lemma follows.
\end{proof}

The following Pliss lemma \cite{Pliss} is crucial to our proof.
\begin{lemma}(\cite[Lemma 3.5]{LVY})\label{pli}
Given $a_*\le c_2<c_1$ there exists $\theta=\frac{c_1-c_2}{c_1-a_*}$ such that, given any real numbers $a_1, \cdots, a_N$ with
\[\sum_{i=1}^Na_i\le c_2N \ \text{and}\ a_i\ge a_*\ \text{for every\ } i,\]
there exist $l>N\theta$ and $1\le n_1 <\cdots <n_l\le N$ such that
 \[\sum_{i=n+1}^{n_j}a_i\le c_1(n_j-n) \ \text{for all}\ 0\le n<n_j \ \text{and\ }j=1, \cdots ,l.\]
\end{lemma}

\begin{lemma}\label{limitzero}
Let $v=\dot{c}_{o,\xi}(0)$ where $\xi\in \L_c^L(\C)$. If $p\in H^s(v)$, then $\lim_{t\to +\infty}d^s(\dot c_{p,\xi}(t), \dot c_{o,\xi}(t))=0$.
\end{lemma}
\begin{proof}
Assume not. Take a shortest curve $\beta: [0,1]\to W^s(v)$ with $\beta(0)=v$ and $\b(1)=\dot c_{p,\xi}(0)$. Define $\b(s,t)=g^t(\b(s))$ for every $t\ge 0$ and $0\le s\le1$. By the assumption and \cite[Lemma 2.13]{BCFT}, there exists $c>0$ such that $l^s(\b([0,1],t))\searrow c$ as $t\to +\infty$, where $l^s$ denotes the length of the curve with respect to $d^s$. In other words, for any $0<\rho\ll \min\{c,\d\}$ where $\d$ is from \eqref{smalldelta}, there exists $T_1>L$ large enough such that 
$$l^s(\b([0,1],t))\in [c, c+\rho/100), \quad \forall t\ge T_1.$$

By Lemma \ref{core}, for any $t\ge L$, $\int_{T_0}^t\lambda(g^sv)ds \ge \eta(t-T_0)$. Applying Pliss lemma \ref{pli} with 
$a_i=-\int_0^1\lambda(g^s(g^{T_0+i-1}v))ds$, $a_*= -\|\lambda\|$, $c_2=-\eta, c_1=-\frac{9\eta}{10}$, and $N$ large enough, we know that there exists a large Pliss time $T_2\gg T_1$. It implies that
\begin{equation}\label{ave}
\int_{T_1}^{T_2}\lambda(g^sv)ds \ge \frac{5\eta}{6}(T_2-T_1).
\end{equation}
Clearly, there exists $s_0\in [0,1]$ such that $l^s(\b([0,s_0],T_1)=\rho.$
By \eqref{smalldelta}, \eqref{ave} and Lemma \ref{stablegeo}, we know 
$$l^s(\b([0,s_0],T_2))\le \rho e^{-\frac{\eta}{2}(T_2-T_1)}\le \frac{\rho}{10}$$ 
if $T_2$ is large enough.
But then 
\begin{equation*}
\begin{aligned}
l^s(\b([0,1],T_2))&=  l^s(\b([0,s_0],T_2))+ l^s(\b([s_0,1],T_2))\\
&\le  l^s(\b([0,s_0],T_1))-\frac{9\rho}{10}+l^s(\b([s_0,1],T_1))\\
&=  l^s(\b([0,1],T_1))-\frac{9\rho}{10}\le c+\frac{\rho}{100}-\frac{9\rho}{10}<c.
\end{aligned}
\end{equation*}
A contradiction. The lemma follows.
\end{proof}

\begin{lemma}\label{decay}
Let $v=\dot{c}_{o,\xi}(0)$ where $\xi\in \L_c^L(\C)$. If $p\in X$, then 
$$d_K(\dot c_{p,\xi}(t+T_3), \dot c_{o,\xi}(t))\le \d e^{\frac{\eta}{2}T_4}e^{-\frac{\eta}{2}t}, \ \forall t>T_5,$$
for some $T_3(p,\xi)\in \RR$ and $0<T_4(p,\xi) <T_5(p,\xi)$ independent of $t>0$.
\end{lemma}
\begin{proof}
Take $T_3=b_\xi(p,o)$, so that $c_{p,\xi}(T_3)\in H^s(v)$. For simplicity we assume $T_3=0$ below and therefore $p\in H^s(v)$. By Lemma \ref{limitzero}, there exists $T_4=T_4(p,\xi)$ such that 
$$d_K(\dot c_{p,\xi}(t), \dot c_{o,\xi}(t))\le d^s(\dot c_{p,\xi}(t), \dot c_{o,\xi}(t))\le  \d, \ \forall t>T_4$$
where $\d$ is from \eqref{smalldelta}.

We claim that there exists $T_5=T_5(p,\xi)>T_4(p,\xi)$ such that
$$\int_{T_4}^t\lambda(g^sv)ds \ge \frac{99\eta}{100}(t-T_4),\ \forall t>T_5.$$ 
Indeed, by Lemma \ref{core}, for any $t\ge L$ we have $\int_{T_0}^t\lambda(g^sv)ds \ge \eta(t-T_0)$. So 
$$(T_4-T_0)\|\lambda\|+\int_{T_4}^t\lambda(g^sv)ds\ge  \eta(t-T_0)$$
which implies that
$$\frac{1}{t-T_4}\int_{T_4}^t\lambda(g^sv)ds\ge \frac{\eta(t-T_0)-(T_4-T_0)\|\lambda\|}{t-T_4}> \frac{99\eta}{100}$$
if $t$ is large enough. So the claim holds.

By the claim and Lemma \ref{stablegeo}, if $t>T_5$,
$$d_K(\dot c_{p,\xi}(t), \dot c_{o,\xi}(t))\le d^s(\dot c_{p,\xi}(t), \dot c_{o,\xi}(t))\le \d e^{-\frac{\eta}{2}(t-T_4)}=\d e^{\frac{\eta}{2}T_4}e^{-\frac{\eta}{2}t}.$$
The proof of the lemma is complete.
%where $C=$ is independent of $t$.
\end{proof}

\subsection{Regular radial limit set}
Given $\l>0$ small enough, we define $K_\l:=\{v\in SM: \l^s(v)\ge \l\}$ and lift it into a fundamental domain which we still denote by $K_\l$. Here we use function $\l^s$ instead of function $\l$, so Lemma \ref{nhdsing} can be applied. Indeed, the computation below does not involve unstable manifolds and function $\l^u$. 
\begin{definition}
The \emph{regular radial limit set} is defined as 
$$\L_{r}(\C):=\bigcup_{\l>0}\L_r^\l(\C)$$
where
\[\L_r^\l(\C):=\{\xi\in \pX: \exists \{\c_n\}_{n=1}^\infty\subset \C \text{\ s.t.\ }c_{o,\xi}([0,+\infty))\cap \c_n K_\l\neq \emptyset\}.\]
\end{definition}

\begin{lemma}\label{limitzero1}
Let $v=\dot{c}_{o,\xi}(0)$ where $\xi\in \L_r^\l(\C)$. If $p\in H^s(v)$, then $\lim_{t\to +\infty}d^s(\dot c_{p,\xi}(t), \dot c_{o,\xi}(t))=0$.
\end{lemma}
\begin{proof}
Assume not. By uniform continuity of $\l^s:SM\to \RR$, there exists $\d'>0$ such that if $d_K(v,w)\le \d'$, then $|\l^s(v)-\l^s(w)|<\l/100$. Take a shortest curve $\beta: [0,1]\to W^s(v)$ with $\beta(0)=v$ and $\b(1)=\dot c_{p,\xi}(0)$. Define $\b(s,t)=g^t(\b(s))$ for every $t\ge 0$ and $0\le s\le1$. By the assumption and \cite[Lemma 2.13]{BCFT}, there exists $c>0$ such that $l^s(\b([0,1],t))\searrow c$ as $t\to +\infty$. In other words, for any $0<\rho\ll \min\{c,\d'\}$, there exists $T_1>0$ such that 
$$l^s(\b([0,1],t))\in [c, c+\frac{1}{2}\rho (1-e^{-\frac{\l}{3}\d'})), \quad \forall t\ge T_1.$$

Let $T_2\ge T_1$ be the first time after $T_1$ when $c_{o,\xi}$ enters $\a K_\l$ for some $\a\in \C$, and then let $T_3=T_2+\d'$. Such $T_2$ exists since $\xi\in \L_r^\l(\C)$. We have $\l^s(g^sv)\ge 2\l/3$ for any $s\in [T_2,T_3]$ by the choice of $\d'$.

Note that there exists $s_0\in [0,1]$ such that $l^{s}(\b([0,s_0],T_2)=\rho.$
By the choice of $\d',\rho$ and Lemma \ref{stablegeo}, we know 
$$l^s(\b([0,s_0],T_3))\le \rho e^{-\frac{\l}{3}(T_3-T_2)}= \rho e^{-\frac{\l}{3}\d'}.$$
But then 
\begin{equation*}
\begin{aligned}
l^s(\b([0,1],T_3))&=  l^s(\b([0,s_0],T_3))+ l^s(\b([s_0,1],T_3))\\
&\le  l^s(\b([0,s_0],T_2))-\rho(1-e^{-\frac{\l}{3}\d'})+l^s(\b([s_0,1],T_2))\\
&=  l^s(\b([0,1],T_2))-\rho (1-e^{-\frac{\l}{3}\d'})\\
&\le c+\frac{1}{2}\rho (1-e^{-\frac{\l}{3}\d'})-\rho (1-e^{-\frac{\l}{3}\d'})<c.
\end{aligned}
\end{equation*}
A contradiction. The lemma follows.
\end{proof}

\subsection{Uniformly recurrent and regular vectors}
For a regular vector $v\in \text{Reg}$, we want to describe its regularity quantitatively using the exponentially contracting rate of the geodesic flow along the stable manifold of $v$. By Lemma \ref{stablegeo}, $\l(v)$, or more precisely $\l^s(v)$, indicates the regularity of $v$ along its local stable manifold. In this section, in order to have exponentially decreasing rate along global stable manifolds, we consider uniformly recurrent and regular vectors.
\begin{definition}\label{uniformlyrecurrent}
A vector $v\in SM$ is said to be \emph{uniformly recurrent} if for any neighborhood $U$ of $v$ in $SM$
$$\liminf_{t \to \pm \infty} \frac{1}{T}\int_0^T \chi_U(g^t(v))dt > 0,$$
where $\chi_U$ is the characteristic function of $U$.
\end{definition}

The notion of uniformly recurrent vectors has appeared in \cite{BBE}, where the authors constructed a strong stable manifold for any uniformly recurrent and regular vector and showed there is an exponential contraction along the strong stable manifold \cite[Proposition 3.10]{BBE}. Though \cite{BBE} deals with manifolds of nonpositive curvature with higher rank, similar results hold for rank one manifolds.

We say that $v\in SX$ is uniformly recurrent, if its projection to $SM$, $dp(v)$, is uniformly recurrent. Let $\UR$ denote both the set of uniformly recurrent vectors in $SM$ and its lift to $SX$. The next two lemmas are stated in \cite[p. 192]{BBE} without a proof. For completeness, we provide proofs in the Appendix.

\begin{lemma}\label{uniformrec}
Let $M$ be a closed rank one manifold of nonpositive curvature. We have
\begin{enumerate}
  \item $\UR$ is $g^t$-invariant and flip invariant;
  \item $\nu(\UR)=1$ for any $\nu\in \M_{\GGG}(SM)$.
\end{enumerate}
\end{lemma}

\begin{lemma}\label{recrate}
If $v\in SX$ is uniformly recurrent, then for any open neighborhood $U$ of $v$, any $T>0$, there exist $\{\phi_n\}$ in the isometry group $\text{Iso}(X)$ of $X$, $t_n\to \infty$, and $\sigma>0$, such that $d\phi_ng^{t_n}v\in U, t_{n+1}-t_n>T$ and $t_n<\frac{n}{\sigma}$ for any $n\in \NN$.
\end{lemma}

We derive exponential decay along stable manifolds in the following two propositions. Write $\URR:=\UR\cap \text{Reg}$, the set of uniformly recurrent and regular vectors. For $w\in SX$, let $J^s(w)$ denote the space of normal stable Jacobi fields along $c_w$.
\begin{proposition}\label{rate}
Let $M$ be a closed rank one manifold of nonpositive curvature and $v\in \URR\subset SX$. Then there exist a neighborhood $U$ of $v$ in $W^s(v)$ and constants $\lambda=\lambda(v)>0, C=C(v)>1$ such that for every $w\in U$, every $Y\in J^s(w)$,
\begin{equation}\label{jacobi}
\|Y(t)\|< Ce^{-\lambda t}\|Y(0)\|, \quad \forall t>0.
\end{equation}
\end{proposition}

Let $v\in \URR$. It is proved in \cite[Proposition 3.10]{BBE} that 
$$W^s(v)=\{w\in SX: \lim_{t\to +\infty}d(g^tv, g^tw)=0\}.$$ 
Moreover, we have the following quantitative estimate.

\begin{proposition}\label{contracting}
Let $M$ be a closed rank one manifold of nonpositive curvature and $v\in \URR\subset SX$. Then there exist a constant $\lambda=\lambda(v)> 0$ such that for any $w\in W^s(v)$,
\begin{equation}\label{e:imp}
d_K(g^tv,g^tw) \le Ce^{d^s(v,w)/\lambda}d^s(v,w)e^{-\lambda t}, \quad \forall t>0
\end{equation}
where $C>1$ is a universal constant.
\end{proposition}
 
The proof of Propositions \ref{rate} and \ref{contracting} is given in the Appendix. From the proof, we obtain some explicit estimates, which enable us to formulate the following quantitative definition. Let $-a^2$ be a lower bound for the sectional curvature of $M$.
\begin{definition}
Given $\lambda>0$, we say that $v\in SX$ is \emph{$\lambda$-uniformly recurrent and regular}, denoted by $v\in \UR(\lambda)$, if $v\in \URR$ and
\begin{enumerate}
   \item there exist $C'=C'(v)>1$ and a neighborhood $U$ of $v$ in $W^s(v)$ containing $W^s(v, 8\log 2\cdot \lambda\sqrt{1+a^2})$ such that 
\begin{equation*}
\|Y(t)\|< C'e^{-\lambda t}\|Y(0)\|, \quad \forall t > 0
\end{equation*}
holds for every $w\in U$ and every $Y\in J^s(w)$;
\item  for any $w\in W^s(v)$,
$$d_K(g^tv,g^tw)\le Cd^s(v,w)e^{d^s(v,w)/\lambda}e^{-\lambda t}, \quad \forall t > 0$$
where $C>1$ is a universal constant;
\item the above two items also hold for $-v$ instead of $v$. 
\end{enumerate}
Moreover, if $v\in \text{Reg}$ satisfies (1), (2) and (3) above, we denote $v\in \R(\lambda)$.
\end{definition}
It is clear that $\URR=\bigcup_{\lambda>0}\UR(\lambda)$ by Propositions \ref{rate} and \ref{contracting}. We note again that $\UR$ and thus $\URR$ are $g^t$-invariant and flip invariant by Lemma \ref{uniformrec}.

The proof of the following lemma is given in the Appendix. 
\begin{lemma}\label{contracting1}
Let $v\in \R(\lambda)$. For any $p\in H^s(v)$, $T>0$, let $w_T:=\dot c_{p, \pi g^Tv}(0)$, $w:=\dot c_{p,v^+}(0)$ and $S= d(\pi w_T, \pi g^Tv)$. Then for any $0\le t\le \min\{S,T\}$=T,
\begin{equation*}
d_K(g^{T-t}v,g^{S-t}w_T) \le C_0d^s(v,w)e^{d^s(v,w)/\lambda}e^{-\lambda (T-t)},
\end{equation*}
where $C_0$ is a universal constant.
\end{lemma}

At the end of this subsection, we define the following subsets of $SX$, which will be frequently used later. 
\begin{definition}\label{lsets}
Given $\l>0$ and $k,N\in \NN$, define
\begin{equation}\label{keyset}
    \begin{aligned}
    \L_{k,\l, N}:=\Big\{&v\in \UR(\l)\subset SX: \frac{1}{t}\int_0^t\chi_{\UR(\l)}(g^sv)ds\ge 1-\frac{1}{k},\\
    &\text{\ and \ }\frac{1}{t}\int_0^t\chi_{\iota(\UR(\l))}(g^sv)ds\ge 1-\frac{1}{k}, \quad \forall t\ge N\Big\}.
    \end{aligned}
\end{equation}
Then we define $\tilde\L_k:=\cup_{i=1}^\infty\cup_{N=1}^\infty \L_{k,\frac{1}{i},N}$.
\end{definition}
\subsection{Gibbs cocycles and Bounded distortion}
Now consider a H\"{o}lder continuous function $F:SM\to \RR$ with  H\"{o}lder exponent $\a\in (0,1)$ and H\"{o}lder constant $K$, that is 
$$K=\sup_{v,w\in SM, v\neq w}\frac{|F(v)-F(w)|}{d_K^\a(v,w)}.$$ 
Denote $\|F\|=\max_{v\in SM}|F(v)|$.
We lift $F$ to the universal cover $X$, still denoted by $F$. 

\begin{definition}\label{cocdef}
Let $\xi\in \pX$. We define the \emph{Gibbs cocycle} for any $p, q\in X$ as
$$C_{F,\xi}(q,p):=\lim_{t\to +\infty}\int_{p}^{c_{p,\xi}(t)}F-\int_{q}^{c_{q,\xi}(t+s)}F$$
where $s=b_\xi(q,p)$, whenever the limit exists.
\end{definition}

It is clear that Gibbs cocycles satisfy the following properties whenever they are well defined.
\begin{lemma}\label{cocycleprop}
\begin{enumerate}
  \item For all $\c\in \Gamma$, $C_{F,\xi}(q,p)=C_{F,\c\xi}(\c q,\c p)$.
  \item $C_{F,\xi}(q,p)+C_{F,\xi}(p,r)=C_{F,\xi}(q,r)$.
  \item If $c_{q,p}(+\infty)=\xi$, then $C_{F,\xi}(q,p)=-\int_q^p F$.
\end{enumerate}
\end{lemma}

The following lemmas are often referred as ``bounded distortion'' property, which is key ingredient to prove the existence of Gibbs cocycles. 

\subsubsection{Uniformly recurrent and regular vectors}
\begin{lemma}\label{Bowen}
Let $v\in \UR(\lambda)$ for some $\l>0$,  $p\in X$, $T_1=b_{v^+}(p, \pi v)$ and $w=g^{T_1}\dot c_{p,v^+}(0)\in W^s(v)$. Then for any $T>0$, we have
$$\left|\int_p^{\pi g^Tv}F-\int_{\pi v}^{\pi g^Tv}F\right|\le C_1d^s(v,w)^\a  e^{d^s(v,w)\a/\lambda}/\l+\|F\|d(p,\pi v)$$
for some constant $C_1=C_1(K,\a)$. Moreover, 
$$\left|\int_p^{\pi g^Tw}F-\int_{\pi v}^{\pi g^Tv}F\right|\le C_2d^s(v,w)^\a  e^{d^s(v,w)\a/\lambda}/\l+\|F\|d(p,\pi v)$$
for some constant $C_2=C_2(K,\a)$.

\end{lemma}
\begin{proof}
Denote $w_T=\dot c_{p,\pi g^Tv}(0)$ and $S=d(p,\pi g^Tv)$. Note that
$$d(p,\pi g^{S-T}w_T)=|S-T|\le d(p,\pi v).$$
%Then we have
%$$d(\pi v,\pi g^{S-T}w)\le d(p,\pi g^{S-T}w)+d(p,\pi v)\le 2d(p,\pi v).$$
It follows by Lemma \ref{contracting1} that
\begin{equation*}
\begin{aligned}
&\left|\int_p^{\pi g^Tv}F-\int_{\pi v}^{\pi g^Tv}F\right|\le \left|\int_{\pi g^{S-T}w_T}^{\pi g^Tv}F-\int_{\pi v}^{\pi g^Tv}F\right|
+\left|\int_p^{\pi g^{S-T}w_T}F\right|\\
\le &K(C_0d^s(v,w))^\a e^{d^s(v,w)\a/\lambda}\int_0^\infty e^{-\lambda\a s}ds+\|F\|d(p,\pi v)\\
\le &\frac{KC_0^\a}{\lambda\a}\cdot d^s(v,w)^\a  e^{d^s(v,w)\a/\lambda}+\|F\|d(p,\pi v)\\
:=&C_1d^s(v,w)^\a  e^{d^s(v,w)\a/\lambda}/\l+\|F\|d(p,\pi v).
\end{aligned}
\end{equation*}
where $C_1=C_1(K,\a):=\frac{KC_0^\a}{\a}$. This proves the first inequality. 

It follows by Proposition \ref{contracting} that
\begin{equation*}
\begin{aligned}
&\left|\int_p^{\pi g^Tw}F-\int_{\pi v}^{\pi g^Tv}F\right|\le \left|\int_{\pi w}^{\pi g^Tw}F-\int_{\pi v}^{\pi g^Tv}F\right|
+\left|\int_p^{\pi w}F\right|\\
\le &KC^\a d^s(w,w)^\a e^{d^s(v,w)\a/\lambda}\int_0^\infty e^{-\lambda\a s}ds+\|F\|d(p,\pi v)\\
\le &\frac{KC^\a}{\lambda\a}\cdot d^s(v,w)^\a e^{d^s(v,w)\a/\lambda}+\|F\|d(p,\pi v)\\
:=&C_2d^s(v,w)^\a e^{d^s(v,w)\a/\lambda}/\l+\|F\|d(p,\pi v)
\end{aligned}
\end{equation*}
where $C_2:=\frac{KC^\a}{\a}$. The lemma follows.
\end{proof}
\begin{corollary}\label{cocexi1}
If $v\in \UR(\l)$ for some $\l>0$, then for any $p,q\in X$,  Gibbs cocycle $C_{F,v^+}(q,p)$ is well defined, that is, the limit in Definition \ref{cocdef} exists.
\end{corollary}
\begin{proof}
By Lemma \ref{cocycleprop}(2), it is enough to prove that $C_{F,v^+}(q,\pi(v))$ is well defined for every $q\in X$. Denote $a_t:=\int_{\pi(v)}^{c_vt)}F-\int_{q}^{c_{q,v^+}(t+s_0)}F$ where $s_0=b_{v^+}(q,\pi(v))$. Then for any $t_2>t_1$ large enough, by Proposition \ref{contracting}
\begin{equation*}
\begin{aligned}
&|a_{t_1}-a_{t_2}|=\left|\int_{c_{v}(t_1)}^{c_{v}(t_2)}F-\int_{c_{q,v^+}(t_1+s_0)}^{c_{q,v^+}(t_2+s_0)}F\right|\\
\le &K\int_{t_1}^{t_2}(d_K(\dot c_{v}(t), \dot c_{q,v^+}(t+s_0)))^\a dt\\
\le &KC^\a d^s(v,\dot c_{q,v^+}(s_0))^\a e^{d^s(v,\dot c_{q,v^+}(s_0))\a/\lambda}\int_{t_1}^{t_2} e^{-\lambda\a t}dt\\
\le & KC^\a d^s(v,\dot c_{q,v^+}(s_0))^\a e^{d^s(v,\dot c_{q,v^+}(s_0))\a/\lambda}\frac{e^{-\lambda\a t_1}}{\l\a}
\end{aligned}
\end{equation*}
which converges to zero exponentially as $t_1\to \infty$. Thus $C_{F,v^+}(q,\pi(v))$ is well defined and the proof of the corollary is complete. 
\end{proof}

\subsubsection{Core limit set}
\begin{lemma}\label{Bowen1}
Let $v=\dot{c}_{o,\xi}(0)$ where $\xi\in \L_c^L(\C)$. For any $p\in X$ and $T>T_5$, we have
$$\left|\int_p^{c_{p,\xi}(T+T_3)}F-\int_{o}^{\pi g^Tv}F\right|\le K\d^\a e^{\frac{\eta\a}{2}T_4}\int_{T_5}^{T}e^{-\frac{\eta\a}{2}t}dt+\|F\|(d(p,o)+2T_5)$$
where $T_3(p,\xi)$, $T_4(p,\xi)$ and $T_5(p,\xi)$ are from Lemma \ref{decay} and independent of $T$.
\end{lemma}
\begin{proof}
By Lemma \ref{decay}, 
$$d_K(\dot c_{p,\xi}(t+T_3), \dot c_{o,\xi}(t))\le \d e^{\frac{\eta}{2}T_4}e^{-\frac{\eta}{2}t}, \ \forall t>T_5.$$
Note that $|T_3|\le d(p,o)$. Then the lemma follows from an analogous computation as in the proof of Lemma \ref{Bowen}.
\end{proof}

\begin{corollary}\label{cocexi2}
If $p, q\in X$ and $\xi\in \L_c^L(\C)$, then Gibbs cocycle $C_{F,\xi}(q,p)$ is well defined.    
\end{corollary}

\begin{proof}
The proof is analogous to that of Corollary \ref{cocexi1}. It is enough to prove that $C_{F,\xi}(q,o)$ is well defined. Denote $a_t:=\int_{o}^{c_{o,\xi}(t)}F-\int_{q}^{c_{q,\xi}(t+s_0)}F$ where $s_0=b_\xi(q,o)$. Then for any $t_2>t_1>T_5$, by Lemma \ref{decay}
\begin{equation*}
\begin{aligned}
&|a_{t_1}-a_{t_2}|=\left|\int_{c_{o,\xi}(t_2)}^{c_{o,\xi}(t_1)}F-\int_{c_{q,\xi}(t_2+s_0)}^{c_{q,\xi}(t_1+s_0)}F\right|\\
\le &K\int_{t_1}^{t_2}d_K(\dot c_{o,\xi}(t), \dot c_{q,\xi}(t+s_0))^\a dt\\
\le &K\int_{t_1}^{t_2}(\d e^{\frac{\eta}{2}T_4}e^{-\frac{\eta}{2}t})^\a dt\le K \d^\a e^{\frac{\eta\a}{2}T_4}\frac{2}{\eta\a}e^{-\frac{\eta\a}{2}t_1}
\end{aligned}
\end{equation*}
which converges to zero exponentially as $t_1\to \infty$. We are done.
\end{proof}

\subsubsection{Regular radial limit set}

\begin{lemma}\label{cocyle2}
Assume that $F$ is locally constant on an open neighborhood of $\text{Sing}$. Then for any $p, q\in X$ and $\xi\in \pX$, Gibbs cocycle $C_{F,\xi}(q,p)$ is well defined.
\end{lemma}
\begin{proof}
By Lemma \ref{cocycleprop}(2), it is enough to prove the existence of $C_{F,\xi}(p,o)$ for any $\xi\in \pX$ and $p\in X$. Pick $\rho$ small enough such that $F\equiv c$ on a $\rho$-neighborhood of $\text{Sing}$. By Lemma \ref{nhdsing}, there exist $\l>0$ and $T>0$ such that if $\l^s(g^tv)\le \l$ for all $t\in [-T,T]$, then $d_K(v, \text{Sing})<\rho$.

Assume first that $\xi\in \L_r^\l$ and denote $v=\dot c_{o,\xi}(0)$, $w=\dot c_{p,\xi}(0)$. 
By uniform continuity of $\l^s:SM\to \RR$, there exists a $\d'>0$ such that if $d_K(v,w)\le \d'$, then $|\l^s(v)-\l^s(w)|<\l/100$. By Lemma \ref{limitzero1}, there exists $t_0\in \NN$ such that for every $t\ge t_0$, 
$$d_K(g^{t}v, g^{t+s_0}w)\le d^s(g^{t}v, g^{t+s_0}w)<\d'$$
where $s_0=b_\xi(p, o)$. Then $|\l^s(g^{t}v)-\l^s(g^{t+s_0}w)|<\l/100.$

Since $\xi\in \L_r^\l$, there exists a sequence of successive times $t_0\le T_1\le T_2\le T_3\le \cdots$ such that 
\begin{itemize}
    \item for any $t\in [T_{2i+1},T_{2i+2}]$, either $g^tv\in \C K_\l$ or $g^{t+s_0}w\in \C K_\l$;
    \item  for any $t\in [T_{2i+2},T_{2i+3}]$, $\max\{\l^s(g^tv), \l^s(g^{t+s_0}w)\}<\l$.
\end{itemize}
Then for every $t\in [T_{2i+1},T_{2i+2}]$, $\min\{\l^s(g^{t}v),\l^s(g^{t+s_0}w))\}\ge \frac{5\l}{6}$. Then by Lemma \ref{stablegeo},
\begin{equation}\label{decresingpart}
    \begin{aligned}
&\left|\int_{c_{o,\xi}(T_{2i+1})}^{c_{o,\xi}(T_{2i+2})}F-\int_{c_{p,\xi}(T_{2i+1}+s_0)}^{c_{p,\xi}(T_{2i+2}+s_0)}F\right|\\
\le &K(d_K(g^{T_{2i+1}}v, g^{T_{2i+1}+s_0}w))^\a \int_{T_{2i+1}}^{T_{2i+2}}e^{-\a\l t/2}dt.        
    \end{aligned}
\end{equation}
On the other hand, if $|T_{2i+3}-T_{2i+2}|>2T$, then by Lemma \ref{nhdsing}, $F\equiv c$ on the time interval $[T_{2i+2}+T, T_{2i+3}-T]$. We have 
\begin{equation}\label{constantpart}
    \begin{aligned}
&\left|\int_{c_{o,\xi}(T_{2i+2}+T)}^{c_{o,\xi}(T_{2i+3}-T)}F-\int_{c_{p,\xi}(T_{2i+2}+s_0+T)}^{c_{p,\xi}(T_{2i+3}+s_0-T)}F\right|=0.
    \end{aligned}
\end{equation}
On the two intervals $[T_{2i+2}, T_{2i+2}+T]\cup [T_{2i+3}-T, T_{2i+3}]$, or if $|T_{2i+3}-T_{2i+2}|\le 2T$, the difference of the two integrals are bounded above by $2T\cdot K(d_K(g^{T_{2i+2}}v, g^{T_{2i+2}+s_0}w))^\a$. Note that from the proof of Lemma \ref{limitzero1}, $d^s(g^{T_{2i+4}}v, g^{T_{2i+4}+s_0}w)\le e^{-\frac{\l}{3}\d'}d^s(g^{T_{2i+2}}v, g^{T_{2i+2}+s_0}w)$ if $|T_{2i+4}-T_{2i+2}|\ge \d'$. In the extreme case that $|T_{2i+4}-T_{2i+2}|< \d'$, we combine the three intervals together and \eqref{decresingpart} holds on $[T_{2i+1}, T_{2i+4}]$.

Now noticing that  $d_K(g^{t}v, g^{t+s_0}w)$ is decreasing in $t$, we have that for every $t>t_0$,
\begin{equation}\label{total}
    \begin{aligned}
&\left|\int_{o}^{c_{o,\xi}(t)}F-\int_{p}^{c_{p,\xi}(t+s_0)}F\right|\le K(d_K(g^{t_0}v, g^{t_{0}+s_0}w))^\a\int_{t_0}^te^{-\a\l s/2}ds\\
&+2T\cdot K(d_K(g^{t_0}v, g^{t_{0}+s_0}w))^\a\sum_{i=1}^\infty e^{-\frac{\l}{3}\d'\a i}+2t_0\|F\|+d(p,o)\|F\|.\\
    \end{aligned}
\end{equation}

Similarly to the proof of Corollaries \ref{cocexi1} and \ref{cocexi2}, we denote 
$$a_t:=\int_{o}^{c_{o,\xi}(t)}F-\int_{p}^{c_{p,\xi}(t+s_0)}F.$$ 
As in the above proof, for any $t_2>t_1>t_0$ we have
\begin{equation*}
    \begin{aligned}
\left|a_{t_1}-a_{t_2}\right|\le &K(d_K(g^{t_1}v, g^{t_{1}+s_0}w))^\a\int_{t_1}^{\infty}e^{-\a\l t/2}dt\\
+&2T\cdot K(d_K(g^{t_1}v, g^{t_{1}+s_0}w))^\a\sum_{i=1}^\infty e^{-\frac{\l}{3}\d'\a i}
    \end{aligned}
\end{equation*}
which converges to $0$ as $t_1\to\infty$. This proves the lemma when $\xi\in \L_r^\l$.

If $\xi\notin \L_r^\l$, then there exists $T_1>0$ such that $c_{o,\xi}([T_1, +\infty))\cap \C K_\l=\emptyset$. We have two cases:
\begin{enumerate}
    \item If there exists $\{\c_n\}_{n=1}^\infty\subset \C$ such that $c_{p,\xi}([0,+\infty))\cap \c_n K_\l\neq \emptyset$, by repeating the above argument for $p$ instead of $o$, we prove the lemma.
\item If there exists $T_2>0$ such that $c_{p,\xi}([T_2, +\infty))\cap \C K_\l=\emptyset$, it is easy to see by Lemma \ref{nhdsing} that $F$ is constant on the interval $[T_3+T,\infty)$ and hence
$$C_{F,\xi}(p, o)=\int_{o}^{c_{o,\xi}(T_3+T)}F-\int_{p}^{c_{p,\xi}(T_3+T)}F$$
where $T_3=\max\{T_1,T_2\}$.
\end{enumerate}
The proof of the lemma is complete.
\end{proof}

\section{Patterson-Sullivan construction}

\subsection{Poincar\'{e} series, critical exponent and topological pressure}
Let $M$ be a closed rank one Riemannian manifold of nonpositive curvature and $X$ its universal cover with $M=X/\Gamma$. Then $\Gamma$ is an infinite discrete subgroup of the isometry group $\text{Iso}(X)$. Fix a reference point $o\in X$ and a fundamental domain $\F$ containing $o$. For any $s \in \mathbb{R}$, \emph{Poincar\'e series} is defined as
\begin{equation*}
P_F(s,o):= \sum_{\c \in \Gamma}e^{\int_o^{\c o}(F-s)}.
\end{equation*}
The \emph{critical exponent of $\Gamma$} is defined as
$$\delta_F:=\limsup_{n\to +\infty}\frac{1}{n}\log \sum_{\c \in \Gamma, n-1\le d(o, \c o)<n}e^{\int_o^{\c o}F}.$$
Poincar\'{e} series $P_F(s,o)$ diverges when $s<\delta_F$ and converges when  $s>\delta_F$. We say $\Gamma$ is of \emph{divergent type} if $P_F(s,o)$ diverges when $s=\delta_F$.
%Let $S\F$ denote the unit vector $v\in SX$ with $\pi v\in \F$. Denote the \emph{regular limit set} for $\lambda>0$ as
%$$\Lambda_R^\lambda(\Gamma):=\{v(+\infty)\in \Lambda\Gamma: v\in S\F\cap \UR(\l)\}.$$
%\begin{equation}\label{poincare series}
%P_F^\lambda(s,p,\F):= \sum_{\c \in [\Gamma_\lambda,\F]}e^{\int_p^{\c z_{\c}}(F-s)}.
%\end{equation}
%The \emph{$\l$-critical exponent of $\Gamma$} is defined as
%$$\delta_F^\lambda:=\limsup_{n\to\infty}\frac{1}{n}\log \sum_{\c \in [\Gamma_\lambda,\F], n-1\le d(p, z_\c)<n}e^{\int_p^{\c z_{\c}}F}.$$

%\begin{proposition}
%\begin{enumerate}
%  \item $\delta_F^\lambda$ is independent of $p$, $z_\c$ and $\F$.
%  \item The Poincar\'e series $P_F^\lambda(s,p,\F)$ converges if $s> \delta_F^\lambda$, and diverges if $s< \delta_F^\lambda$.
%      \end{enumerate}
%\end{proposition}

%It is clear that whether $\Gamma$ is of divergent type or not is independent of the choices of $p$, $z_\c$ and $\F$.
The following fact is interesting.
\begin{lemma}\label{criticalexponent}
\begin{enumerate}
    \item For every $c>0$, 
    $$\delta_F=\limsup_{n\to +\infty}\frac{1}{n}\log \sum_{\c \in \Gamma, n-c\le d(o, \c o)<n}e^{\int_o^{\c o}F}.$$
\item If $\d_F\ge 0$, then
 $$\delta_F=\limsup_{n\to +\infty}\frac{1}{n}\log \sum_{\c \in \Gamma,  d(o, \c o)<n}e^{\int_o^{\c o}F}.$$
\end{enumerate}   
\end{lemma}
\begin{proof}
The proof is identical to that of \cite[Lemma 3.3(vii)]{PPS}, and thus omitted here.
\end{proof}

The following lemma relates the topological pressure of the geodesic flow with the pressure over ``uniformly regular'' subsets.
Recall Definition \ref{lsets} for the definition of sets $\L_{k,\l, N}$ and $\tilde \L_k$.
\begin{lemma}\label{pesin}
Assume that $P(\text{Sing},F)<P(F)$. Then for every $k\in \NN$, we have
$$P(F)=P(\tilde{\L}_k, F)=\sup_{i\in \NN}\sup_{N\in \NN}P(\L_{k,\frac{1}{i},N}, F).$$
\end{lemma}
\begin{proof}
For an arbitrary subset $Z\subset SM$, let $P_Z(F)$ denote the pressure of $F$ on $Z$ using Carathe\'{o}dory-Pesin construction, see \cite[Theorem 11.1]{Pe3}.
Since \text{Sing} is closed and $g^t$-invariant, we know $P(F)=P_{SM}(F)$ and $P(\text{Sing},F)=P_{\text{Sing}}(F)$.
Now $P(\text{Sing},F)<P(F)$. Since 
$$P(F)=P_{SM}(F)=\max\{P_{\text{Sing}}(F),P_{\text{Reg}}(F)\},$$ 
we have $P(F)=P_{\text{Reg}}(F)$.

Let $\mu\in \M_{g^t}^e(SM)$ with $\mu(\text{Reg})=1$. Then $\mu(\URR)=1$. It is easy to see that $\iota_*\mu \in \M^e_{g^t}(SM)$. For any $0<\rho<1$ and $k\in \NN$, pick $i\in \NN$ large enough such that $\mu(\UR(\frac{1}{i}))>\max\{1-\rho, 1-\frac{1}{k}\}$ and
$$\mu(\iota(\UR(\frac{1}{i})))=(\iota_*\mu)(\UR(\frac{1}{i}))>\max\{1-\rho, 1-\frac{1}{k}\}.$$ 
Thus by Birkhorff ergodic theorem, if $N$ is large enough, we have $\mu(\L_{k,\frac{1}{i},N})>1-\rho$. So $\mu(\tilde \L_k)=1$ for every $k\in \NN$. 

From definition $\tilde \L_k$ is a $g^t$-invariant subset of $SM$. By the variational principle  \cite[Theorem A2.1]{Pe3},
\begin{equation*}
\begin{aligned}
P(F)=&\sup_{\mu\in \M^e_{g^t}(SM)}\Big\{h_\mu(f)+\int Fd\mu\Big\}\\
=&\sup_{\mu\in \M^e_{g^t}(SM),\ \mu(\text{Reg})=1 }\Big\{h_\mu(f)+\int Fd\mu\Big\}\\
=&\sup_{\mu\in \M^e_{g^t}(SM),\ \mu(\tilde \L_k)=1}\Big\{h_\mu(f)+\int Fd\mu\Big\}\\
%&=\sup_{\mu\in \M_{g^t}(SM),\ \mu(\L_U)=1}\left\{h_\mu(f)+\int Fd\mu\right\}\\
=&P_{\mathcal{L}(\tilde \L_k)}(F)\le P_{\tilde \L_k}(F)=\sup_{i\in \NN}\sup_{N\in \NN}P_{\L_{k,\frac{1}{i},N}}(F)\\
\le &\sup_{i\in \NN}\sup_{N\in \NN}P(\L_{k,\frac{1}{i},N}, F)\le P(\tilde \L_k, F)\le P(F)
\end{aligned}
\end{equation*}
where $\mathcal{L}(\tilde{\L}_k)$ is defined on \cite[p. 88]{Pe3}.
The lemma follows.
\end{proof}

\begin{proof}[Proof of Proposition \ref{pro9}]
Let $0<\e<\text{inj}(M)/4$. For any distinct $\c_1,\c_2\in \C$ satisfying $n-\e\le d(o, \c_io)<n, i=1,2$, we know that the two vectors $\dot c_{o,\c_io}(0), i=1,2$ are $(n,\e)$-separated. Indeed, otherwise, we have 
$$d(\c_1 o, \c_2o)\le 2\e+d( c_{o,\c_1o}(n), c_{o,\c_2o}(n))\le 3\e<\inj(M),$$ 
a contradiction. Then we know from Lemma \ref{criticalexponent}(1) that $\delta_F\le P(F).$

Let us prove the other direction. Let $k\in \NN$ be arbitrary. For every $i,N\in \NN$, let $S$ be a maximal $(T,\e)$-separated subset of $\L_{k,\frac{1}{i},N}$ with $T\gg N$. We still write $\l=1/i$ for convenience. Pick any $v\in S$ and still denote by $v$ its lifting to $X$ such that $\pi v\in \F$. Then there exists $\c\in \Gamma$ such that $\pi g^Tv\in \c \F$. Denote $w_T:=\dot c_{o, \pi g^Tv}(0), w:=\dot c_{o,\c o}(0)$ and $S_1=d(o, \pi g^Tv)$. Let $D:=\diam \F$ and so $|S_1-T|\le D$. Since $v\in \UR(\l)$, by Lemma \ref{contracting1}, for any $t_0>0$
$$d_K(g^{t_0}v, g^{S_1-T+t_0}w_T)\le C_0d^s(v,w')e^{d^s(v,w')/\lambda}e^{-\lambda t_0}$$
where $w'\in W^s(v)$ such that $\pi w'\in c_{o,\pi g^Tv}$. Since $v\mapsto H^s(v)$ is continuous in the sense of \cite[Proposition 6.3]{CKW1}, there exists $D'>0$ such that $d^s(v,w')\le D'$ for any $v, w'\in S\F$ with $w'\in W^s(v)$. Pick $t_0>0$ large enough (independent of $v$ and $T$) such that the last term above is less than $\lambda/100$.
By comparison theory,
\begin{equation*}
\begin{aligned}
&d_K(g^{S_1-T+t_0}w_T,g^{S_1-T+t_0}w)\\
\le &\frac{|S_1-T|+t_0}{S_1}d_K(g^{S_1}w_T,g^{S_1}w)\le \frac{D+t_0}{S_1}(4D+2).
\end{aligned}
\end{equation*}
Choose $T$ large enough such that the last term is less than $\lambda/100$. We obtain
$$d_K(g^{t_0}v, g^{S_1-T+t_0}w)\le \lambda/50.$$

By definition of $\L_{k,\l,N}$, there exists $s\in [0,\frac{1}{k}T]$ such that $\iota(g^{T-s}v)\in \UR(\l)$. Reversing the direction, we obtain similarly
$$d_K(g^{T-s-t_0}v, g^{S_2-s-t_0}w)\le \lambda/50$$
for some $S_2$ with $|S_2-T|\le 3D$. Thus by convexity of the distance function, for any $t\in [S_1-T+t_0,S_2-s-t_0]$, we have $g^tw$ is within a $\lambda/50$-neighborhood of the geodesic segment $(v,T)$.
By definition of $\L_{k,\l,N}$, there exist $s_1,s_2\in [0, \frac{1}{k}T]$ such that $g^{t_0+s_1}v\in \UR(\l)$ and $\iota(g^{T-s-t_0-s_2}v)\in \UR(\l)$.
%As $t_0$ is independent of $T$ and $v,\iota(g^sv)\in \UR(\l)$, there exists $\l_0:=\l_0(t_0,  \l)$ such that $g^{t_0}v,g^{t_0}(\iota(g^sv))\in \UR(\l_0)$. 
From the proof of Proposition \ref{rate}, we see that \eqref{jacobi} holds for the orbit segment of $w$ corresponding to $g^{[t_0+s_1, T-s-t_0-s_2]}v$. Thus applying Lemma \ref{contracting1} (with slight modification) twice,
\begin{equation}\label{e:Bowen1}
\left|\int_{\pi v}^{\pi g^Tv}F-\int_{o}^{\c o}F\right|\le L_1+(4t_0+6D+\frac{6}{k}T)\|F\|
\end{equation}
for some constant $L_1$ depending on $\l$, but independent of $v, T$.

The above argument shows that there is a well-defined map $q: S\to \Gamma$.
Consider a maximal $\e/3$-separated set $F_\e$ of $S\F$. We claim that $q$ is at most $(\# F_{\e})^2$ to $1$.
Indeed, $F_\e$ is $\e/3$-spanning of $S\F$ and given a $\c\in \C$, $\c F_\e$ is $\e/3$-spanning of $\c S\F$. Then there is a map $q': q^{-1}(\c)\to F_\e\times \c F_\e$ by choosing $q'(v)=(v_1,v_2)$ if $v\in B(v_1,\e/3)$ and $g^Tv\in B(v_2,\e/3)$. By the convexity of the distance function and $(T,\e)$-separatedness of $q^{-1}(\c)$, $q'$ is injective. Thus $\# q^{-1}(\c)\le (\# F_{\e})^2$. This proves the claim.

Combining with We have \eqref{e:Bowen1}, we have
\begin{equation*}
\begin{aligned}
\sum_{v \in S}e^{\int_{\pi v}^{\pi g^T v}F}\le \sum_{\c \in \Gamma, T-2D\le d(o, \c o)<T+2D} (\# F_{\e})^2 e^{L_1+(4t_0+6D+\frac{6}{k}T)\|F\|}e^{\int_{o}^{\c o}F}.
\end{aligned}
\end{equation*}
By Lemma \ref{criticalexponent}(1), we obtain $P(\L_{k,\frac{1}{i},N},F)\le \delta_F+\frac{6}{k}\|F\|$ for every $i,N\in \NN$. Thus by Lemma \ref{pesin},
$$P(F)=P(\tilde{\L}_k,F)\le \delta_F+\frac{6}{k}\|F\|.$$ 
Since $k$ is arbitrary, the proposition follows.
\end{proof}
%\begin{corollary}
%Let $(M,g)$ be a compact Riemannian manifold of nonpositive curvature and $X$ be its universal cover, then $$P(g^t,F)=\sup_{\lambda>0}\delta_F^\lambda.$$
%\end{corollary}

\subsection{Patterson-Sullivan measures under Condition (A)}
In this subsection, we assume that Condition (A) holds.
\subsubsection{Construction of $\{\mu_{F, \a o}: \a\in\C\}$}
Fix a reference point $o\in X$. For each $s > \delta_F$, consider the measure:
$$\mu_{F,o,s}:=\frac{1}{P_F(s,o)}\sum_{\c \in \Gamma}e^{\int_o^{\gamma o}(F-s)}\delta_{\c o},$$
where $\delta_{\c o}$ is the Dirac measure at point $\c o$ and $P_F(s,o)=\sum_{\c\in \C}e^{\int_o^{\gamma o}(F-s)}$ is the Poincar\'{e} series.
Then $\Gamma o \subset \text{supp} \mu_{F,o,s} \subset \overline{\Gamma o}$, where $\Gamma o$ is the orbit of $o\in X$ under the action of $\Gamma$.

Without loss of generality, we assume that $\Gamma$ is of divergent type.
Otherwise, we follow Patterson's  method (cf.~\cite{Pat}) to modify the definition of $\mu_{F,o,s}$ as follows.
As in the  proof of \cite[Proposition 3.9]{PPS}, let $h:[0,\infty)\to [0,\infty)$ be a non-decreasing map such that
\begin{enumerate}
  \item for any $\e>0$, there exists $r_\e\geq 0$ such that $h(t+r)\le e^{\e t}h(r)$ for any $t\ge 0$, $r\ge r_\e$;
  \item $\bar{P}_F(s,o):= \sum_{\c \in \Gamma}e^{\int_o^{\c o}(F-s)}h(d(o,\c o))$ diverges if and only if $s\le \d_F$.
\end{enumerate}
Then we consider
$$\mu_{F,o,s}:=\frac{1}{\bar P_F(s,o)}\sum_{\c \in \Gamma}e^{\int_o^{\gamma o}(F-s)}h(d(o,\c o))\delta_{\c o}.$$

From the definition, we see that $\mu_{F,o,s}(\overline{X})=1$. So consider a weak$^{\star}$ limit
\begin{equation}\label{ps1}
  \lim_{s_{k}\searrow \d_F}\mu_{F, o, s_{k}}=\mu_{F,o}.  
\end{equation}
\begin{lemma}
$\supp\mu_{F,o}\subset \pX$ and $\mu_{F,o}(\pX)=1$.
\end{lemma}
\begin{proof}
Since $\bar P_F(\d_F, o)=\infty$, we know $\supp\mu_{F,o}\subset \pX$. Clearly, $\mu_{F,o}(\pX)=\mu_{F,o}(\overline{X})=1$.
\end{proof}

We continue to define the measures $\mu_{F, \a o,s}$ for any $\a\in\C, s>\d_F$ as follows.
$$\mu_{F,\a o,s}:=\frac{1}{P_F(s,o)}\sum_{\c \in \Gamma}e^{\int_{\a o}^{\c o}(F-s)}\delta_{\c o}.$$
\begin{lemma}\label{equiv}
Let $(s_k)_{k=1}^\infty$ be the sequence in \eqref{ps1}. Then for every $\a\in \Gamma$ the limit
$\lim_{s_{k}\searrow \d_F}\mu_{F, \a o, s_{k}}$
exists and denoted by $\mu_{F, \a o}$. Moreover, we have $\a_*\mu_{F,o}=\mu_{F, \a o}$ for every $\a\in \Gamma$.
\end{lemma}
\begin{proof}
Let $A\subset \pX$ be a Borel measurable set. Then for every $\a\in \Gamma$,
\begin{equation*}
    \begin{aligned}
&(\a_*\mu_{F,o,s_k})(A)=\mu_{F,o,s_k}(\a^{-1}A)\\
=&\frac{1}{P_F(s_k,o)}\sum_{\c \in \Gamma}e^{\int_o^{\gamma o}(F-s_k)}\delta_{\c o}(\a^{-1}A)=\frac{1}{P_F(s_k,o)}\sum_{\c \in \Gamma}e^{\int_o^{\gamma o}(F-s_k)}\delta_{\a \c o}(A)\\
=&\frac{1}{P_F(s_k,o)}\sum_{\c' \in \Gamma}e^{\int_{\a o}^{\gamma' o}(F-s_k)}\delta_{\c' o}(A)=\mu_{F,\a o,s_k}(A).
    \end{aligned}
\end{equation*}
Since  $\lim_{s_{k}\searrow \d_F}\a_*\mu_{F,o,s_k}=\a_*\mu_{F,o}$, we have $\lim_{s_{k}\searrow \d_F}\mu_{F,\a o,s_k}=\a_*\mu_{F,o}$.
\end{proof}

\subsubsection{Full $\mu_o$-measure of the core limit set}
\begin{proposition}\label{full}
    $\mu_{F,o}((\L_c(\C))^c)=0$.
\end{proposition}
\begin{proof}
Denote by $\CCC:=\{[o,\c o]: \c\in \C\}$. Fix a sufficiently small $0<\e\ll \inj(M)/4$. Denote $\tilde{\CC}_t:=\{v\in S_oX: v=\dot c_{o,\c o}(0), t-\e<d(o,\c o)\le t, \c\in \C\}$. For $t>0$ large enough, $\tilde\CC_t$ is a $(t,2\e)$-separated set. 

Recall in \cite{CT16} that the notation $[\DDD]$ for $\DDD\subset SM\times \RR^+$ means
$$[\DDD]:=\{(v,n)\in SX\times \NN: (g^{-s}v, n+s+t)\in \DDD \text{\ for some\ }s,t\in [0,1]\}.$$
By \cite[Proposition 5.2]{BCFT}, there exists $\b_1>0$ such that
$P([\PPP]\cup[\SSS], F, \e, 3\e)<P(F)-2\b_1$. Therefore, there exists a constant $C_1>0$ such that
$$\L_t([\PPP]\cup[\SSS], F, \e, 3\e)\le C_1e^{t(P(F)-\b_1)}, \quad \forall t>0.$$

We want to estimate 
$$\sum_{(p\vee s)(v,t)>L, v\in \tilde\CC_t}e^{\F_{2\e}(v,t)}$$
where $p\vee s:=\max\{p,s\}.$ Given $v\in \tilde\CC_t$ with $\lfloor p(v,t)\rfloor =i$ and $\lfloor s(v,t)\rfloor =k$, we have
$$(v,i)\in [\PPP], (g^iv, t-i-k)\in \GGG^1, (g^{t-k}v,k)\in [\SSS].$$
Given $i,k\in \{0,1,\cdots, \lceil t\rceil\}$, define
$$C(i,k):=\{v\in \tilde\CC_t: \lfloor p(x,t)\rfloor =i, \lfloor s(x,t)\rfloor =k\}.$$
For each $0\le i \le \lceil t\rceil$, define $E_i^\P \subset [\P]_i$ to be a maximal $(i,\e)$-separated set. $E_j^{\GGG^1}\subset \GGG^1_j$ and $E_k^\SSS \subset [\SSS]_k$ are defined similarly. According to the proof of \cite[Lemma 4.8]{CT16}, there exists an injection $\pi: C(i,k)\to E_i^\P\times E_{t-i-k}^{\GGG^1}\times E_k^\SSS$
by $\pi(v)=(v_1, v_2,v_3)$ such that
\begin{itemize}
    \item $v_1\in E_i^\P$ satisfies $v\in \bar B_i(v_1,\e)$,
     \item $v_2\in E_{t-i-k}^{\GGG^1}$ satisfies $g^iv\in \bar B_{t-i-k}(v_2,\e)$,
     \item $v_3\in E_k^\SSS$ satisfies $g^{t-k}v\in \bar B_k(v_3,\e)$.
\end{itemize}
So if $v\in C(i,k)$, 
$$\F_{2\e}(v,t)\le \F_{3\e}(v_1, i)+\F_{3\e}(v_2, t-i-k)+\F_{3\e}(v_3, k).$$
Let $L\in \NN$. Using \cite[Proposition 4.7]{CT16}, we have
\begin{equation}\label{badsegment}
    \begin{aligned}
&\sum_{(p\vee s)(v,t)>L, v\in \tilde\CC_t}e^{\F_{2\e}(v,t)}= \sum_{i\vee k> L}\sum_{v\in C(i,k)}e^{\F_{2\e}(v,t)}\\
\le &\sum_{i\vee k> L}\L_i([\PPP],\e, 3\e)\L_k([\SSS],\e, 3\e)\L_{t-i-k}(\GGG^1,\e, 3\e)\\
\le &C_2\sum_{i\vee k> L}\L_i([\PPP],\e, 3\e)\L_k([\SSS],\e, 3\e)e^{(t-i-k)P(F)}\\
\le &C_1^2C_2\sum_{i\vee k\ge L}e^{(i+k)(P(F)-\b_1)}e^{(t-i-k)P(F)}\\
= &C_1^2C_2e^{tP(F)}\sum_{i\vee k> L}e^{-(i+k)\b_1}.
    \end{aligned}
\end{equation}
Denote $K(L):=\sum_{i\vee k> L}e^{-(i+k)\b_1}$. Then $\lim_{L\to \infty}K(L)=0$.

Fix $\e=1/l\ll\inj(M)/4$ for sufficiently large $l\in \NN$. Let $\xi\in (\L_c)^c$. For any $L\in \NN$, there exists a sufficiently small open neighborhood $U(\xi,L)$ of $\xi$ in $\overline{X}$ such that if $\c\in \C$ satisfies $\c o\in U$ and $j\e-\e<d(o, \c o)\le j\e$ for some $j\in \NN$, then $(p\vee s)(v,j\e)>L$ where $v=\dot c_{o,\c o}(0)\in \tilde\CC_{j\e}$. Denote $U_L:=\cup_{\xi\in (\L_c)^c}U(\xi,L)$. Then $(\L_c)^c\subset \cap_{L=1}^\infty U_L$. By \eqref{badsegment} and Condition (A), we have
\begin{equation*}
    \begin{aligned}
&\mu_{F,o}(U_L)\le \lim_{s_k\searrow \d_F}\frac{1}{P_F(s_k,o)}e^{\e\|F\|}\sum_{j=1}^\infty\sum_{(p\vee s)(v,t)>L, v\in \tilde\CC_{j\e}}e^{\F_{2\e}(v,j\e)}e^{-s_k(j-1)\e}\\
%&\le \lim_{s_k\to \infty}\frac{\sum_{\c\in \C, n\le d(o, \c o)<n+1} K(M)C_1C_2^2e^{(n+1)P(F)}}{\sum_{\c\in \C, n\le d(o, \c o)<n+1}e^{\int_o^\c o}(F-s_k)}\\
\le &\lim_{s_k\searrow \d_F}\frac{ e^{\e\|F\|}C_1^2C_2 K(L)\sum_{j=1}^\infty e^{j\e P(F)}e^{-s_k(j-1)\e}}{\sum_{n=1}^\infty\sum_{\c\in \C, n-1< d(o, \c o)\le n}e^{\int_o^{\c o}(F-s_k)}}\\
\le &\lim_{s_k\searrow \d_F}\frac{ e^{\e\|F\|}C_1^2C_2 K(L)\sum_{n=1}^\infty l e^{nP(F)+|P(F)|}e^{-s_k(n-1)}e^{s_k\e}}{\sum_{n=1}^\infty e^{-s_kn}\sum_{\c\in \C, n-1< d(o, \c o)\le n}e^{\int_o^{\c o}F}}\\
\le &\lim_{s_k\searrow \d_F}\frac{ e^{\e\|F\|}C_1^2C_2 K(L)e^{|P(F)|+s_k\e}l\sum_{n=1}^\infty e^{nP(F)}e^{-s_k(n-1)}}{\sum_{n=1}^\infty e^{-s_kn}Ce^{n\d_F}}\\
= &\frac{e^{\e\|F\|}C_1^2C_2 K(L)e^{|P(F)|+\d_F\e+\d_F}l}{C}.
    \end{aligned}
\end{equation*}
Since $K(L)\to 0$ as $L\to \infty$, we see that $\mu_{F,o}((\L_c)^c)=0$.
\end{proof}

\begin{proposition}\label{Busemann2}
For every $\c\in \C$, we have
$$\frac{d\mu_{F,\c o}}{d\mu_{F,o}}(\xi)=e^{-C_{F-\d_F, \xi}(\c o,o)}$$
for $\mu_{F,o} \ae\ \xi\in \pX$.
\end{proposition}
\begin{proof}
By Proposition \ref{full}, it is sufficient to prove the proposition for $\xi\in \L_c$. So the Busemann cocycle $C_{F-s, \xi}(\c o,o)$ is well defined by Corollary \ref{cocexi2} for any $s\ge \d_F$.

Fix $\d$ from \eqref{smalldelta}. By Lemma \ref{limitzero} and Corollary \ref{cocexi2}, for any small $\e>0$, there exists $T>0$ large enough such that 
$$d^s(\dot c_{o,\xi}(t), \dot c_{\c o,\xi}(t+s_0))<\d/10, \quad \forall t\ge T$$ 
and
\begin{equation}\label{e:est1}
\begin{aligned}
\left|C_{F-s,\xi}(\c o,o)-\int_{o}^{c_{o,\xi}(T)}(F-s)+\int_{\c o}^{c_{\c o,\xi}(T+s_0)}(F-s)\right|<\e/5
\end{aligned}
\end{equation}
where $s_0=b_\xi(\c o,o)$. Since $\xi\in \L_c$, let $\c_n\in \C$ be as in Definition \ref{corelimit}. Then $\c_n o\to \xi$ as $n\to \infty$. For small $0<\rho \ll \min\{\e, \d/10\}$, pick $n$ large enough such that 
$$\max\{\angle_o(\c_n o, c_{o,\xi}(T)), \angle_{\c o}(\c_n o, c_{\c o,\xi}(T+s_0))\}\ll \rho.$$
Therefore by comparison theorem,
$$\max\{d_K(\dot c_{o,\c_n o}(T), \dot c_{o,\xi}(T)), d_K(\dot c_{\c o,\c_n o}(T+s_0), \dot c_{\c o,\xi}(T+s_0))\}< \rho.$$
Therefore, if $\rho$ is small enough, then
\begin{equation}\label{e:est2}
\begin{aligned}
\left|\int_{o}^{c_{o,\xi}(T)}(F-s)-\int_{o}^{c_{o,\c_n o}(T)}(F-s)\right|&<\e/5\\
\left|\int_{\c o}^{c_{\c o,\xi}(T+s_0)}(F-s)-\int_{\c o}^{c_{o,\c_n o}(T+s_0)}(F-s)\right|&<\e/5.
\end{aligned}
\end{equation}

By the choice of $T$ and $\rho$, $d_K(\dot c_{o,\c_n o}(T),\dot c_{\c o,\c_n o}(T+s_0))<\d/5$.
Recall that $[o,\c o]\in \GGG^L$. By a similar argument as in the proof of Lemma \ref{decay}, there exists $T_1>T$ such that for any $T_1<t<d(o,\c_no)-L$,
\[\frac{1}{t-T}\int_T^t\lambda(g^s \dot c_{o,\c_n o}(0))ds\ge \frac{99\eta}{100}.\]
Then following a similar argument as in the proof of Lemma \ref{Bowen1} and noting that $L$ is fixed, we have
\begin{equation}\label{e:est3}
\begin{aligned}
\left|\int_{c_{o,\c_n o}(T_1)}^{\c_n o}(F-s)-\int_{c_{\c o,\c_n o}(T_1+s_0)}^{\c_n o}(F-s)\right|<\e/5
\end{aligned}
\end{equation}
by enlarging $T$ if necessary. But we still need consider the time interval $[T,T_1]$. Here we emphasize that the choice of $T_1$ is universal for all orbit segments in $\GGG^L$, and hence independent of the particular orbit segment $[o,\c_n o]$. Thus we can enlarge $n$ enough, so that  \eqref{e:est1}, \eqref{e:est2} still hold for $T_1$ instead of $T$, and meanwhile on this new orbit segment $[o,\c_n o]$, \eqref{e:est3} also holds.
Combining \eqref{e:est1}, \eqref{e:est2} (both with $T_1$ instead of $T$) and \eqref{e:est3}, we know
\begin{equation*}
\begin{aligned}
C_{F-s,\xi}(\c o,o)=\lim_{n\to \infty }\int_{o}^{\c_n o}(F-s)-\int_{\c o}^{\c_n o}(F-s).
\end{aligned}
\end{equation*}

Now take a sequence of neighborhoods $\{U_i\}_{i=1}^\infty$ of $\xi$ in $\overline X$ such that $\{\xi\}=\cap_{i} U_i$ and $\mu_{F,o}(\partial U_i)=\mu_{F,\c o}(\partial U_i)=0, \forall i\in \NN$. We have showed that for any $\c_i o\in U_i$,
\begin{equation*}
\begin{aligned}
\frac{e^{\int_{\c o}^{\c_i o}(F-s)}}{e^{\int_{o}^{\c_i o}(F-s)}}\to e^{-C_{F-s,\xi}(\c o,o)}
\end{aligned}
\end{equation*}
as $i\to \infty$. From the above proof, this convergence is uniform for all $s$ in a compact interval $[\d_F, \d_F+c]$ for some $c>0$.
%Note that $\lim_{r\to +\infty}\frac{h(t+r)}{h(r)}=1$ for any $t>0$. 
Then
\begin{equation*}
\begin{aligned}
&\frac{d\mu_{F,\c o}}{d\mu_{F,o}}(\xi)=\lim_{i\to \infty}\frac{\mu_{F, \c o}(U_i)}{\mu_{F, o}(U_i)}=\lim_{i\to \infty}\lim_{s_k\searrow \d_F}\frac{\mu_{F, \c o,s_k}(U_i)}{\mu_{F, o,s_k}(U_i)}\\
=&\lim_{s_k\searrow \d_F}\lim_{i\to \infty}\frac{\mu_{F, \c o,s_k}(U_i)}{\mu_{F, o,s_k}(U_i)}=e^{-C_{F-\d_F,\xi}(\c o, o)}.
\end{aligned}
\end{equation*}
The proof of the proposition is complete.
\end{proof}

\subsubsection{Construction of Patterson-Sullivan measures}
By Proposition \ref{full}, we have $\mu_{F,o}((\L_c)^c)=0$. For any $\xi\in \L_c$ and any $q\in X$, $C_{F-\d_F,\xi}(q,o)$ is well defined by Corollary \ref{cocexi2}. Then we define a family of measures $\{\mu_{F,q}\}_{q\in X}$ on $\L_c\subset \pX$ as follows:
\[d\mu_{F,q}(\xi):=e^{-C_{F-\d_F,\xi}(q,o)}d\mu_{F,o}, \quad \forall \xi\in \Lambda_c.\]
When $q=\c o, \c\in \C$, the above definition of $\mu_{F, \c o}$ coincides with the previous one, according to Proposition \ref{Busemann2}. Note that $\mu_{F, \c o}, \c\in \C$ are all probability measures on $\pX$. However, for general $q\in X$, $\xi\mapsto C_{F-\d_F,\xi}(q,o)$ may be not bounded. So it is not evident that $\mu_{F,q}$ is a finite or Radon measure on $\pX$.
\begin{lemma}\label{gamma}
$\{\mu_{F,q}\}_{q\in X}$ is $\Gamma$-equivariant, i.e.,
 $$\mu_{F,\gamma q}(A)=\mu_{F,q}(\c^{-1}A)$$
  for any $\c\in \Gamma$ and any Borel set $A\subset \pX$.
\end{lemma}
\begin{proof}
For any Borel set $A\subset \pX$, by Lemmas \ref{cocycleprop} and \ref{equiv},
\begin{equation*}
\begin{aligned}
&\mu_{F, \c q}(A)=\int_Ad\mu_{F, \c q}(\xi)=\int_A e^{-C_{F-\d_F,\xi}(\c q, o)}d\mu_{F,o}(\xi)\\
=&\int_A e^{-C_{F-\d_F,\xi}(\c q,\c o)-C_{F-\d_F,\xi}(\c o, o)}d\mu_{F,o}(\xi)\\
=&\int_Ae^{-C_{F-\d_F,\xi}(\c q,\c o)}d\mu_{F,\c o}(\xi)=\int_{\c^{-1}A} e^{-C_{F-\d_F,\eta}(q,o)}d\mu_{F, o}(\eta)\\
=&\int_{\c^{-1}A} d\mu_{F, q}(\eta)=\mu_{F, q}(\c^{-1}A).
\end{aligned}
\end{equation*}
The lemma is proved.
\end{proof}

In summary, we have proved the following:
\begin{proposition}\label{ps_a}
$\{\mu_{F,q}\}_{q\in X}$ is a $\d_F$-dimensional Busemann density, that is
\begin{enumerate}
  \item $\mu_{F,\gamma q}(\gamma A)=\mu_{F,q}(A)$ for any $\c\in \Gamma$ and any Borel set $A\subset \pX$;
  \item $\frac{d\mu_{F,q}}{d\mu_{F,p}}(\xi)=e^{-C_{F-\d_F,\xi}(q,p)}$ for almost every $\xi\in \pX$.
\end{enumerate}
\end{proposition}

\subsection{Patterson-Sullivan measures under Condition (B)}
Under Condition (B), we can still follow the steps in Subsection 3.2 to construct a family of Patterson-Sullivan measures. In fact, by Lemma \ref{cocyle2} the Buseman function is well defined for every $\xi\in \pX$. In this case, we can refine Proposition \ref{Busemann2} as follows. So the corresponding version of Proposition \ref{full} is not needed.

\begin{proposition}\label{Busemann3}
For every $\c\in \C$, we have
$$\frac{d\mu_{F,\c o}}{d\mu_{F,o}}(\xi)=e^{-C_{F-\d_F, \xi}(\c o,o)}$$
for every $\xi\in \pX$.
\end{proposition}
\begin{proof}
Pick $\rho$ small enough such that $F\equiv c$ on a $\rho$-neighborhood of $\text{Sing}$. By Lemma \ref{nhdsing}, there exist $\l>0$ and $T>0$ such that if $\l^s(g^tv)\le \l$ for all $t\in [-T,T]$, then $d_K(v, \text{Sing})<\rho$. 

Assume first that $\xi\in \L_r^\l$. Then we can follow the lines of the proof of Proposition \ref{Busemann2}. Indeed, by Lemma \ref{limitzero1}, there exists $T_1>0$ such that 
\begin{equation}\label{smalldis}
\begin{aligned}
d^s(\dot c_{o,\xi}(T_1), \dot c_{\c o,\xi}(T_1+s_0))<\d/10
\end{aligned}
\end{equation}
where $s_0=b_\xi(\c o, o).$ We have to prove \eqref{e:est3}, i.e.,
\begin{equation}\label{tri}
\begin{aligned}
\left|\int_{c_{o,\c_n o}(T_1)}^{\c_n o}(F-s)-\int_{c_{\c o,\c_n o}(T_1+s_0)}^{\c_n o}(F-s)\right|<\e/5
\end{aligned}
\end{equation}
provided that $d(c_{o,\c_n o}(T_1),c_{\c o,\c_n o}(T_1+s_0))<\d/10$. This can be achieved by following the ideas of the proof of Lemmas \ref{Bowen} and \ref{cocyle2}. More precisely, if any one of the three geodesic rays, $c_{o,\c_n o}([T_1,+\infty))$, $[c_{\c o,\c_n o}(T_1+s_0), c_{o,\c_n o}(+\infty))$ and $[c_{\c o,\c_n o}(T_1+s_0), \c_n o]$, intersects $\C K_\l$, we have similar esitmates as in \eqref{decresingpart}. For the remaining time interval, we have estimates as in \eqref{constantpart} and \eqref{total}. In this way, we get a similar estimates as in \eqref{total}
\begin{equation*}
\begin{aligned}
&\left|\int_{c_{o,\c_n o}(T_1)}^{\c_n o}(F-s)-\int_{c_{\c o,\c_n o}(T_1+s_0)}^{\c_n o}(F-s)\right|\\
\le &2K(d^s(g^{T_1}w_1, g^{T_1}w_2))^\a\int_{T_1}^{\infty}e^{-\a\l t/2}dt\\
+&2T\cdot K(d^s(g^{T_1}w_1, g^{T_1}w_2))^\a\sum_{i=1}^\infty e^{-\frac{\l}{3}\d'\a i}
\end{aligned}
\end{equation*}
where $w_1=\dot c_{o,\c_n o}(0)$, and $w_2\in W^s(w_1)$ with $c_{\c o,\c_n o}(T_1+s_0)\in c_{w_2}$.
By enlarging $T_1$, we get the above \eqref{tri} and finish the proof in the case $\xi\in \L_r^\l$.

If $\xi\notin \L_r^\l$ but there exists $\{\c_n\}_{n=1}^\infty\subset \C$ such that $c_{\c o,\xi}([0,+\infty))\cap \c_n K_\l\neq \emptyset$, the above argument works as well.

The remaining case is when there exists $T_2>0$ such that 
$$c_{o,\xi}([T_2, +\infty))\cap \C K_\l=c_{\c o,\xi}([T_2, +\infty))\cap \C K_\l=\emptyset.$$
The main difference and main difficulty is that we may not have \eqref{smalldis}. In other words, it is possible that $\lim_{t\to +\infty}d^s(\dot c_{o,\xi}(t), \dot c_{\c o,\xi}(t+s_0))=c$ for some $c>0$. Then $d^s(\dot c_{o,\c_n o}(T_1),\dot c_{\c o,\c_n o}(T_1+s_0))$ is almost $c$ by the choice of $T_1$.

To overcome this difficulty, take a shortest curve $\beta: [0,1]\to W^s(v)$ with $\beta(0)=v$ and $\b(1)=v_1$ where $v_1=\dot c_{\c o,\xi}(s_0)$. Then there exists $\tilde{\b}$ satisfying $\tilde{\b}(s,t+t_s)=\b(s,t)$ where $t_s=b_{\xi}(\pi \tilde{\b}(s, 0),o)$. We claim that each geodesic ray $\tilde{\b}(s,[0,+\infty))$ cannot intersect $\C K_\l$ infinitely many times. Otherwise, by a similar argument as in the proof of Lemma \ref{limitzero1}, we have a contradiction to $\lim_{t\to +\infty}d^s(\dot c_{o,\xi}(t), \dot c_{\c o,\xi}(t+s_0))=c$.

We divide the curve $\tilde{\b}(s,0)$ into $N=\lceil\frac{d^s(v,\tilde\b(1,0))}{\d/100}\rceil$ small pieces with length less than $\d/100$. Let $0=s_0<s_1<s_2<\cdots<s_N$ be the boundary points of these pieces. Then for each piece, \eqref{smalldis} holds. Thus we can now repeat the above argument to show that if $\c_no\to \xi$, then
\begin{equation*}
\begin{aligned}
C_{F-s,\xi}(\pi\tilde{\b}(s_{i+1}, 0),\pi\tilde{\b}(s_{i}, 0))=\lim_{n\to \infty }\int_{\pi\tilde{\b}(s_{i}, 0)}^{\c_n o}(F-s)-\int_{\pi\tilde{\b}(s_{i+1}, 0)}^{\c_n o}(F-s).
\end{aligned}
\end{equation*}
Since each side of the above equality has cocycle property, we then have 
\begin{equation*}
\begin{aligned}
C_{F-s,\xi}(\c o,o)=\lim_{n\to \infty }\int_{o}^{\c_n o}(F-s)-\int_{\c o}^{\c_n o}(F-s).
\end{aligned}
\end{equation*}
We are done with the proof.
\end{proof}
%\begin{remark}
%There is an alternative way to overcome the above difficulty. If $\lim_{t\to +\infty}d^s(\dot c_{o,\xi}(t), \dot c_{\c o,\xi}(t+s_0))=c>0$,
%\end{remark}

\begin{proof}[Proof of Theorem \ref{ps}]
Under Condition (A), Theorem \ref{ps} is exactly Proposition  \ref{ps_a}. Under Condition (B), using Proposition \ref{Busemann3}, we can prove Theorem \ref{ps} by following all steps in Subsection 3.2 except Proposition \ref{full}.  
\end{proof}

\subsection{Shadow lemma}
We obtain some properties of Patterson-Sullivan measures in this subsection. Particularly, we prove a version of (half) shadow lemma.

The following property of rank one geodesics plays significant roles in nonpositive curvature.
\begin{lemma}(\cite[Lemma 2.1]{Kn}, \cite[Lemma 3.1]{BB})\label{key}
Let $c$ be a rank one geodesic on $X$. For each $\epsilon > 0$ there are neighborhoods $U$ of $c(-\infty)$ and $V$ of $c(+\infty)$ such that for all $\xi\in U$ and $\eta\in V$ there exists a rank one geodesic $h$ connecting $\xi$ and $\eta$ such that $\dot h(0)\in B(\dot c(0),\epsilon)$.

Moreover, if $c$ is an axis of $\gamma\in \Gamma$, then there exists $n_0$ such that
$\gamma^n(\overline X\setminus U)\subset V$ and $\gamma^{-n}(\overline X\setminus V)\subset U$
for all $n>n_0$.

In particular, the endpoints of a rank one axis can be connected by a rank one geodesic to any other point in $\pX.$
\end{lemma}
\begin{lemma}\label{supp}
For any $p\in X$, $\supp \mu_{F,p}= \pX$.
\end{lemma}

\begin{proof}
Suppose $\supp(\mu_{F,p})\neq \pX$, then there exists $\xi \in \pX$ which is not in $\supp(\mu_{F, p})$. So we can find an open neighborhood $U$ of $\xi$ in $\pX$ such that $\mu_{F, p}(U)=0$.

Take any $\eta\in \pX$, and then $\eta'\in \pX$ with $\eta'\neq\eta$. Choose a neighborhood $V$ of $\eta'$ in $\pX$ which does not contain $\eta$. Since rank one axes are dense $SM$ (cf. \cite{Bal0}), there exists $\c\in \Gamma$ such that $\c(\pX\setminus U)\subset V$ by Lemma \ref{key}. Thus $\eta\in \c U$.
By the $\Gamma$-equivariance, $\mu_{F, \c p}(\c U)=\mu_{F, p}(U)=0$. Since $\mu_{F, p}$ is equivalent to $\mu_{F, \c p}$, we have $\mu_{F, p}(\c U)=0$. Since $\eta$ is arbitrary, $\supp(\mu_{F, p}) = \emptyset$, which is a contradiction to the nontriviality of $\mu_{F,p}$.
\end{proof}

For each $\sigma \in \overline X$, define the projection map
$$pr_{\sigma}: X \rightarrow \pX, ~~q \mapsto c_{\sigma,q}(+\infty).$$
\begin{proposition}\label{alph}
For any small $r>0$, $k\in \NN, \l>0$ and $N\in \NN$, there exists $\rho(r, k, \lambda, N)>1$ such that for any $p\in \F, x\in X$, $\xi=c_{p,x}(-\infty)$ with $c_{p,x}(0)=p\in \F, v:=\dot c_{p,x}(0)\in \L_{k,\l, N}$ and $d(p,x)\gg N$, we have
\begin{equation}\label{e:alph}
\mu_{F,p}(pr_{\xi}(B(x,r))\cap G_\xi(k, \l, N))\le \rho e^{\int_p^x(F-\d_F)+\frac{8k-2}{k(k-1)}d(p,x)\|F\|}
\end{equation}
where $G_\xi(k, \l, N):=\{\eta\in \pX: c_{\xi,\eta}(0)\in \F, \dot{c}_{\xi,\eta}(0)\in \L_{k,\l, N}\}$.
\end{proposition}

\begin{proof}
We have
\begin{equation}\label{e:alph1}
\begin{aligned}
\mu_{F,p}(pr_{\xi}(B(x,r))\cap G_\xi(k, \l, N))=\int_{pr_{\xi}(B(x,r))\cap G_\xi(k, \l, N)}e^{-C_{F-\d_F,\eta}(p,x)}\mu_{F,x}(\eta).
\end{aligned}
\end{equation}
Take $\eta \in pr_{\xi}(B(x,r))\cap  G_\xi(k, \l, N)$ and $y\in c_{p,\eta}\cap H^s(x,\eta).$ Take any $y'\in c_{p,\eta}\cap B(x,r)$. Then $|b_\eta(y',x)|\le d(y',x)\le r$ and hence $d(y,y')\le r$. We see that $y\in B(x,2r)$. Since $v=\dot{c}_{p,x}(0)\in  \L_{k,\l, N}$ and $d(p,x)\gg N$, by \eqref{keyset}, there exists $s\in [(1-\frac{1}{k})d(p,x), d(p,x)]$ such that
$\iota(g^sv)\in \UR(\l)$. In a $r$-neighborhood of $\iota(g^sv)$, there is a local product structure with constant $\kappa>0$ (cf. \cite[Lemma 4.4]{BCFT}). Then by Lemma \ref{Bowen}
\begin{equation}\label{up2}
\begin{aligned}
&\left|\int_p^x(F-\d_F)-\int_p^y(F-\d_F)\right|\\
\le& \left|\int_p^{c_{p,x}(s)}(F-\d_F)-\int_p^{c_{p,y}(s)}(F-\d_F)\right|+\frac{2}{k}d(p,x)\|F\|+2r\|F\|\\
\le &C_1(4\kappa r)^\a e^{4\kappa r\a/\lambda}/\l+6r\|F\|+\frac{2}{k}d(p,x)\|F\|.
\end{aligned}
\end{equation}

Denote $z_0:=c_{\xi,\eta}(0)\in \F$ and $z_1\in c_{\xi,\eta}\cap H^s(x,\eta)$. Similarly, we have $z_1\in B(x,2r)$. Note that $d(z_0,p)\le D:=\diam \F$ and thus 
$$|d(z_0,z_1)-d(p,x)|\le d(z_0,p)+d(x,z_1)\le D+2r.$$ 
Since $w:=\dot{c}_{\xi,\eta}(0)\in  \L_{k,\l, N}$, by \eqref{keyset}, there exists $s'\in [d(z_0,z_1), \frac{k}{k-1}d(z_0,z_1)]$ such that $g^{s'}w\in \UR(\l)$.
Thus by Lemma \ref{Bowen},
\begin{equation}\label{up1}
\begin{aligned}
&\lim_{t\to \infty}\left|\int_x^{c_{x,\eta}(t)}(F-\d_F)-\int_y^{c_{y,\eta}(t)}(F-\d_F)\right|\\
\le &\lim_{t\to \infty}\left(\left|\int_x^{c_{x,\eta}(t)}(F-\d_F)-\int_{z_1}^{c_{z_1,\eta}(t)}(F-\d_F)\right|\right)\\
+&\lim_{t\to \infty}\left(\left|\int_{z_1}^{c_{z_1,\eta}(t)}(F-\d_F)-\int_y^{c_{y,\eta}(t)}(F-\d_F)\right|\right)\\
\le &2C_2(4\kappa r)^\a e^{4\kappa r\a/\lambda}/\l+8\|F\|r+\frac{4}{k-1}d(z_0,z_1)\|F\|\\
\le &2C_2(4\kappa r)^\a e^{4\kappa r\a/\lambda}/\l+8\|F\|r+\frac{4}{k-1}d(p,x)\|F\|+\frac{4(D+2r)}{k-1}\|F\|.
\end{aligned}
\end{equation}

Therefore, by \eqref{up2} and \eqref{up1} we have
\begin{equation}\label{e:alph2}
\begin{aligned}
&\left|\int_p^x (F-\d_F)+C_{F-\d_F,\eta}(p,x)\right|\\
=&\lim_{t\to \infty}\left|\int_p^x(F-\d_F)-\int_p^y(F-\d_F)+\int_x^{c_{x,\eta}(t)}(F-\d_F)-\int_y^{c_{y,\eta}(t)}(F-\d_F)\right|\\
\le &C_1(4\kappa r)^\a e^{4\kappa r\a/\lambda}/\l+6r\|F\|+\frac{2}{k}d(p,x)\|F\|\\
+&2C_2(4\kappa r)^\a e^{4\kappa r\a/\lambda}/\l+8\|F\|r+\frac{4}{k-1}d(p,x)\|F\|+\frac{4(D+2r)}{k-1}\|F\|\\
=:&L_1(r, k,\l, N)+\frac{6k-2}{k(k-1)}d(p,x)\|F\|.
\end{aligned}
\end{equation}
Let $\a\in \C$ be such that $\a x\in \F$. As we see above, for any $\eta\in G_\xi(k, \l, N)$, there exists $s'\in [d(z_0,z_1), \frac{k}{k-1}d(z_0,z_1)]$ such that $\dot c_{\xi,\eta}(s')\in \UR(\l)$. Then 
\begin{equation*}
\begin{aligned} 
d(\a c_{\xi,\eta}(s'),o)&\le d(c_{\xi,\eta}(s'),x)+D\le 2r+\frac{1}{k-1}d(z_0,z_1)+D\\
&\le \frac{1}{k-1}d(p,x)+\frac{k}{k-1}(D+2r).
 \end{aligned}
\end{equation*}
Similarly, we have
$$d(\a c_{\xi,\eta}(s'),\a x)\le \frac{1}{k-1}d(p,x)+\frac{k}{k-1}(D+2r).$$
Note that $\a c_{\xi,\eta}(s')\in \UR(\l)$ and $\mu_{F,o}(\pX)=1$. 

Following a similar computation as above, we have
\begin{equation}\label{up}
\begin{aligned} 
&\mu_{F,x}(pr_{\xi}(B(x,r))\cap G_\xi(k,\l,N))\le \mu_{F,x}(G_\xi(k,\l,N))\\
=&\mu_{F,\a x}(\a G_\xi(k,\l,N))=\int_{\a G_\xi(k,\l,N)}e^{-C_{F-\d_F}(\a x,o)}d\mu_{F,o}(\zeta)\\
\le & e^{L_1(r,k,\l,N)+L_1(D+r,k,\l,N)+(\frac{2}{k-1}d(p,x)+\frac{2k}{k-1}(D+2r))\|F\|}\mu_{F,o}(\a G_\xi(\l,k,N))\\
\le &e^{ L_1(r,k,\l,N)+L_1(D+r,k,\l,N)+\frac{2k}{k-1}(D+2r)\|F\|}e^{\frac{2}{k-1}d(p,x)\|F\|}\mu_{F,o}(\pX)\\
=: &L_2(r,k,\l,N)e^{\frac{2}{k-1}d(p,x)\|F\|}.   
 \end{aligned}
\end{equation}
By \eqref{e:alph1}, \eqref{e:alph2} and \eqref{up}, taking $\rho=e^{L_1}L_2$ gives \eqref{e:alph}.
\end{proof}
\section{Equilibrium states}

\subsection{Construction of invariant measures}
Let $P:SX \rightarrow \pX \times \pX$ be the projection given by $P(v)=(c_{v}(-\infty),c_{v}(+\infty))$. Denote by $\p^2X:=P(SX)$ the set of pairs $(\xi,\eta)\in \pX\times \pX$ which can be connected by a geodesic on $X$. Recall that $\iota: SM\to SM$ is the flip map, i.e., $\iota(v)=-v$.
Fix a reference point $o\in X$ as before. We can define a measure $\bar{\mu}_{F}$
on $\p^2X$ by the following formula:
$$d \bar \mu_{F}(\xi,\eta) = e^{C_{F\circ \iota-\d_F,\xi}(o, \pi(v))+C_{F-\d_F,\eta}(o, \pi(v))}d\mu_{F\circ \iota,o}(\xi) d\mu_{F,o}(\eta)$$
where $P(v)=(\xi,\eta)$. By Corollary \ref{cocexi2}, Lemma \ref{cocyle2} and Proposition \ref{full}, the Busemann cocycle $C_{F\circ \iota-\d_F,\xi}(o, \pi(v))$ and $C_{F-\d_F,\eta}(o, \pi(v))$ in the above definition are well defined almost everywhere.

\begin{lemma}\label{equi}
We have
\begin{enumerate}
  \item For any $p\in X$,
  $$d \bar \mu_{F}(\xi,\eta) = e^{C_{F\circ \iota-\d_F,\xi}(p, \pi(v))+C_{F-\d_F,\eta}(p, \pi(v))}d\mu_{F\circ \iota,p}(\xi) d\mu_{F,p}(\eta).$$
  \item $\bar\mu_{F}$ is invariant under the action of $\Gamma$.
  \item $\iota_* \bar\mu_F=\bar \mu_{F\circ\iota}$.
\end{enumerate}
\end{lemma}
\begin{proof}
Let $p\in X$. Then
\begin{equation*}
\begin{aligned}
&e^{C_{F\circ \iota-\d_F,\xi}(p, \pi(v))+C_{F-\d_F,\eta}(p, \pi(v))}d\mu_{F\circ \iota,p}(\xi) d\mu_{F,p}(\eta)\\
=&e^{C_{F\circ \iota-\d_F,\xi}(p, \pi(v))+C_{F-\d_F,\eta}(p, \pi(v))}e^{-C_{F\circ \iota-\d_F,\xi}(p, o)-C_{F-\d_F,\eta}(p, o)}d\mu_{F\circ \iota,o}(\xi) d\mu_{F,o}(\eta)\\
=&e^{C_{F\circ \iota-\d_F,\xi}(o, \pi(v))+C_{F-\d_F,\eta}(o, \pi(v))}d\mu_{F\circ \iota,o}(\xi) d\mu_{F,o}(\eta)\\
=&d \bar \mu_{F}(\xi,\eta).
\end{aligned}
\end{equation*}
This proves (1).

To prove that $\bar\mu_{F}$ is invariant under the action of $\c\in \Gamma$, it is enough to show
\begin{equation*}
\begin{aligned}
&e^{C_{F\circ \iota-\d_F,\c\xi}(o, \pi(\c v))+C_{F-\d_F,\c\eta}(o, \pi(\c v))}d\mu_{F\circ \iota,o}(\c \xi) d\mu_{F,o}(\c\eta)\\
=&e^{C_{F\circ \iota-\d_F,\xi}(\c^{-1}o, \pi(v))+C_{F-\d_F,\eta}(\c^{-1}o, \pi(v))}d\mu_{F\circ \iota,\c^{-1}o}(\xi) d\mu_{F,\c^{-1}o}(\eta)\\
=&e^{C_{F\circ \iota-\d_F,\xi}(o, \pi(v))+C_{F-\d_F,\eta}(o, \pi(v))}d\mu_{F\circ \iota,o}(\xi) d\mu_{F,o}(\eta).
\end{aligned}
\end{equation*}
The last equality follows from item (1). 

(3) follows directly from the definition.
\end{proof}

Denote $\b_F(\xi,\eta):= e^{C_{F\circ \iota-\d_F,\xi}(o, \pi(v))+C_{F-\d_F,\eta}(o, \pi(v))}$. It is not evident that $\b_F$ is bounded on $\p^2X$. So $\bar\mu_F$ may be not a Radon measure on $\p^2X$.

Let $a$ be a regular geodesic axis in $X$ such that $a(0)\in \F$. By Lemma \ref{key}, there exists a pair of neighborhoods $U,V$ of $a(-\infty)$ and $a(+\infty)$ respectively in $\pX$, such that for each pair $(\xi,\eta)\in U \times V$ there is a unique regular geodesic connecting $\xi$ and $\eta$. 
By Lemma \ref{supp}, $\mu_{F\circ \iota,o}(U)>0$ and  $\mu_{F,o}(V)>0$.
For every $Q>0$, denote by $\p^2X(Q)$ the set of pairs $(\xi,\eta)\in \p^2X$ such that $\b_F(\xi,\eta)\le Q$ and $c_{\xi\eta}$ passes through $\F$. 
%Take a regular geodesic axis $a$ with $\dot a(0)\in S\F$. Since $\dot a(0)\in \UR\cap \text{Reg}$, we know  $\b_F(a(-\infty),a(+\infty))$ is bounded by Lemma \ref{Bowen}.  
Then $0<\bar\mu_F((U\times V)\cap \p^2X(Q))<+\infty$ for $Q$ sufficiently large. The restrict $\bar\mu_F$ onto $\C\cdot((U\times V)\cap \p^2X(Q))$, denoted by
$$\bar\mu_{F,Q}:=\bar\mu_F|_{\C\cdot((U\times V)\cap \p^2X(Q))},$$
gives a Radon measure since $\mu_{F,o}$ and $\mu_{F\circ \iota,o}$ are probability measures on $\pX$ and $\b_F\le Q$. 
Then $\bar\mu_{F,Q}$ induces a $g^t$-invariant Radon measure $\mu_{F}$ on $SX$ with
\begin{equation*}
\nu_{F}(A)=\int_{\C\cdot((U\times V)\cap \p^2X(Q))} \text{Vol}(\pi(P^{-1}(\xi,\eta)\cap A))d  \bar \mu_{F}(\xi,\eta),
\end{equation*}
for all Borel sets $A\subset SX$. Here  $P^{-1}(\xi,\eta)$ is either a single geodesic on $X$ or a flat totally geodesic submanifold of $X$. In either case, Vol is the volume element on  $P^{-1}(\xi,\eta)$.

Note that $\nu_F$ is a $\C$-invariant Radon measure. We can project $\nu_{F}$ to get a Radon measure on $SM$, and then normalize it to a probability measure, still denoted by $\nu_{F}$. $\nu_F$ is a $g^t$-invariant probability measure on $SM$ (cf. \cite[Theorem 2.1]{Kai}).
%From Lemma \ref{equi}(1), the definition of $\bar\mu_F$ is independent of $p\in \pX$. Since $\mu_{F,o}$ is a probability measure on $\pX$, we usually use the original definition of $\bar\mu_F$.

By Lemma \ref{uniformrec}, $\nu_F(\UR)=1$. Define $\BBB^\l:=\{v^+: v\in S\F\cap \UR(\l)\}$. Then $P(\UR(\l))\subset \C\cdot (\BBB^\l\times \BBB^\l)$. Then if $\l_0$ is small enough, $\bar\mu_{F,Q}((U\times V)\cap (\BBB^{\l_0}\times \BBB^{\l_0}))>0$ by definition of $\nu_F$. Thus $\bar\mu_F((U\times V)\cap (\BBB^{\l_0}\times \BBB^{\l_0}))>0$ . Moreover, $\b_F$ is bounded on $\BBB^{\l_0}\times \BBB^{\l_0}$ by Lemma \ref{Bowen}. We further define
$$\bar\mu_{F,\l_0}:=\bar\mu_F|_{\C\cdot((U\times V)\cap (\BBB^{\l_0}\times \BBB^{\l_0})},$$
and a $g^t$-invariant and $\C$-invariant Radon measure on $SX$
\begin{equation*}
\mu_{F}(A)=\int_{\C\cdot((U\times V)\cap (\BBB^{\l_0}\times \BBB^{\l_0}))} \text{Vol}(\pi(P^{-1}(\xi,\eta)\cap A))d  \bar \mu_{F}(\xi,\eta).
\end{equation*}
So $\mu_{F}$ descends to $SM$ and can be normalized to a probability measure on $SM$. We finally obtain the $g^t$-invariant probability measure on $SM$ we want, still denoted by $\mu_F$.

\subsection{Ergodicity}
To prove the ergodicity of $\mu_F$, we follow the idea of \cite{Kn}, with significant modifications due to the definition of $\mu_F$. 

Let $a$ be a regular axis in $X$ and  $U,V$ neighborhoods of $a(-\infty)$ and $a(+\infty)$ respectively in $\pX$, as in the construction of measure $\mu_F$.  Consider the sets
\begin{equation*}
\begin{aligned}
\mathcal{G}(U,V):=&\{\text{geodesic\ } c: (c(-\infty), c(+\infty))\in (\BBB^{\l_0}\cap U)\times (\BBB^{\l_0}\cap V)\},\\
\mathcal{G}'(U,V):=&\{\dot c(t):~ c\in \mathcal{G}(U,V),~t\in\mathbb{R} \},\\
\mathcal{G}_{rec}(U,V):=&\{c : c\in \mathcal{G}(U,V),~\mbox{and} ~\dot{c}(0) ~\mbox{is recurrent}\},\\
\mathcal{G}'_{rec}(U,V):=&\{\dot c(t): c\in \mathcal{G}_{rec}(U,V),~t\in\mathbb{R} \}.
\end{aligned}
\end{equation*}
By construction, $\mu_F(\mathcal{G}'(U,V))>0$. By Poincar\'e recurrence theorem, we have $\mu_F(\mathcal{G}'(U,V)\setminus \mathcal{G}'_{rec}(U,V))=0$.

Let $f:SX \to \mathbb{R}$ be a continuous $\Gamma$-invariant function.
By Birkhoff ergodic theorem,
for $\mu_F \ae v\in SX$, the following two functions
$$f^{\pm}(v):= \lim_{T\rightarrow \pm\infty}\frac{1}{T}\int^{T}_{0}f(\dot c_v(t))dt$$
are well-defined and equal.
Obviously, $f^{\pm}$ are constants along each orbit of the geodesic flow.
So we can write $f^{\pm}(c):=f^{\pm}(\dot c(0))$ for every geodesic $c$.
Denote 
\begin{equation*}
\begin{aligned}
\mathcal{\widetilde{G}}_{rec}(U,V):=&\{c : c \in \mathcal{G}_{rec}(U,V), ~f^+(c)=f^-(c)\},\\
\mathcal{\widetilde{G}}'_{rec}(U,V):=&\{\dot c(t): c\in\mathcal{\widetilde{G}}_{rec}(U,V),~t\in\mathbb{R}\}.
\end{aligned}
\end{equation*}
Therefore we have $\mu_F(\mathcal{G}'(U,V)\setminus \mathcal{\widetilde{G}}'_{rec}(U,V))=0.$

The following lemma is based on a Fubini type argument.
\begin{lemma}\label{Gxi}
There exists a geodesic $c_1\in \mathcal{\widetilde{G}}_{rec}(U,V)$ with $c_1(-\infty)=\xi_1\in  \BBB^{\l_0}\cap U$, such that
$$G_{\xi_1}:= \{\eta \in V : \exists~c\in\mathcal{\widetilde{G}}_{rec}(U,V) ~\mbox{with}~c(-\infty)=\xi_1,  c(+\infty)=\eta\}$$
has full $\mu_{F,o}$-measure in $\BBB^{\l_0}\cap V$, i.e., $\mu_{F,o}(G_{\xi_1})=\mu_{F,o}(\BBB^{\l_0}\cap V)$.
\end{lemma}
\begin{proof}
Let $\mathcal{E}=\mathcal{G}_{rec}(U,V)\setminus\mathcal{\widetilde{G}}_{rec}(U,V),$ and $\mathcal{E}'=\{\dot c(t): c\in\mathcal{E},~t\in\mathbb{R}\}$. Then $\mu_F(\mathcal{E}')=0$ from the discussion above. Moreover, for any $\xi\in U$, let
$$G^{co}_{\xi}= \{\eta \in V :  \exists~c\in\mathcal{E}~\mbox{with}~ c(-\infty)=\xi, ~c(+\infty)=\eta\}.$$
Choosing $v\in P^{-1}(\xi,\eta)\subset \mathcal{E}'$, we have
\begin{equation*}
\begin{aligned}
\bar\mu_{F,\l_0}(\mathcal{E}')=\int_{P(\mathcal{E}')}e^{C_{F\circ \iota-\d_F,\xi}(o, \pi(v))+C_{F-\d_F,\eta}(o, \pi(v))}d\mu_{F\circ \iota,o}(\xi) d\mu_{F,o}(\eta)=0.
\end{aligned}
\end{equation*}
It follows that $\int_{\BBB^{\l_0}\cap U}(\int_{G^{co}_{\xi}}d\mu_{F,o}(\eta) )d\mu_{F,o}(\xi)=0$, which implies that for  $\mu_{F,o} \ae \xi \in \BBB^{\l_0}\cap U$,
$\mu_{F,o}(G^{co}_{\xi})=0$. Thus $\mu_{F,o}(G_{\xi})=\mu_{F,o}(\BBB^{\l_0}\cap V)$ for $\mu_{F,o}$-a.e.~$\xi \in \BBB^{\l_0}\cap U$. Pick such $\xi_1$ and corresponding $c_1$, then the lemma follows.
\end{proof}

The following lemma is based on a Hopf type argument.
\begin{lemma}\label{lem8}
$f^{+}(c)=f^+(c_1)$ for almost all $c\in\mathcal{\widetilde{G}}_{rec}(U,V)$.
\end{lemma}
\begin{proof}
By Lemma \ref{Gxi}, there exists a geodesic $c_1\in \mathcal{\widetilde{G}}_{rec}(U,V)$ with $c_1(-\infty)=\xi_1\in  \BBB^{\l_0}\cap U$ such that $\mu_{F,o}(G_{\xi_1})=\mu_{F,o}(\BBB^{\l_0}\cap V)$. Then almost every $c\in \mathcal{\widetilde{G}}_{rec}(U,V)$ satisfies that $c(+\infty)\in G_{\xi_1}$. By the definition of $G_{\xi_1}$, we know that there is a geodesic $c_2\in \mathcal{\widetilde{G}}_{rec}(U,V)$ with $c_2(-\infty)=c_1(-\infty)=\xi_1$ and $c_2(+\infty)=c(+\infty)$. Then by \cite[Proposition 4.1]{Kn}, after a reparameterization we have 
$$\lim_{t\to +\infty}d_K(\dot c_2(t),\dot c(t))= 0$$
which implies that $f^+(c)=f^+(c_2)~(=f^-(c_2))$. Similarly we have $f^-(c_2)=f^-(c_1)~(=f^+(c_1))$. Therefore we get $f^+(c)=f^+(c_1)$ for almost all $c\in\mathcal{\widetilde{G}}_{rec}(U,V)$.
\end{proof}

Now we are ready to prove the ergodicity of $\mu_F$.
\begin{theorem}\label{ergodic}
$\mu_F$ is ergodic.
\end{theorem}
\begin{proof}
It is sufficient to prove that for any continuous $\Gamma$-invariant function $f:SX \to \mathbb{R}$, the function $f^+$ is constant $\mu_F \ae$ on $SX$. Let 
$$\widetilde{V}:=\{\eta \in V : \exists~c \in \widetilde {\mathcal{G}}_{rec}(U,V)~\mbox{with}~\eta=c(+\infty)\}.$$
By Lemma \ref{lem8} we know $f^{+}(c)=f^+(c_1)$ for almost all geodesics $c$ with $c(+\infty)\in \widetilde{V}$. By a Fubini type argument as in the proof of Lemma \ref{Gxi}, we have that $\mu_{F,o}(\widetilde{V})=\mu_{F,o}(\BBB^{\l_0}\cap V)$.

Let $Y:=\C\cdot \widetilde{V}\subset \C\cdot (\BBB^{\l_0}\cap V)$. We just showed that $Y\cap \BBB^{\l_0}\cap V$ has full measure in $\BBB^{\l_0}\cap V$ with respect to $\mu_{F,o}$, and hence with respect to $\mu_{F,q}$ for any $q\in X$. Therefore, 
\begin{equation*}
\begin{aligned}
&\mu_{F,o}((\C\cdot (\BBB^{\l_0} \cap V))\setminus Y)\le \mu_{F,o}(\C\cdot ((\BBB^{\l_0}\cap V)\setminus \widetilde{V}))\\
\le &\sum_{\b\in \C}\mu_{F,o}(\b\cdot ((\BBB^{\l_0}\cap V)\setminus \widetilde{V}))=\sum_{\b\in \C}\mu_{F,\b^{-1} o}((\BBB^{\l_0}\cap V)\setminus \widetilde{V})=0.
\end{aligned}
\end{equation*}
It follows that $\mu_{F,o}(Y)=\mu_{F,o}(\C\cdot (\BBB^{\l_0} \cap V))$.
Then the set
$$Z := \{\dot c(t) : c(+\infty)\in Y, ~t\in\mathbb{R}\}$$
has full $\mu_{F}$-measure in $SX$. Since $f$ is a $\Gamma$-invariant function,
$f^{+}$ is also $\Gamma$-invariant. Therefore,
 $f^{+}\equiv f^+(c_1)$ $\mu_F \ae$on $Z$. This implies that $f^+$ is constant $\mu_F \ae$on $SX$.
 So $\mu_F$ is ergodic.
\end{proof}

\begin{corollary}\label{regularfull}
  $\mu_F(\text{Reg})=1$  
\end{corollary}
\begin{proof}
From the definition of $\mu_F$, $\text{Reg}$ has positive $\mu_F$-measure. Since $\text{Reg}$ is invariant under the geodesic flow, it has full $\mu_F$-measure by the ergodicity of $\mu_F$.
\end{proof}

%\begin{remark}
%Denote by $\mathcal{B}^{\lambda}=P(\UR(\lambda)\cap S\F)$ the subset of pairs in $\pX \times \pX$
%which can be connected by geodesics with initial unit vector in $\UR(\lambda)\cap S\F$. Note that the connecting geodesic must be unique since we consider regular geodesics.
%Since $(\xi,\eta)\in \mathcal{B}^{\lambda}$ for some $\l>0$, the geodesic connecting $\xi$ and $\eta$ is regular and hence unique. So for $\mu_F \ae (\xi,\eta,t)$,
%$$d \mu_{F}(\xi,\eta,t)=dtd \bar \mu_{F}(\xi,\eta).$$
%\end{remark}

\subsection{Proof of Theorem \ref{es}}

\begin{lemma}\label{uppergibbs}
Let $n\gg N$. For $\mu_{F} \ae v\in \L_{k,\lambda,N}$, $\e>0$ sufficiently small, there exists $L=L(\e, k, \l, N)>0$ such that
\begin{equation}\label{e:bowenball}
\mu_{F}(B_n(v,\e)\cap\L_{k,\lambda,N})\le Le^{\int_0^n(F(g^tv)- \d_F)dt+\frac{10k-4}{k(k-1)}n\|F\|}.
\end{equation}
\end{lemma}
\begin{proof}
Let $o\in \F\subset X$ be the reference point in the definition of $\mu_{F}$. We lift $v$ to $X$, still denoted by $v$ such that $\pi v\in \F$. So $d(o,\pi v)\le D$.
For $\e\ll \inj(M)$, we can lift $B_n(v,\e)$ to $X$, such that for any lifted $w\in B_n(v,\e)$, we have $d_K(g^tv,g^tw)\le \e$ for any $0\le t\le n$. Moreover, we can assume that $\pi(B_n(v,\e))\subset \F$. Pick any $w\in B_n(v,\e) \cap \L_{k,\lambda,N}$.
Denote $x:=c_{v}(n)$ and $\eta:= c_w(-\infty)$. Let $c_{\eta,x}$ be the geodesic connecting $\eta$ and $x$ with $c_{\eta,x}(n)=x$.
Then
$$d(o,c_{\eta,x}(0))\le d(o, \pi w)+d(\pi w, c_{\eta,x}(0))\le D+\e.$$
Thus $\eta\in pr_x(B(o,D+\e))$. Therefore,
$$P(B_n(v,\e)\cap\L_{k,\lambda,N})\subset \bigcup_{\eta\in pr_x(B(o,D+\e))}\{\eta\}\times (pr_\eta(B(x,\e))\cap G_\eta( k, \l, N))$$
where $G_\eta( k, \l, N)$ is defined as in Proposition \ref{alph}.

Note that $pr_\eta(B(x,\e)) \subset pr_\eta(B(c_w(n),2\e))$. Then
\begin{equation*}
\begin{aligned}
&\mu_{F,o}(pr_\eta(B(c_w(n),2\e))\cap G_\eta(k, \l, N))\\
\le &L_1\mu_{F,\pi w}(pr_\eta(B(c_w(n),2\e))\cap G_\eta(k,\l, N))\\
\le &L_1\rho e^{\int_{0}^n (F(g^t w)-\d_F)dt+\frac{8k-2}{k(k-1)}n\|F\|} \\
\le &L_1\rho L_2 e^{\int_{0}^n (F(g^t v)-\d_F)dt+(\frac{8k-2}{k(k-1)}+\frac{2}{k})n\|F\|}
\end{aligned}
\end{equation*}
where $\rho=\rho(\e, k, \lambda, N)$ is from Proposition \ref{alph}, and $L_1=L_1(D, \e, \l)$, $L_2=L_2(\e,\lambda)$. In the last inequality above, we used a similar argument as in the proof of Proposition \ref{alph} based on the fact that $v,w\in \L_{k,\l, N}$.

On the other hand, we have
$$\mu_{F,o}(pr_x(B(o,D+\e)))\le \mu_{F,o}(\pX)=1.$$
Note that the diameter of $B(\pi v, \e)$ is no more than $2\e$. By the definition of $\mu_F$, we obtain \eqref{e:bowenball} with $L:=2\e L_1\rho L_2$ up to a normalization constant.
\end{proof}

\begin{lemma}(Cf. \cite[Theorem I.I]{Kat})\label{Katok}
 Let $f:Y\to Y$ be a homeomorphism on a compact metric space $(Y,d)$, and $\mu$ an ergodic measure. Then for any $0<\rho<1$,
 \[h_\mu(f)=\lim_{\e\to 0}\liminf_{n\to \infty}\frac{1}{n}\log N(n,\e,\rho)\]
 where $N(n,\e,\rho)$ denotes the minimal number of $(n,\e)$-Bowen balls which cover a set of measure more than $1-\rho$.
\end{lemma}

We reformulate and prove Theorem \ref{es} as follows.
\begin{proposition}
$\mu_F$ is an equilibrium state, that is,
\begin{equation*}
h_{\mu_{F}}(g^1)+\int Fd\mu_{F} =\d_F.
\end{equation*}
\end{proposition}
\begin{proof}
Since $\mu_F$ is ergodic, by Birkhoff ergodic theorem, there exists a subset $W\subset SM$ of full $\mu_F$-measure, such that for any $w\in W$
\[\lim_{n\to \infty}\frac{1}{n}\int_0^nF(g^tw)dt=\int Fd\mu_F.\]
Let $\kappa>0$ be small. Denote $W_K\subset W$ the set of $w\in W$ such that
\begin{equation}\label{e:birkhoff}
\left|\frac{1}{n}\int_0^nF(g^tw)dt-\int Fd\mu_F\right|<\kappa, \quad \forall n\ge K.
\end{equation}
Then $W=\cup_{K=1}^\infty W_K$.

By Corollary \ref{regularfull}, $\mu_F(\URR)=\mu_{F\circ \iota}(\URR)=1$.
For any $0<\rho<1$ and any $k\in \NN$, pick $\l>0$ such that $\mu_F(\UR(\l))>\max\{1-\rho, 1-\frac{1}{k}\}$. By Lemma \ref{equi}(3), $\mu_F(\iota(\UR(\l)))=\mu_{F\circ\iota}((\UR(\l)))$. By shrinking $\l$ if necessary, we can assume $\mu_F(\iota(\UR(\l)))>\max\{1-\rho, 1-\frac{1}{k}\}$. Thus by Birkhorff ergodic theorem, if $N$ is large enough, we have $\mu_F(\L_{k,\l,N})>1-\rho$. Pick $K$ large enough  such that $\mu_F(\L_{k,\l,N}\cap W_K)>1-\rho$.  

Let $n\gg N$ and $\{B_n(v_i,\e)\}_{i=1}^S$ be a minimal set of $(n,\e)$-Bowen balls which cover $\L_{k,\l,N}\cap W_K$. Then $B_n(v_i,\e)\cap \L_{k,\l,N}\cap W_K\neq \emptyset$ for each $1\le i\le S$. Pick $w_i\in B_n(v_i,\e)\cap \L_{k,\l,N}\cap W_K$. Then $B_n(v_i,\e)\subset B_n(w_i,2\e)$. By Lemma \ref{uppergibbs} and \eqref{e:birkhoff}, 
\begin{equation*}
\begin{aligned}
&\mu_F(B_n(v_i,\e)\cap \L_{k,\l,N}\cap W_K)\le \mu_F(B_n(w_i,2\e)\cap \L_{k,\l,N}\cap W_K)\\
\le &Le^{\int_0^n(F(g^tw_i)- \d_F)dt+\frac{10k-4}{k(k-1)}n\|F\|}\le Le^{n(\int Fd\mu_F+\kappa-\d_F)+\frac{10k-4}{k(k-1)}n\|F\|}.
\end{aligned}
\end{equation*}
Thus $S>(1-\rho)L^{-1}e^{-n(\int Fd\mu_F+\kappa-\d_F)-\frac{10k-4}{k(k-1)}n\|F\|}$.
Then by Katok's entropy formula Lemma \ref{Katok}
\begin{equation*}
\begin{aligned}
h_{\mu_{F}}(g^1)&\ge \lim_{\e\to 0}\liminf_{n\to \infty}\frac{1}{n}\log \left((1-\rho)L^{-1}e^{-n(\int Fd\mu_F+\kappa-\d_F)-\frac{10k-4}{k(k-1)}n\|F\|}\right)\\
&=-\int Fd\mu_{F}-\kappa+\d_F-\frac{10k-4}{k(k-1)}\|F\|.
\end{aligned}
\end{equation*}
Since $\kappa$ and $k$ are arbitrary, we have $h_{\mu_{F}}(g^1)+\int Fd\mu_{F}\ge \d_F=P(F)$.
By variational principle, $h_{\mu_{F}}(g^1)+\int Fd\mu_{F}\le P(F)$, thus the proposition follows.
\end{proof}

\section{Bernoulli property of equilibrium states}
In this section, we provide a proof of Theorem \ref{bernoulli}, that is, the unique equilibrium state $\mu_F$ is Bernoulli. This will be done based on the classic argument showing that the Kolmogorov property implies Bernoulli property for smooth invariant measures of hyperbolic systems. In \cite{CT}, Call and Thompson showed that $\mu_F$ has the Kolmogorov property. They also showed the unique MME is Bernoulli by utilizing the Patterson-Sullivan construction of the MME. The key progress here is that the above Pattterson-Sullivan construction of $\mu_F$ provides the local product structure we need.

The argument was carried out by Chernov and Haskell \cite{CH} for smooth invariant measures of suspension flows over some nonuniformly hyperbolic maps with sigularities. The argument is also true for hyperbolic invariant measures with local product structure. We follow the lines in \cite[Section 7]{CT}.

In \cite{CH}, if an invariant measure $\mu$ has the Kolmogorov property and there exists an $\e$-regular covering
with non-atomic conditionals for $\mu$ for any $\e > 0$, then any finite partition $\xi$ of the
phase space with piecewise smooth boundary and a constant $C > 0$ such that
$\mu(B(\partial \xi, \d)) \le C\d$ for all $\d > 0$ is \emph{Very Weak Bernoulli}. A refining sequence of such
partitions with diameter going to $0$ suffices to conclude the Bernoulli property for
$\mu$. Such a sequence of partitions exists in this setting by \cite[Lemma 4.1]{OW}. 

Thus in our case, to conclude that $\mu_F$ is Bernoulli, we only need to show that $\e$-regular coverings for $\mu_F$ exist for all $\e > 0$. First, let us give Chernov and Haskell's definition of $\e$-regular covering.

\begin{definition}\label{recdef}
A \emph{rectangle} in $SM$ is a measurable set $R\subset SM$, equipped with a distinguished point $z \in R$ with the property that for all points $x, y \in R$ the local weak stable manifold $W_\loc^{cs}(x)$ and the local unstable $W_\loc^{u}(y)$ intersect each other at a single point, denoted by $[x,y]$, which lies in $R$.
\end{definition}
Notice that a rectangle $R$ can be thought of as the Cartesian product of $W_\loc^{cs}(z)\cap R$ and $W_\loc^{u}(z)\cap R$. Given a probability measure $\mu$ on $SM$, there is a natural product measure
$$\mu^p_R:=\mu^u_z\times \tilde\mu^{cs}_z,$$
where $\mu^u_z$ is the conditional measure induced by $\mu$ on $W_\loc^u(z)\cap R$ with respect to the measurable partition of $R$ into local unstable manifolds, and $\tilde\mu^{cs}_z$ is the corresponding factor measure on $W^{cs}_\loc(z)$.

\begin{definition} \label{coveringdef}
Given any $\e> 0$, we define an $\e$-regular covering for $\mu$ of the
phase space $SM$ to be a finite collection $\mathfrak{R}=\mathfrak{R}_\e$ of disjoint rectangles such that
\begin{enumerate}
  \item $\mu(\cup_{R\in \mathfrak{R}}R) > 1-\e$;
  \item Given any two points $x, y\in R \in \mathfrak{R}$, which lie in the same unstable or weakly
stable manifold, there is a smooth curve on that manifold which connects $x$
and $y$ and has length less than $100\diam R$;
  \item For every $R\in \mathfrak{R}$, with distinguished point $z\in R$, the product measure $\mu^p_R:=\mu^u_z\times \tilde\mu^{cs}_z$ satisfies
$$\left|\frac{\mu^p_R(R)}{\mu(R)}-1\right| <\e.$$
Moreover, $R$ contains a subset $G$ with $\mu(G) > (1-\e)\mu(R)$ such that for all $x\in G$,
$$\left|\frac{d\mu^p_R}{d\mu}(x)-1\right| <\e.$$
\end{enumerate}
\end{definition}

As discussed above, to finish the proof of Theorem \ref{bernoulli}, it remains to prove the following lemma.
\begin{lemma}\label{covering}
For any $\d>0$ and $\e>0$, there exists an $\e$-regular covering $\mathcal{R}_{\e}$ of connected rectangles of $SM$ for $\mu_F$, with $\diam(R)<\d$ for any $R\in \mathcal{R}_{\e}$.
\end{lemma}

\begin{proof}
$\mu_F$ is a hyperbolic measure by Theorem \ref{bcft}. By \cite[Lemma 8.3]{Pe1}, \cite[Lemma 1.8]{Pe2} and \cite[Lemma 9.5.7]{BP}, for a hyperbolic measure $\mu_F$, we can find a finite
collection of disjoint rectangles $R$ covering a Pesin set for $\mu_F$. By choosing a Pesin set with $\mu_F$-measure at least $1-\e$, we have Definition \ref{coveringdef}(1). Moreover, the rectangles $R$ can be chosen such that $\diam R<\d$. Since the metrics on local unstable or local weak stable manifold are uniformly equivalent to the Riemannian distance if $\d$ is small enough, Definition \ref{coveringdef}(2) is also satisfied.

To verify Definition \ref{coveringdef}(3), we use the local product structure of $\mu_F$:
\begin{equation*}
\mu_{F}(A)=\int_{\C\cdot((U\times V)\cap (\BBB^{\l_0}\times \BBB^{\l_0}))} \text{Vol}(\pi(P^{-1}(\xi,\eta)\cap A)) \b_F(\xi,\eta)d\mu_{F\circ \iota,o}(\xi) d\mu_{F,o}(\eta),
\end{equation*}
where the density function $\b_F(\xi,\eta):=e^{C_{F\circ \iota-\d_F,\xi}(o, \pi(v))+C_{F-\d_F,\eta}(o, \pi(v))}$ and $P(v)=(\xi,\eta)\in \C\cdot((U\times V)\cap (\BBB^{\l_0}\times \BBB^{\l_0}))$.
If $v\in \text{Reg}$, $P^{-1}(v^-, v^+)$ consists of a single geodesic $c_v$, and thus $\text{Vol}$ just becomes the length along the geodesic $c_v$. Since $\mu_F(\text{Reg}) = 1$, we in fact have for $\mu_F \ae v$,
$$d\mu_F(v)=\b_F(\xi, \eta)d\mu_{F,o} (\xi)d\mu_{F,o} (\eta)dt.$$

Let $R$ be a rectangle of sufficiently small diameter constructed above. We lift all objects to the universal cover $X$. Take a lift of $R$, which is still denoted by $R$. Since the local weak stable
and local unstable manifolds at $v$ intersect transversely if and only $v\in \text{Reg}$, it follows
that $R\subset \text{Reg}$. Notice that for $x\in \text{Reg}$, there exists a continuous map $\phi_x: W_\loc^u(x)\to \pX$ given by $\phi_x(v)=v^+$ where $v\in W_\loc^u(x)$. The conditional measure $\mu^u_x$ on $R\cap W_\loc^u(x)$ is given by
\begin{equation}\label{conditional}
d\mu^u_x(v) = \frac{\b_F(x^-,\phi_xv) d\mu_{F,o}(\phi_xv)}{\int_{R\cap W_\loc^u(x)\cap \phi_x^{-1}(\C\cdot(\BBB^{\l_0}\cap V))}\b_F(x^-,\phi_xv) d\mu_{F,o} (\phi_xv)}, 
\end{equation}
and $$\int d \mu_F(v)=c\int_{(\C\cdot(\BBB^{\l_0}\cap U))\times \RR}\Big(\int_{R\cap W_\loc^u(x)\cap \phi_x^{-1}(\C\cdot(\BBB^{\l_0}\cap V))} d\mu^u_x(v)\Big) d\mu_{F,o} (x^-)dt$$
for some normalization constant $c>0$. 

Here we remark that both the conditional measure $\mu_x^u$ and the factor measure $d\mu_{F,o} (x^-)dt$ have no atom. Indeed, let $\xi$ be an increasing measurable partition of $SM$ subordinate to Pesin local unstable manifolds with respect to $\mu_F$ (cf. \cite{LY2}),
and $\{\mu_x^\xi\}$ the conditional measures of $\mu_F$ with respect to $\xi$ which is equivalent to $\{\mu_x^u\}$ above. If $\mu_x^\xi$ has an atom at $x$, then
\begin{equation*}
\begin{aligned}
h_{\mu_F}(g^1)=&\lim_{n\to \infty}\frac{1}{n}H_{\mu_F}(g^{-n}\xi|\xi)=\lim_{n\to \infty}-\frac{1}{n}\log \mu_x^\xi((g^{-n}\xi)(x))\\
\le &\lim_{n\to \infty}-\frac{1}{n}\log \mu_x^\xi(\{x\}) =0.
\end{aligned}
\end{equation*}
Reversely, $h_{\mu_F}(g^1)=0$ implies that $\mu_x^s$ has an atom. So $\mu_F$ has an atom by the local product structure. But $\mu_F$ is ergodic with full support, a contraction. 
Thus $\mu_x^u$ has no atom. Analogously, $\mu_x^s$ and hence the factor measure $d\mu_{F,o} (x^-)dt$ cannot have atoms.

Given two points $x, y\in R$, the local weak stable holonomy map $\pi^{cs}_{xy}: W_\loc^{u}(x)\cap R \to W_\loc^{u}(y)\cap R$ is defined by
$$\pi_{xy}^{cs}(w)\in W_\loc^{u}(y)\cap W_\loc^{cs}(w), \quad w\in W_\loc^{u}(x).$$
Note that $\phi_x(w)=\phi_{y}(\pi^{cs}_{xy}w):=\eta_w$. By \eqref{conditional}, the Jacobian of the holonomy map is 
\begin{equation*} 
\begin{aligned}
\left|\frac{d(\pi_{yx}^{cs})_*\mu^u_y}{d\mu^u_x}(w)\right|=\left|\frac{\b_F(y^-,\eta_w)}{\b_F(x^-,\eta_w)}\cdot\frac{\int_{\phi_x(R\cap W_\loc^u(x))\cap \C\cdot(\BBB^{\l_0}\cap V)}\b_F(x^-,\eta_w) d\mu_{F,o} (\eta_w)}{\int_{\phi_x(R\cap W_\loc^u(x))\cap \C\cdot(\BBB^{\l_0}\cap V)}\b_F(y^-,\eta_w) d\mu_{F,o} (\eta_w)}\right|.
\end{aligned}
\end{equation*}
By taking $\text{diam\ }R<\d$ small enough, $x^-,y^-$ are close in $\pX$. We want to prove in this case, $\b_F(y^-,\eta_w)$ and $\b_F(x^-,\eta_w)$ are uniformly close. As $\b_F$ is $\C$-invariant, without loss of generality, we assume that $(x^-,\eta_w)\in (U\times V)\cap (\BBB^{\l_0}\times \BBB^{\l_0})$ and $(y^-,\eta_w)\in (U\times V)\cap (\BBB^{\l_0}\times \BBB^{\l_0})$. Take $p_1\in \F\cap c_{x^-\eta_w}$ and $p_2\in \F\cap c_{y^-\eta_w}$. Let $0<\rho\ll \e$. Then by the fact that $x^-,y^-\in \BBB^{\l_0}$ and Lemma \ref{Bowen}, there exists $T>0$ large enough such that
\begin{equation*}
\begin{aligned}
\left|\int_{c_{o,x^-(T)}}^{x^-}(F\circ\iota-\d_F)-\int_{c_{p_1,x^-(T+s_1)}}^{x^-}(F\circ\iota-\d_F)\right|&<\rho,\\
\left|\int_{c_{o,y^-(T)}}^{y^-}(F\circ\iota-\d_F)-\int_{c_{p_2,y^-(T+s_2)}}^{y^-}(F\circ\iota-\d_F)\right|&<\rho
\end{aligned}
\end{equation*}
where $s_1=b_{x^-}(p_1,o)$ and $s_2=b_{x^-}(p_2,o)$. If $y^-$ is close enough to $x^-$, one can pick $p_1$ and $p_2$ close enough. Since $T$ is fixed, one has
\begin{equation*}
\begin{aligned}
\left|\int_{o}^{c_{o,x^-(T)}}(F\circ\iota-\d_F)-\int_{o}^{c_{o,y^-(T)}}(F\circ\iota-\d_F)\right|&<\rho,\\
\left|\int_{p_1}^{c_{p_1,x^-(T+s_1)}}(F\circ\iota-\d_F)-\int_{p_2}^{c_{p_2,y^-(T+s_2)}}(F\circ\iota-\d_F)\right|&<\rho.
\end{aligned}
\end{equation*}
Combining the above four inequalities, we get 
$$|C_{F\circ \iota-\d_F,x^-}(o, p_1)-C_{F\circ \iota-\d_F,y^-}(o, p_2)|<4\rho.$$ 
Similarly we can prove if $x^-,y^-$ are close enough, 
$$|C_{F-\d_F,\eta_w}(o, p_1)-C_{F-\d_F,\eta_w}(o, p_2)|<4\rho.$$
 Therefore, $\frac{\b_F(y^-,\eta_w)}{\b_F(x^-,\eta_w)}<e^{8\rho}$ and hence
$$\left|\frac{d(\pi_{yx}^{cs})_*\mu^u_y}{d\mu^u_x}(w)-1\right| <\e$$ 
for $x,y\in R$.
By definition of conditional measures this implies Definition \ref{coveringdef}(3).
\end{proof}

\section{Equidistribution and counting}
\subsection{Geometric flow box}

From now on we fix $v_0\in \L_{k,\l,N}\cap \iota(\L_{k,\l,N})\subset SX$ for some $\l>0$ and $k,N\in \NN$. Let $o:=\pi(v_0)$, which will be the reference point in the Patterson-Sullivan construction as well as in the following discussions. We also fix a scale $0<\e\ll \min\{\frac{1}{16}, \frac{\inj (M)}{4}\}$ until the end of this section.

The \emph{Hopf map} $H: SX\to \pX\times \pX\times \RR$ relative to $o\in X$ is defined as
$$H(v):=(v^-, v^+, s(v)), \text{\ where\ }s(v):=b_{v^-}(\pi v, o).$$
It is clear that $s$ is continuous. Moreover, $s(g^t v)=s(v)+t$ for any $v\in SX$ and $t\in \RR$. 

Using Hopf map, for each $\theta>0$ and $0<\a\le \frac{3}{2}\e$, we can define
\begin{equation*}
\begin{aligned}
&\bP=\bP_\theta:=\{w^-: w\in S_oX \text{\ and\ }\angle_o(w, v_0)\le \theta\},\\
&\bF=\bF_\theta:=\{w^+: w\in S_oX \text{\ and\ }\angle_o(w, v_0)\le \theta\},\\
&B^\a=B_\theta^\a:=H^{-1}(\bP\times \bF\times [0,\a]),\\
&S=S_\theta:=B_\theta^{\e^2}=H^{-1}(\bP\times \bF\times [0,\e^2]).
\end{aligned}
\end{equation*}
$B^\a=B_\theta^\a$ is called a flow \emph{box} with depth $\a$, and $S=S_\theta$ is a \emph{slice} with depth $\e^2$. $\theta>0$ is small enough as specified below, and is usually dropped from the notation.

The following lemma is a variation of Lemma \ref{key}, see also \cite[Lemma 1]{Ri}.
\begin{proposition}\label{crucial} 
Let $X$ be a simply connected rank one manifold of nonpositive curvature and $v_0\in SX$ is regular. Then for any $\e>0$, there is a $\theta_1 > 0$ such that for any $\xi \in \bP_{\theta_1}$ and $\eta \in \bF_{\theta_1}$,
there is a unique geodesic $c_{\xi,\eta}$ connecting $\xi$ and $\eta$, i.e.,
$c_{\xi,\eta}(-\infty)=\xi$ and $c_{\xi,\eta}(+\infty)=\eta$.

Moreover, the geodesic $c_{\xi,\eta}$ is regular and $d(\dot c_{v}(0),\dot c_{\xi,\eta})<\epsilon/10$.
\end{proposition}

Based on Proposition \ref{crucial}, we have the following result.

\begin{lemma}\label{diameter}(\cite[Lemma 2.13]{Wu2})
Let $v_0, o, \e$ be as above and $\theta_1$ be given in Proposition \ref{crucial}. Then for any $0<\theta<\theta_1$,
\begin{enumerate}
  \item $\text{diam\ } \pi H^{-1}(\bP\times \bF\times \{0\})<\frac{\e}{2}$;
  \item $H^{-1}(\bP\times \bF\times \{0\})\subset SX$ is compact;
  \item $B_\theta^\a$ is compact;
  \item $\text{diam\ } \pi B_\theta^\a <4\e$ for any $0<\a\le \frac{3\e}{2}$.
\end{enumerate}
\end{lemma}

\begin{corollary}\label{equicon}(\cite[Lemma 2.14]{Wu2})
Given $v_0, o, \e>0$ as above, there exists $\theta_2>0$ such that for any $0<\theta<\theta_2$, if $\xi,\eta\in \bP_\theta$ and any $q$ lying within $4\e$ of $\pi H^{-1}(\bP_\theta\times \bF_\theta\times [0,\infty))$, we have $|b_\xi(q,o)-b_\eta(q,o)|<\e^2$. Similar result holds if the roles of $\bP_\theta$ and $\bF_\theta$ are reversed.
\end{corollary}

We always consider $0<\theta< \min\{\theta_1,\theta_2\}$ with the following properties. 
\begin{enumerate}
    \item As $\theta\mapsto \bar\mu_{F,\l_0}(\bP_\theta\times \bF_\theta)$ is nondecreasing, hence it has at most countably many discontinuities. Choose $\theta$ to be the continuity point of this function, i.e., 
    \begin{equation}\label{e:choice0}
\begin{aligned}
\lim_{\rho\to \theta}\bar\mu_{F,\l_0}(\bP_\rho\times \bF_\rho)=\bar\mu_{F,\l_0}(\bP_\theta\times \bF_\theta).
\end{aligned}
\end{equation}
\item Choose $\theta$ such that $\bar\mu_{F,\l_0}(\partial\bP_\theta\times \partial\bF_\theta)=0.$ By the product structure of $B^\a$ and $S$ and the definition of $\mu_F$, we have for any $\a\in (0,\frac{3\e}{2})$,
\begin{equation}\label{e:choice}
\begin{aligned}
\lim_{\rho\to \theta} \mu_F(S_\rho)=\mu_F(S_\theta), \quad \lim_{\rho\to \theta} \mu_F(B^\a_\rho)=\mu_F(B^\a_\theta), \quad \mu_F(\partial B^\a_\theta)=0.
\end{aligned}
\end{equation}
\end{enumerate}

Finally, given $\rho_0>0$ and $k\in \NN$, we choose $\l>0$ and $N\in \NN$ such that $\mu_F(\L)>\frac{1}{1+\rho_0}$ where 
\begin{equation*}
\begin{aligned}
\L:=\L_{k,\l,N}\cap \iota(\L_{k,\l,N}).
\end{aligned}
\end{equation*}
We choose $v_0$ as a Lebesgue density point of $\L$ with respect to $\mu_F$ and assume that
\begin{equation}\label{density}
\begin{aligned}
\mu_F(B^\e_{\theta}\cap \L)>\frac{1}{1+\rho_0}\mu_F(B^\e_{\theta}).
\end{aligned}
\end{equation}
We will discuss this assumption at the end of this section.

\subsection{Counting intersection components}
For convenience, we collect some important subsets of $\C$ concerning the intersections of flow boxes.
\begin{equation*}
\begin{aligned}
\C(t,\a)=\C_\theta(t,\a)&:=\{\c\in \C: S_\theta\cap g^{-t}\c B_\theta^\a\neq \emptyset\},\\
\C^*=\C^*_\theta&:=\{\c\in \C: \c\bF_\theta \subset \bF_\theta \text{\ and\ }\c^{-1}\bP_\theta\subset \bP_\theta\},\\
\C^*(t,\a)&:=\C^*\cap \C(t,\a),\\
\C'(t,\a)&:=\{\c\in \C^*(t,\a): \c\neq \b^n \text{\ for any\ } \b\in \C, n\ge 2\}.
\end{aligned}
\end{equation*}
Given $\xi\in \pX$ and $\c\in \C$, define $b_\xi^\c:=b_\xi(\c o, o)$.
\begin{lemma}\label{intersection2}(\cite[Lemmas 4.2-4.5]{Wu2})
We have
\begin{enumerate}
    \item Let $c$ be an axis of $\c\in \C$ and $\xi=c(-\infty)$. Then $b_\xi^\c=|\c|$.
    \item Given any $\c\in \C$ and any $t\in \RR$, we have
$$S\cap g^{-t}\c B^\a=\{w\in P^{-1}(\bP\times \c\bF): s(w)\in [0,\e^2]\cap (b_{w^{-}}^\c-t+[0,\a])\}.$$
    \item If $\c\in \C^*(t,\a)$, then $|\c|\in [t-\a-\e^2, t+2\e^2]$.
    \item If $\c\in \C^*(t,\a)$, then
$$S^\c:=H^{-1}(\bP\times \c\bF\times [0,\e^2])\subset S\cap g^{-(t+2\e^2)}\c B^{\a+4\e^2}.$$
\end{enumerate}
\end{lemma}

Now we consider the following subsets relative to $\L$. Denote $\tilde B_{\theta,\L}^\a:=P(B_{\theta}^\a\cap \L)\times [0,\a]$, which is a saturation of $B_{\theta}^\a\cap \L$ along the flow direction inside $B_{\theta}^\a$.
\begin{equation*}
\begin{aligned}
\C_\L(t,\a)=\C_{\theta,\L}(t,\a)&:=\{\c\in \C: \tilde S_{\theta,\L}\cap g^{-t}\c (\tilde B_{\theta,\L})\neq \emptyset\},\\
\C^*=\C^*_\theta&:=\{\c\in \C: \c\bF_\theta \subset \bF_\theta \text{\ and\ }\c^{-1}\bP_\theta\subset \bP_\theta\},\\
\C_\L^*(t,\a)&:=\C^*\cap \C_\L(t,\a),\\
\C'_\L(t,\a)&:=\{\c\in \C_\L^*(t,\a): \c\neq \b^n \text{\ for any\ } \b\in \C, n\ge 2\}.
\end{aligned}
\end{equation*}

\begin{lemma}(Closing lemma)\label{closing}
For every $0<\rho<\theta<\theta_1$, there exists some $t_0> 0$ such that for all $t\ge t_0$, we have $\C_{\rho,\L}(t,\a)\subset \C_{\theta,\L}^*(t,\a)$.
\end{lemma}
\begin{proof}
By \cite[Lemma 3.11]{Wu2}, for every $0<\rho<\theta<\theta_1$, there exists some $t_0> 0$ such that for all $t\ge t_0$, we have $\C_\rho(t,\a)\subset \C_\theta^*$. 
It is clear from definition that $\C_{\rho,\L}(t,\a)\subset \C_\rho(t,\a)$ and $\C_{\rho,\L}(t,\a)\subset \C_{\theta,\L}(t,\a)$. Thus
\[\C_{\rho,\L}(t,\a)\subset \C_\rho(t,\a)\cap  \C_{\theta,\L}(t,\a)\subset \C_\theta^*\cap  \C_{\theta,\L}(t,\a)=\C_{\theta,\L}^*(t,\a).\]
The lemma follows.
\end{proof}

Define $\bP_{\L}:=\{w^-: w\in B^{3\e/2}\cap \L\}, \bF_{\L}:=\{w^+: w\in B^{3\e/2}\cap \L\}$ and
$$S_\L:=H^{-1}(\bP_\L\times \bF_\L\times [0,\e^2]), \quad B^\a_\L:=H^{-1}(\bP_\L\times \bF_\L\times [0,\a]).$$
$S_\L$ and $B_\L^\a$ can also be called \emph{rectangles}, which have local product structure in the sense of Definition \ref{recdef}. Obviously, $B^\a\cap \L\subset \tilde B^\a_\L\subset B^\a_\L$.
By \eqref{density}, 
\[\mu_F(B^\e_\L)\ge \mu_F(B^\e\cap \L)>\frac{1}{1+\rho_0}\mu_F(B^\e).\]
By definition of $\mu_F$ and $B^\e_\L$, we have $\bar{\mu}_{F,\l_0}(P(B^\e_\L))>\frac{1}{1+\rho_0}\bar\mu_{F,\l_0}(P(B^\e))$, i.e.,  $\bar{\mu}_{F,\l_0}(\bP_\L\times \bF_\L)>\frac{1}{1+\rho_0}\bar\mu_{F,\l_0}(\bP\times \bF)$. It follows that for $\forall \a\in (0,\frac{3}{2}\e]$,
\begin{equation}\label{density2}
\begin{aligned}
\mu_F(B^\a_\L) >\frac{1}{1+\rho_0}\mu_F(B^\a).
\end{aligned}
\end{equation}
In particular, $\mu_F(S_\L)>\frac{1}{1+\rho_0}\mu_F(S)$.

Given $\c\in \C_\L^*(t,\a)$, define $S_\L^\c:=H^{-1}(\bP_\L\times \c\bF_\L\times [0,\e^2])$.
\begin{lemma}\label{depth2}
Given any $\c\in \C$ and any $t\in \RR$, we have
$$S_\L\cap g^{-t}\c B^\a_\L\subset S_\L^\c.$$
\end{lemma}
\begin{proof}
By Lemma \ref{intersection2}(2), it is sufficient to prove 
$$S_\L\cap g^{-t}\c B^\a_\L\subset P^{-1}(\bP_\L\times \c\bF_\L).$$
Let $w\in S_\L\cap g^{-t}\c B^\a_\L$. Then $w^-\in \bP_\L$. Since $g^t\c^{-1}w\in g^{t}\c^{-1}_*S_{\L}\cap  B_{\L}^\a$, $w^+\in \c \bF_\L$. The lemma follows.
\end{proof}

\begin{lemma}\label{depth1}
If $\c\in \C_\L^*(t,\a)$, then 
$$S_\L^\c\subset S\cap g^{-(t+2\e^2)}\c B^{\a+4\e^2}$$
%where $\tilde\L:=\{v\in \R(\l/2): g^t(v)\in \R(\l/2), \iota(g^tv)\in \R(\l/2)\}.$
\end{lemma}
\begin{proof}
Since $S_\L^\c \subset S^\c$, the lemma follows from Lemma \ref{intersection2}(4).
\end{proof}

The following notations in the asymptotic estimates are standard.
\begin{equation*}
\begin{aligned}
f(t)=a^{\pm C}g(t)&\Leftrightarrow a^{-C}g(t)\le f(t)\le a^{C}g(t) \text{\ for all\ } t;\\
f(t) \lesssim g(t) &\Leftrightarrow \limsup_{t\to \infty}\frac{f(t)}{g(t)}\le 1;\\
f(t) \gtrsim g(t) &\Leftrightarrow \liminf_{t\to \infty}\frac{f(t)}{g(t)}\ge 1;\\
f(t) \sim g(t) &\Leftrightarrow \lim_{t\to \infty}\frac{f(t)}{g(t)}= 1;\\
f(t)\sim a^{\pm C}g(t)&\Leftrightarrow a^{-C}g(t)\lesssim f(t)\lesssim a^{C}g(t).
\end{aligned}
\end{equation*}

\begin{lemma}\label{scaling1}
Given $0<\a\le \frac{3}{2}\e$ and $\c\in \C^*_{\L}(t,\a)$, then
$$\mu_F(S_\L^\c)=e^{\pm L(\e,\l)}e^{\Per_F(\c)-\d_Ft}\mu_F(S_\L)$$
for some constant $L(\e,\l)>0$, where $\Per_F(\c):=\int_0^{|\c|}F(\dot c(s))ds$ with $c$ being the unique hyperbolic axis of $\c$.
\end{lemma}

\begin{proof}
We first estimate $e^{C_{F\circ \iota-\d_F,\xi}(o, \pi(v))+C_{F-\d_F,\eta}(o, \pi(v))}$ given $\xi\in \bP_\L, \eta\in \bF_\L$ and $v=\dot c_{\xi,\eta}(0)\in B^\a$.
Then there exists $w\in B^{3\e/2}\cap\L$ with $w^-=\xi$. Noticing that $\iota(w)\in \UR(\l)$, we have by Lemma \ref{Bowen}
\begin{equation*}
\begin{aligned}
|C_{F\circ \iota-\d_F,\xi}(o, \pi(v))|&\le |C_{F\circ \iota-\d_F,\xi}(o, \pi(w))+C_{F\circ \iota-\d_F,\xi}(\pi(w), \pi(v))|\\
&\le 2C_2(4\kappa\e)^\a e^{4\kappa\e\a/\l}/\l+8(\|F\|+\d_F)\e=:L_1(\e,\l)
\end{aligned}
\end{equation*}
where $\kappa>1$ is the constant in the definition of local product structure (\cite[Definition 4.2]{BCFT}). Similarly, $|C_{F-\d_F,\eta}(o, \pi(v))|\le L_1(\e,\l)$.

Then we estimate $e^{C_{F\circ \iota-\d_F,\xi}(o, \pi(v))+C_{F-\d_F,\eta}(o, \pi(v))}$ given $\xi\in \bP_\L, \eta\in \c\bF_\L$ and $v=\dot c_{\xi,\eta}(0)\in B^\a$. We have proved $|C_{F\circ \iota-\d_F,\xi}(o, \pi(v))|\le  L_1(\e,\l)$. 
 Denote $v':=\dot c_{o,\eta}(0)$. As $\c\in  \C^*_{\L}(t,\a)$, there exists $w_1\in \tilde S_\L\cap g^{-t}\c\tilde B^\a_\L$. On the other hand, since $\eta\in\c \bF_\L$, there exists $w_2\in  B^{3\e/2}\cap \L$ with $\c w_2^+=\eta$. Using an idea of shadowing, more precisely, setting $w_3\in W_\loc^u(g^tw_1)\cap W^{cs}_\loc(\c w_2)$, then the two orbit segments $(g^{-t}w_3,t)$ and $(w_1,t)$ are $4\kappa\e$-close, $w_3^+=\eta$ and $d(o, \pi(g^{-t}w_3))\le 4\kappa\e+4\e.$ 
Thus in fact, the four orbit segments $(g^{-t}w_3,t)$, $(w_1,t)$, $(v', t)$ and $(v,t)$ are $(4\kappa\e+12\e)$-close to each other, so a similar upper bound as above also holds for $|\int_o^{\pi (g^tv')}(F-\d_F)-\int_{\pi v}^{\pi(g^tv)}(F-\d_F)|$. Note that $\c w_2\in \L$, so a similar upper bound as above also holds for $|C_{F-\d_F,\eta}(\pi(g^tv'), \pi(g^tv))|$. Therefore, a similar upper bound, still denoted by $L_1(\e,\l)$ for simplicity, also holds for $|C_{F-\d_F,\eta}(o, \pi(v))|$.

Thus we have
\begin{equation}\label{ratio}
\begin{aligned}
&\frac{\mu_F(S_\L^\c)}{\mu_F(S_\L)}=\frac{\e^2\bar \mu_{F,\l_0}(\bP_\L \times \c\bF_\L)}{\e^2\bar \mu_{F,\l_0}(\bP_\L \times \bF_\L)}\\
=&\frac{e^{\pm 2L_1(\e,\l)}\mu_{F,o}(\bP_\L\cap \C\cdot(U\cap \BBB^{\l_0}))\mu_{F,o}(\c\bF_\L\cap \C\cdot(V\cap \BBB^{\l_0}))}{e^{\pm 2L_1(\e,\l)}\mu_{F,o}(\bP_\L\cap \C\cdot(U\cap \BBB^{\l_0}))\mu_{F,o}(\bF_\L\cap \C\cdot(V\cap \BBB^{\l_0}))}\\
=&e^{\pm 4L_1(\e,\l)}\frac{\mu_{F,\c^{-1}o}(\bF_\L\cap \C\cdot(V\cap \BBB^{\l_0}))}{\mu_{F,o}(\bF_\L\cap \C\cdot(V\cap \BBB^{\l_0}))}\\
=&e^{\pm 4L_1(\e,\l)}\frac{\int_{\bF_\L\cap \C\cdot(V\cap \BBB^{\l_0})}e^{-C_{F-\d_F,\eta}(\c^{-1}o,o)}d\mu_{F,o}(\eta)}{\mu_{F,o}(\bF_\L\cap \C\cdot(V\cap \BBB^{\l_0}))}.
%=&e^{\pm 2L(\e,\l)}\e^{\pm h\e^2}e^{-h|\c|}=e^{\pm 2L(\e,\l)}e^{-h|\c|}.
\end{aligned}
\end{equation}
Let us estimate $C_{F-\d_F,\eta}(\c^{-1}o,o)=C_{F-\d_F,\c\eta}(o,\c o)$. As above, since $\c\eta\in\c \bF_\L$, there exists $w_2\in B^{3\e/2}\cap \L$ with $\c w_2^+=\c\eta$. Then for any $T>0$,
\begin{equation}\label{ratio1}
\begin{aligned}
&\left|\int_{\c o}^{c_{\c o, \c\eta}(T)}(F-\d_F)-\int_{c_{o,\c\eta}(t)}^{c_{o,\c\eta}(t+T)}(F-\d_F)\right|\\
\le&2C_2(4\kappa\e)^\a e^{4\kappa\e\a/\l}/\l+8(\|F\|+\d_F)\e=L_1(\e,\l).
\end{aligned}
\end{equation}

On the other hand, as $\c\in \C^*_{\L}(t,\a)$, there exists $w_0$ lying on the unique axis of $\c$ such that $w_0\in S$ and $g^{|\c|}w_0=\c w_0\in \c B^\a$. Note that as above there exists $w_1\in \tilde S_\L\cap g^{-t}\c(\tilde B^\a_\L)$. Clearly $|t-|\c||\le 8\e$. 
By a similar argument as above,
\begin{equation}\label{ratio2}
\begin{aligned}
|\Per_F(\c)-\int_{o}^{c_{o,\c\eta}(t)}F|\le &L_2(\e,\l)
\end{aligned}
\end{equation}
for some $L_2(\e,\l)>0$ independent of $t$. Combining \eqref{ratio}, \eqref{ratio1} and \eqref{ratio2}, we have 
\begin{equation*}
\begin{aligned}
\frac{\mu_F(S_\L^\c)}{\mu_F(S_\L)}=&e^{\pm 4L_1(\e,\l)}\frac{\int_{\bF_\L\cap \C\cdot(V\cap \BBB^{\l_0})}e^{-C_{F-\d_F,\eta}(\c^{-1}o,o)}d\mu_{F,o}(\eta)}{\mu_{F,o}(\bF_\L\cap \C\cdot(V\cap \BBB^{\l_0}))}\\
=&e^{\pm L(\e,\l)}e^{\Per_F(\c)-\d_Ft}
\end{aligned}
\end{equation*}
where $L(\e,\l):=5L_1(\e,\l)+L_2(\e,\l)$.
\end{proof}
\begin{remark}\label{constant}
From the above calculation, $L(\e,\l)\to 0$ as $\e\to 0$ and there exists a constant $Q'>0$ independent of $\e$ such that $L(\e,\l)\le Q'\e^\a$ if $\e$ is small enough.
\end{remark}

For clarity, we use an underline to denote objects in $M$ and $SM$, e.g. for $A\subset SX$, $\underline A:=dp(A)$, where $p: X\to M$ is the covering map.

\begin{proposition}\label{asymptotic}
We have
\begin{equation*}
\begin{aligned}
\frac{e^{-L(\e,\l)}}{1+\rho_0}& \lesssim \frac{\sum_{\c\in \C^*_{\theta,\L}(t,\a)}e^{\Per_F(\c)-\d_Ft}}{\mu_F(B_\theta^\a)} \lesssim e^{L(\e,\l)}(1+\rho_0)(1+\frac{4\e^2}{\a}),\\
\frac{e^{-L(\e,\l)}}{1+\rho_0}& \lesssim \frac{\sum_{\c\in \C_{\theta,\L}(t,\a)}e^{\Per_F(\c)-\d_Ft}}{\mu_F(B_\theta^\a)} \lesssim e^{L(\e,\l)}(1+\rho_0)(1+\frac{4\e^2}{\a}).
\end{aligned}
\end{equation*}
\end{proposition}
\begin{proof}
Recall that $\a\in (0, \frac{3\e}{2}]$. By Lemmas \ref{closing}, \ref{depth2} and \ref{depth1}, for any $0<\rho<\theta$ and $t$ large enough, we have
\begin{equation}\label{inclusion}
\begin{aligned}
\lS_{\rho,\L}\cap g^{-t}\lB_{\rho,\L}^{\a}\subset \bigcup_{\c\in \C^*_{\theta,\L}(t,\a)}\lS_{\theta,\L}^\c \subset \lS_\theta \cap g^{-(t+2\e^2)}\lB_\theta^{\a+4\e^2}.
\end{aligned}
\end{equation}
From the proof of Lemma \ref{scaling1}, we know for any $v\in S_{\theta,\L}^\c$, the geodesic $c_v$ passes through a small neighborhood of $\c B_\theta^\a$. Hence the union in \eqref{inclusion} is actually a disjoint union. By Lemma \ref{scaling1}, 
$$\mu_F(S_{\theta,\L}^\c)=e^{\pm L(\e,\l)}e^{\Per_F(\c)-\d_Ft}\mu_F(S_{\theta,\L}).$$ 
Thus we have
\begin{equation*}
\begin{aligned}
e^{-L(\e,\l)}\mu_F(\lS_{\rho,\L}\cap g^{-t}\lB_{\rho,\L}^{\a})&\le \mu_F(S_{\theta,\L})\sum_{\c\in \C^*_{\theta,\L}(t,\a)}e^{\Per_F(\c)-\d_Ft}\\
&\le e^{L(\e,\l)}\mu_F(\lS_\theta\cap g^{-(t+2\e^2)}\lB_\theta^{\a+4\e^2}).
\end{aligned}
\end{equation*}
Dividing by $\mu_F(\lS_\theta)\mu_F(\lB_\theta^\a)$ and using mixing of $\mu_F$, we get
\begin{equation}\label{e:sim}
\begin{aligned}
e^{-L(\e,\l)}\frac{\mu_F(S_{\rho,\L})\mu_F(B_{\rho,\L}^\a)}{\mu_F(S_\theta)\mu_F(B_\theta^\a)}& \lesssim \frac{\mu_F(S_{\theta,\L})\sum_{\c\in \C^*_{\theta,\L}(t,\a)}e^{\Per_F(\c)-\d_Ft}}{\mu_F(S_\theta)\mu_F(B_\theta^\a)}\\
 & \lesssim \ e^{L(\e,\l)}\frac{\mu_F(B_\theta^{\a+4\e^2})}{\mu_F(B_\theta^\a)}.
\end{aligned}
\end{equation}
By \eqref{e:choice}, \eqref{density} and \eqref{density2}, letting $\rho \nearrow \theta$, we obtain the first equation in the proposition.

To prove the second equation, we consider $\rho<\theta<\rho_1< \theta_0$. Then by Lemma \ref{closing}, 
$\C^*_{\theta,\L}(t,\a)\subset \C_{\theta,\L}(t,\a) \subset \C^*_{\rho_1,\L}(t,\a).$ Combining with \eqref{e:sim},
\begin{equation*}
\begin{aligned}
e^{-L(\e,\l)}\frac{\mu_F(S_{\rho,\L})\mu_F(B_{\rho,\L}^\a)}{\mu_F(S_\theta)\mu_F(B_\theta^\a)}& \lesssim \frac{\mu_F(S_{\theta,\L})\sum_{\c\in \C_{\theta,\L}(t,\a)}e^{\Per_F(\c)-\d_Ft}}{\mu_F(S_\theta)\mu_F(B_\theta^\a)},\\
 e^{L(\e,\l)}\frac{\mu_F(B_{\rho_1}^{\a+4\e^2})}{\mu_F(B_{\rho_1}^\a)}&\gtrsim \frac{\mu_F(S_{\rho_1,\L})\sum_{\c\in \C^*_{{\rho_1},\L}(t,\a)}e^{\Per_F(\c)-\d_Ft}}{\mu_F(S_{\rho_1})\mu_F(B_{\rho_1}^\a)}\\
 &\gtrsim \frac{\mu_F(S_{\rho_1,\L})\sum_{\c\in \C_{\theta,\L}(t,\a)}e^{\Per_F(\c)-\d_Ft}}{\mu_F(S_{\rho_1})\mu_F(B_{\rho_1}^\a)}.
\end{aligned}
\end{equation*}
Letting $\rho_1\searrow \theta$, $\rho\nearrow \theta$ and by \eqref{e:choice}, \eqref{density} and \eqref{density2}, we get the second equation in the proposition.
\end{proof}

\subsection{Measuring along periodic orbits}
Recall that
\begin{equation}\label{e:ct}
\begin{aligned}
\nu_{F,t}:=&\frac{1}{\sum_{c\in C(t)}e^{\Per_F(c)}}\sum_{c\in C(t)}e^{\Per_F(c)}\frac{Leb_c}{t}.
\end{aligned}
\end{equation}
Define
$$\Pi(t):=\{\dot{\lc}(s)\in dp( H^{-1}(\bP\times \bF\times \{0\})): \lc\in C(t), s\in \RR\}.$$

Now we define a map $\Theta: \Pi(t)\to \C(t,\e)$ as follows. Given $\lv\in \Pi(t)$, let $\ell=\ell(\lv)\in (t-\e,t]$ be such that $g^\ell\lv=\lv$. Let $v$ be the unique lift of $\lv$ such that $v\in H^{-1}(\bP\times \bF\times \{0\})\subset B_\theta^\a$.
Define $\Theta(\lv)$ to be the unique axial isometry of $X$ such that $g^\ell v=\Theta(\lv) v$. Then $|\Theta(\lv)|=\ell$. If $\c=\Theta(\lv)$, then $g^tv=g^{t-\ell}\c v\in \c B_\theta^\e$. So $v\in S_\theta\cap g^{-t}\c B_\theta^\e$, and we get
\begin{equation*}
\begin{aligned}
\Theta(\Pi(t))\subset \C(t,\e).
\end{aligned}
\end{equation*}

Now we define $\Pi_\L(t)\subset \Pi(t)$ such that for any $\lv\in \Pi_\L(t)$, $\Theta(\lv)\in \C_\L(t,\e)$, that is, $\tilde S_{\theta,\L} \cap g^{-t}\Theta(\lv) (\tilde B_{\theta,\L}^\e)\neq \emptyset$. Thus $\Theta_\L:=\Theta|_{\Pi_\L(t)}$ satisfies
\begin{equation}\label{e:theta}
\begin{aligned}
\Theta_\L(\Pi_\L(t))\subset \C_\L(t,\e).
\end{aligned}
\end{equation}

Recall that $\tilde B^\e_\L:=P(B^\e\cap \L)\times [0,\e]$. 
\begin{proposition}\label{upper}
We have 
\[\sum_{c\in C(t)}e^{\Per_F(c)}\le \frac{\e\sum_{\c\in \C_\L(t,\e)}e^{\Per_F(\c)}}{t\nu_{F,t}(\tilde\lB^\e_\L)}.\]
\end{proposition}
\begin{proof}
By \cite[Lemma 5.1]{Wu2}, we know that $\Theta_\L$ is also injective. Then the proposition follows from \eqref{e:ct} and \eqref{e:theta}.
\end{proof}

\subsection{Primitive closed geodesics}
In this subsection, we consider the multiplicity of $\c\in \C$. Given $\c\in \C$, let $d=d(\c)\in \NN$ be maximal such that $\c= \b^d$ for some $\b\in \C$. $\c\in \C$ is called \emph{primitive} if $d(\c)=1$, i.e., $\c\neq \b^d$ for any $\b\in \C$ and any $d\ge 2$. 
Recall that $$\C'_\L(t,\a):=\{\c\in \C^*_\L(t,\a): \c\neq \b^n \text{\ for any\ } \b\in \C, n\ge 2\}.$$

\begin{lemma}\label{lower}
Consider $\a=\e-4\e^2$. Then $\Theta_\L(\Pi_\L(t))\supset \C'_\L(t-2\e^2, \a)$. Moreover, 
\begin{equation*}
\begin{aligned}
\sum_{c\in C(t)}e^{\Per_F(c)}&\ge \frac{\a\sum_{\c\in \C'_\L(t-2\e^2, \a)}e^{\Per_F(\c)}}{t\nu_{F,t}(\lB_{\theta}^{\a})}.
\end{aligned}
\end{equation*}
\end{lemma}
\begin{proof}
Let $\c\in \C'_\L(t-2\e^2, \a)$.
Then there exists $v\in H^{-1}(\bP\times \bF\times \{0\})$ such that $g^{|\c|}v=\c v$. By Lemma \ref{intersection2}(3), we have
\begin{equation*}
\begin{aligned}
|\c|&\ge (t-2\e^2)-\a-\e^2=t-\e+\e^2>t-\e,\\
|\c|&\le (t-2\e^2)+2\e^2=t.
\end{aligned}
\end{equation*}
Since $\c$ is primitive, it follows that $\underline c_{\lv}$ is a closed geodesic with length $|\c|\in (t-\e,t]$.
Note that if $\underline c$ is another closed geodesic in the free-homotopic class of $\underline c_{\lv}$, then we can lift $\underline c$ to a geodesic $c$ such that $c$ and $c_v$ are bi-asymptotic. So $c$ and $c_v$ bound a flat strip, which is a contradiction since $v\in \text{Reg}$. It follows that $\underline c_{\lv}$ is the only geodesic in its free-homotopic class. Thus $\underline c_{\lv}\in C(t)$.

It is easy to check that $\tilde{S}_\L \cap g^{-(t-2\e^2)}\c\tilde B_{\L}^\a\subset \tilde{S}_\L \cap g^{-t}\c\tilde B_{\L}^\e$. Thus $\lv\in \Pi_\L(t)$ and $\c=\Theta(\lv)$. So $\Theta_\L(\Pi_\L(t))\supset \C'_\L(t-2\e^2, \a)$ and thus by \eqref{e:ct},
\[\sum_{c\in C(t)}e^{\Per_F(c)}\ge \frac{\a\sum_{\c\in \C'_\L(t-2\e^2, \a)}e^{\Per_F(\c)}}{t\nu_{F,t}(\lB_{\theta}^{\a})}.\]
We are done.
\end{proof}

Define $\C_d(\bP, \bF, t)$ to be the set of all $\c\in \C$ such that
\begin{enumerate}
  \item $\c$ has an axis $c$ with $c(-\infty)\in \bP, c(\infty)\in \bF$;
  \item $|\c|\in (t-\e,t]$;
  \item $d(\c)= d$.
\end{enumerate}
We define $\C_{d,\L}(\bP, \bF, t):=\C_d(\bP, \bF, t)\cap \C^*_\L(t,\e)$.

The following result is standard in ergodic theory, which follows from the classical proof of variational principle \cite[Theorem 9.10]{W}.
\begin{lemma}\label{equilemma}
Let $Y$ be a compact metric space, $\phi=\{\phi^t\}_{t\in \RR}$ a continuous flow on $Y$ and $F\in C(Y,\RR)$. Fix $\e > 0$ and suppose that $E_t\subset Y$ is a $(t,\e)$-separated set for all sufficiently large $t$. Define the measures $\mu_t$ by
$$\mu_t(A) :=\frac{\sum_{y\in E_t}e^{\int_0^tF(\phi^sy)ds}\frac{1}{t}\int_0^t\chi_A(\phi^sy)ds}{\sum_{y\in E_t}e^{\int_0^tF(\phi^sy)ds}}.$$
If $t_k\to \infty$ and the weak$^*$ limit $\mu=\lim_{k\to \infty}\mu_{t_k}$
exists, then
$$h_\mu(\phi^1)+\int_Y Fd\mu\ge \limsup_{k\to \infty}\frac{1}{t_k}\log \sum_{y\in E_{t_k}}e^{\int_0^{t_k}F(\phi^sy)ds}.$$
\end{lemma}

\begin{lemma}\label{multi}
We have
$$\lim_{t\to +\infty}e^{-\d_Ft}\sum_{\c\in \cup_{d\ge 2}\C_{d,\L}(\bP, \bF, t)}e^{\Per_F(\c)}=0.$$
\end{lemma}
\begin{proof}
\textbf{Step 1.}
For every $\c\in \C^*(t,\e)$, let $v_\c\in H^{-1}(\bP\times \bF\times\{0\})$ tangent to the unique axis of $\c$. Define measures on $SM$ as
\[\mu_{\L,t}:=\frac{\sum_{\c\in \C^*_\L(t,\e)}e^{\Per_F(\c)}\frac{1}{t}\text{Leb}_{\lv_\c}}{\sum_{\c\in \C^*_\L(t,\e)}e^{\Per_F(\c)}}.\]
Clearly, $\{v_\c: \c\in \C^*_\L(t,\e)\}$ is a $(t,\e)$-separated set (see also the proof of \cite[Theorem B]{Wu2}).  By Proposition \ref{asymptotic} we have 
$$\liminf_{t\to \infty}\log\sum_{\c\in \C^*_\L(t,\e)}e^{\Per_F(\c)}\ge \d_F.$$
By Lemma \ref{equilemma} and the uniqueness of equilibrium states, $\lim_{t\to \infty}\mu_{\L,t}=\mu_F$.

We claim that for any $\rho'>0$ there exists $t_0=t_0(\rho')>0$ such that if $\b\in \C_\L^*(t,\e), |\b|>t_0$, then 
\[\frac{e^{\Per_F(\b)}}{\sum_{\c\in \C^*_\L(t,\e)}e^{\Per_F(\c)}}\le \rho'.\]
Indeed, assume the contrary. Then there are $\b_n\in \C^*_\L(t_n,\e), t_n\to \infty$ such that 
\[\frac{e^{\Per_F(\b_n)}}{\sum_{\c\in\C^*_\L(t_n,\e)}e^{\Per_F(\c)}}> \rho'.\]
Then $\mu_{\L,t_n}(B_{t_n}(\lv_{\b_n},\e))\ge \rho'$. By passing to a subsequence, assume that $\lv_{\b_n}\to v$. Then any open neighborhood of $v$ has $\mu_F$-measure no less than $\rho'$. A contradiction since $\mu_F$ has no atom. The claim follows.

Given $\rho'>0$ small enough and corresponding $t_0=t_0(\rho')$, we have 
\begin{equation*}
\begin{aligned}
&\sum_{\c \in \C_{d,\L}(\bP, \bF, t)}e^{\Per_F(\c)}=\sum_{\c \in \C_{d,\L}(\bP, \bF, t)}e^{d\Per_F(\b(\c))}\\
=&\sum_{\c \in \C_{d,\L}(\bP, \bF, t), |\b(\c)|\le t_0}e^{d\Per_F(\b(\c))}+\sum_{\c \in \C_{d,\L}(\bP, \bF, t),|\b(\c)|>t_0}e^{d\Per_F(\b(\c))}\\
=&:\Sigma_1+\Sigma_2.
\end{aligned}
\end{equation*}
We will estimate $\Sigma_1$ and $\Sigma_2$ separately in the next two steps.

\textbf{Step 2.} 
If $\c\in \C_{d,\L}(\bP, \bF, t)$, let $c$ be the unique axis of $\c$ with $v=\dot c(0)\in S_\theta$. Then $c$ is also an axis of $\b(\c)$. Since $\c\in \C^*_\L(t,\e)$, there exists $w\in \tilde S_\L\cap g^{-t}\c \tilde B_\L^\e$. We have $d(\pi w, \pi v)\le 4\e$ and $d(\pi g^tw, \pi g^{|\c|}v)\le 4\e$. So $|t-|\c||\le 8\e$. By convexity and triangle inequality, we have
\[d(\pi g^{s}v, \pi g^sw)\le 12\e, \quad \forall s\in [0,t].\]
Note that $g^{s_1}w\in \L, \iota(g^{t+s_2}w)\in \L$ for some $|s_1|\le \e^2, |s_2|<\e$, $w':=\b(\c)\c^{-1}\iota g^{t+s_2}w\in \L$, $v_{\b(\c)}\in B_{|\b(\c)|}(w,12\e)$ and $\iota\b(\c)v_{\b(\c)}\in B_{|\b(\c)|}(w',12\e)$. Then by a similar argument as in establishing \eqref{ratio2} or \eqref{e:Bowen1},
\begin{equation}\label{ratio4}
\begin{aligned}
|\Per_F(\b(\c))-\int_{o}^{\b(\c) o}F|\le &L_4(\e,\l).
\end{aligned}
\end{equation}
We emphasize that it is not clear whether $\b(\c)\in \C^*_\L(t/d,\e/d)$. Nevertheless, the above estimate still holds. For convenience, we introduce
\begin{equation*}
\begin{aligned}
\tilde\C^*_\L(t,\e)&:=\{\b(\c):\c\in \C_{d,\L}(\bP, \bF, dt), d\ge 1\},\\
\overline\C^*_\L(t,\e)&:=\{\c\in \C^*: \c=\b_1\b_2\cdots\b_d, \b_i\in \tilde\C^*_\L(t/d,\e/d), 1\le i\le d, d\ge 1\}.
\end{aligned}
\end{equation*}
Proposition \ref{asymptotic} and hence the above claim in Step 1 still hold if $\C^*_\L(t,\e)$ is replaced by $\tilde\C^*_\L(t,\e)$ or by $\overline\C^*_\L(t,\e)$. 

Let us estimate $\Sigma_2$ now. Since $|\b(\c)|>t_0$, by the above claim,
\[\frac{e^{\Per_F(\b(\c))}}{\sum_{\c\in \tilde\C^*_\L(t,\e)}e^{\Per_F(\c)}}<\rho'.\]
It follows that
\begin{equation*}
\begin{aligned}
\Sigma_2=&\sum_{\c \in \C_{d,\L}(\bP, \bF, t), |\b(\c)|\ge t_0}e^{d\Per_F(\b(\c))}\\
\le &(\rho')^{d-1}\sum_{\b \in  \tilde\C^*_\L(t/d,\e)}e^{\Per_F(\b)}(\sum_{\b \in  \tilde\C^*_\L(t/d,\e)}e^{\Per_F(\b)})^{d-1}.
\end{aligned}
\end{equation*}
Suppose that $\b_1,\cdots,\b_{d-1},\b_{d}\in \tilde\C^*_\L(t/d,\e/d)$. Denote $\a=\b_1\cdots\b_{d-1}\b_{d}.$  Clearly, $\a\in \C^*$. Moreover, $\a\in \overline\C^*_\L(t,\e)$. 
By a similar argument in establishing \eqref{ratio4}, we have
$$|\Per_F(\a)- \sum_{i=1}^d\Per_F(\b_i)|\le d L_5(\e,\l).$$
Then 
\begin{equation*}
\begin{aligned}
\Sigma_2\le (\rho')^{d-1}e^{dL_5(\e,\l)}\sum_{\a \in  \overline\C^*_\L(t,\e)}e^{\Per_F(\a)}.
\end{aligned}
\end{equation*}
Choose $\rho'$ small enough such that $ \rho_2:=\rho'e^{2L_5(\e,\l)}\ll1$. Then we have
\begin{equation}\label{sigma2}
\begin{aligned}
\Sigma_2\le \rho_2^{d-1}\sum_{\a \in  \overline\C^*_\L(t,\e)}e^{\Per_F(\a)} \le c\rho_2^{d-1}e^{\d_Ft}
\end{aligned}
\end{equation}
where $c$ is from (modified) Proposition \ref{asymptotic} only dependent of $\l$ and $\e$. 

\textbf{Step 3.} 
For every $\a\in \tilde\C^*(|\a|,\e)$, we have $e^{\Per_F(\a)}<e^{\d_F|\a|}$. Indeed, otherwise  $e^{\Per_F(\a^n)}\ge e^{n\d_F|\a|}$ where $n\ge 0$ is such that $\a^n\in \C_\L^*(n|\a|,\e)$.
 By the above claim and Proposition \ref{asymptotic}, for every $\rho'>0$, 
\[e^{n\d_F|\a|}\le e^{\Per_F(\a^n)}<\rho'\sum_{\c\in \C^*_\L(n|\a|,\e)}e^{\Per_F(\c)}\le \rho'Ce^{n\d_F|\a|}\]
where $C$ is a constant independent of $n$. Take $\rho'>0$ small enough such that $\rho'C<1$, we get a contradiction. 

\begin{comment}
For every $\c \in \C_{d,\L}(\bP, \bF, t)$ with $t$ large enough, we have 
$$e^{\Per_F(\b(\c))}<e^{(\d_F-\e_0)|\b(\c)|}.$$ 
Indeed, otherwise  $e^{d\Per_F(\b(\c))}\ge e^{d\d_F|\b(\c)|}$.
Since $\c\in \C^*_\L(t,\e)$, by the above claim and Proposition \ref{asymptotic}, for every $\rho'>0$, if $t>t_0(\rho')$
\[e^{(\d_F-\e)(t-\e)}\le e^{d(\d_F-\e_0)|\b(\c)|}\le e^{d\Per_F(\b(\c))}=e^{\Per_F(\c)}\le \rho'\sum_{\a\in \C^*_\L(t,\e)}e^{\Per_F(\a)}\le \rho'Ce^{\d_Ft}\]
where $C$ is a constant independent of $t$. Take $\rho'>0$ small enough such that $\rho'Ce^{|\d_F|\e}<1$, we get a contradiction. 
\end{comment}

Define $\rho_1:=\max_{\a\in \tilde\C^*(|\a|,\e), |\a|\le t_0}e^{\frac{\Per_F(\a)}{|\a|}-\d_F}$. From the above discussion, $0<\rho_1<1$. 
\begin{equation}\label{sigma1}
\begin{aligned}
\Sigma_1\le &\sum_{\c \in \C_{d,\L}(\bP, \bF, t), |\b(\c)|\le t_0}\rho_1^{d|\b(\c)|}e^{d\d_F|\b(\c)|}\\
\le &\sum_{\a\in \C^*, |\a|\le t_0}\rho_1^{t-\e}e^{\d_Ft}e^{|\d_F|\e}= C\rho_1^te^{\d_Ft}
\end{aligned}
\end{equation}
where $C=\rho_1^{-\e}e^{|\d_F|\e}\#\{\a\in \C^*, |\a|\le t_0\}$.

\textbf{Step 4.} 
In summary, by \eqref{sigma2}, \eqref{sigma1}, and noticing that $d\le \frac{t}{\inj(M)}$,
\begin{equation*}
\begin{aligned}
\sum_{d=2}^{\lfloor\frac{t}{\inj(M)}\rfloor}\sum_{\c\in \C_{d,\L}(\bP, \bF, t)}e^{\Per_F(\c)}\le  &\frac{Ct}{\inj(M)}\rho_1^te^{\d_Ft}+ \sum_{d=2}^\infty c\rho_2^{d-1}e^{\d_Ft}\\
= &\Big(\frac{Ct}{\inj(M)}\rho_1^t+\frac{c\rho_2}{1-\rho_2}\Big)e^{\d_Ft}.
\end{aligned}
\end{equation*}
%where $\rho=\frac{\rho_1^2}{1-\rho_1}.$ there exists $c>0$ such that $\sum_{\a \in  \tilde\C^*_\L(t,\e)}e^{\Per_F(\a)}\le ce^{\d_Ft}$. 
Thus
$$\lim_{t\to +\infty}e^{-\d_Ft}\sum_{\c\in \cup_{d\ge 2}\C_{d,\L}(\bP, \bF, t)}e^{\Per_F(\c)}\le \frac{c\rho_2}{1-\rho_2}.$$
As $\rho_2$ could be arbitrarily small, we are done.
\end{proof}

\begin{proposition}\label{lower1}
Consider $\a=\e-4\e^2$. We have
\begin{equation*}
\begin{aligned}
\sum_{c\in C(t)}e^{\Per_F(c)}&\gtrsim \frac{\a}{t\nu_{F,t}(\lB_{\theta}^{\a})}\sum_{\c\in \C^*_\L(t-2\e^2, \a)}e^{\Per_F(\c)}.
\end{aligned}
\end{equation*}
\end{proposition}
\begin{proof}
From the proof of Lemma \ref{lower}, we see that $|\c|\in (t-\e,t]$ if $\c\in \C^*(t-2\e^2,\a)$. Thus
$$\C^*(t-2\e^2,\a)\setminus \C'(t-2\e^2,\a)\subset \cup_{d\ge 2}\C_d(\bP, \bF, t).$$
By Lemmas \ref{multi}, \ref{lower} and Proposition \ref{asymptotic},
\begin{equation*}
\begin{aligned}
\sum_{c\in C(t)}e^{\Per_F(c)}&\ge \frac{\a}{t\nu_{F,t}(\lB_{\theta}^{\a})}\left(\sum_{\c\in \C^*_\L(t-2\e^2, \a)}e^{\Per_F(\c)}-\sum_{\c\in  \cup_{d\ge 2}\C_{d,\L}(\bP, \bF, t)}e^{\Per_F(\c)}\right)\\
&\gtrsim \frac{\a}{t\nu_{F,t}(\lB_{\theta}^{\a}))}\sum_{\c\in \C^*_\L(t-2\e^2, \a)}e^{\Per_F(\c)}.
\end{aligned}
\end{equation*}
We are done.
\end{proof}

\subsection{Proof of Theorems \ref{equi0} and \ref{margulis}}
\begin{proof}[Proof of Theorem \ref{equi0}]
For any $0<\e<\inj(M)/2$, the set $\{\underline{\dot  c}(0): c\in C(t)\}$ is $(t,\e)$-separated by the proof of \cite[Theorem B]{Wu2}.
Now by Propositions \ref{lower1} and \ref{asymptotic}, we know
\begin{equation*}
\begin{aligned}
\sum_{c\in C(t)}e^{\Per_F(c)}\gtrsim &\frac{\a}{t\nu_{F,t}(\lB_{\theta}^{\a})}\sum_{\c\in \C^*_\L(t-2\e^2, \a)}e^{\Per_F(\c)}\\
\gtrsim  &\frac{\a}{t\nu_{F,t}(\lB_{\theta}^{\a})}\left(\frac{e^{-L(\e,\l)}}{1+\rho_0}e^{\d_Ft}\mu_F(B_\theta^\a)\right).
\end{aligned}
\end{equation*}
So $\liminf_{t\to \infty}\frac{1}{t}\log \sum_{c\in C(t)}e^{\Per_F(c)}\ge \d_F.$ Applying Lemma \ref{equilemma}, we know any limit measure of $\nu_{F,t}$ must be $\mu_F$, the unique equilibrium state for $F$. This proves Theorem \ref{equi0}.
\end{proof}

\begin{proposition}\label{sumup}
We have
\begin{equation*}
\begin{aligned}
\sum_{c\in C(t)}e^{\Per_F(c)}&\lesssim C_B(1+\rho_0)^{2}e^{Q\e^\a}\frac{\e}{t}e^{\d_Ft},\\
\sum_{c\in C(t)}e^{\Per_F(c)}&\gtrsim (1+\rho_0)^{-2}e^{- Q\e^\a}\frac{\e}{t}e^{\d_Ft}
\end{aligned}
\end{equation*}
where $$C_B:=\limsup_{t\to\infty}\frac{\mu_F(\tilde \lB_\L^\e)}{\nu_{F,t}(\tilde \lB_\L^\e)},$$ 
$Q>0$ is a  constant independent of $\e$, and $\a$ is the H\"{o}lder constant of $F$.
\end{proposition}
\begin{proof}
By Propositions \ref{upper} and \ref{asymptotic}, \eqref{density}, 
\begin{equation*}
\begin{aligned}
\sum_{c\in C(t)}e^{\Per_F(c)}&\le \frac{\e\sum_{\c\in \C_\L(t,\e)}e^{\Per_F(\c)}}{t\nu_{F,t}(\tilde\lB^\e_\L)}\lesssim C_B\cdot\frac{\e\sum_{\c\in \C_\L(t,\e)}e^{\Per_F(\c)}}{t\mu_F(\tilde B^\e_\L)}\\
&\lesssim C_B e^{L(\e,\l)}(1+\rho_0)^2(1+4\e)\frac{\e}{t}e^{\d_Ft}.
\end{aligned}
\end{equation*}

On the other hand, consider $\a=\e-4\e^2$. By Theorem \ref{equi0} and \eqref{e:choice}, we have $\nu_{F,t}(\lB^\a)\to \mu_F(\lB^\a)=\mu_F(B^\a)$. By Propositions \ref{lower1} and \ref{asymptotic}, and \eqref{density},
\begin{equation*}
\begin{aligned}
&\sum_{c\in C(t)}e^{\Per_F(c)}\gtrsim \frac{\a}{t\nu_{F,t}(\lB_{\theta}^{\a})}\sum_{\c\in \C^*_\L(t-2\e^2, \a)}e^{\Per_F(\c)}\\
\gtrsim &\frac{\a}{t\mu_{F}(B_{\theta}^{\a})}\sum_{\c\in \C^*_\L(t-2\e^2, \a)}e^{\Per_F(\c)}
\gtrsim \frac{e^{-L(\e,\l)}}{(1+\rho_0)^2} (1-4\e)\frac{\e}{t}e^{-2|\d_F|\e^2}e^{\d_Ft}.
\end{aligned}
\end{equation*}

By Remark \ref{constant}, if $\e$ is small enough, there exists a constant $Q>0$ independent of $\e$ such that the inequalities in the proposition hold.
\end{proof}

To finish the proof of Theorem \ref{margulis}, we have to revisit assumption \eqref{density}. It is crucial that for a fixed $\l>0$, we can find a sequence of flow boxes $B_{\theta_n}^{\e_n}$ satisfying \eqref{density} with sizes $\e_n\to  0$. Recall that in the definition of flow box $B_\t^\e$, the choice of $\theta\in (0, \min\{\t_1,\t_2\})$ is dependent of $\e$ (see for example Corollary \ref{equicon}). Thus if $\e_n\to 0$, $\t_n\to 0$ as well. The conditions in the definition of flow boxes can be relaxed as follows. Let $\e>0$ be sufficiently small, $\t=\t(\e)$, and $v\in \L$. Let $\bP$ and $\bF$ be small compact neighborhoods of $v^-$ and $v^+$ respectively satisfying $\bP\subset \bP_{\theta}$ and $\bF\subset \bF_{\theta}$. We may call $H(\bP\times \bF\times [0,\e])$ a \emph{generalized flow box} with size $\e$.

 Without loss of generality, we may suppose that $\L$ is compact. Otherwise, we consider a compact subset of $\L$ whose complement in $\L$ has sufficient small $\mu_F$-measure.  The following lemma says that there exists a finite partition of $\L$ by generalized flow boxes. This is in the same spirit of \cite[Lemma 9.5.7]{BP}, with rectangles replaced by generalized flow boxes.
 
\begin{lemma}\label{partition}
For any sufficiently small $\e>0$, there exist disjoint (upto a set of zero $\mu_F$-measure)  generalized flow boxes $B_i=H(\bP_i\times \bF_i\times [0,\e_i])$ with $0<\e_i<\e$ for $i=1,\cdots,m$, such that $\L=\cup_{i=1}^m(B_i\cap \L)$.
\end{lemma}
\begin{proof}
Fix $\e>0$. Since $\L$ is compact, there exist finitely many flow boxes $B_\t^\e(v_i)$ centered at $v_i\in \L, i=1,\cdots,k$ such that $\L\subset \cup_{i=1}^k B^\e_\t(v_i)$. We can assume that these flow boxes satisfy all the properties in Subsection 5.1, but not \eqref{density} as a priori.

In spirit of  \cite[Lemma 9.5.6]{BP}, we claim that for any two flow boxes, say $B^\e_\t(v_1)=H(\bP_1\times \bF_1\times [0,\e])$ and  $B^\e_\t(v_2)=H(\bP_2\times \bF_2\times [0,\e])$, one can find finitely many disjoint generalized flow boxes $B_1,B_2,\cdots,B_l$ such that $\L\cap (\cup_{i=1}^lB_i)=\L\cap (B^\e_\t(v_1)\cup B^\e_\t(v_2))$.

To prove the claim, it is sufficient to consider case $B^\e_\t(v_1)\cap B^\e_\t(v_2)\neq \emptyset$. We divide $\bP_1\cup \bP_2$ into three parts: $\bP_1\cap \bP_2, \bP_1\setminus \bP_2, \bP_2\setminus \bP_1$. Similarly, divide $\bF_1\cup \bF_2$ into three parts: $\bF_1\cap \bF_2, \bF_1\setminus \bF_2, \bF_2\setminus \bF_1$. These are connected relatively compact neighborhoods in $\pX$. Choose a pair of them, denoted by $\bP', \bF'$. It is possible that in the flow direction there are nontrivial intersections, i.e., $P^{-1}(\bP'\times  \bF')\cap B^\e_\t(v_1)\cap B^\e_\t(v_2)\neq \emptyset$. If so, we then cut $P^{-1}(\bP'\times  \bF')\cap (B^\e_\t(v_1)\cup B^\e_\t(v_2))$ into three parts: $P^{-1}(\bP'\times  \bF')\cap (B^\e_\t(v_1)\cap B^\e_\t(v_2))$, $P^{-1}(\bP'\times  \bF')\cap(B^\e_\t(v_1)\setminus B^\e_\t(v_2))$ and $P^{-1}(\bP'\times  \bF')\cap (B^\e_\t(v_2)\setminus B^\e_\t(v_1))$. We then get at most $27$ parts, each of them is a generalized flow box (recall that the flow direction is defined by Busemann functions normalized at a common reference point $o\in \F$).

There is still a minor issue. It is possible that these are not genuine generalized flow boxes, since the length $\e'$ in the flow direction could be small relative to the size $\t'$ of $\bP'$ and $\bF'$ directions. If this happens, we continue to divide $\bP'$ and $\bF'$ into finitely many small ones such that their sizes are smaller than $\t(\e')$. Then we get genuine generalized flow boxes satisfying the claim, by dropping those with no intersection with $\L$.

Applying the claim to each pair of $\{B^\e_\t(v_i)\}_{i=1}^k$, we obtain the generalized flow boxes $B_i$ as required in the lemma.
\end{proof}
For any $\rho_0>0$, pick $\l>0$ such that $\mu_F(\L)>\frac{1}{1+\rho_0}$. As an immediate consequence of Lemma $\ref{partition}$, for any $\e>0$, there exists a generalized flow box $H(\bP\times \bF\times [0,\e'])$ with size $0<\e'\le \e$ satisfying \eqref{density}, as well as \eqref{e:choice0}, \eqref{e:choice}. More importantly, all the proofs in Section 5 so far go through with a flow box replaced by a generalized flow box. A fundamental difference is that $\bP$ and $\bF$ could be compact neighborhoods in $\pX$ with arbitrary shape. For example, the crucial closing lemma \ref{closing} for any compact neighborhood is guaranteed by \cite[Corollary 3.5]{Ri}.

\begin{proof}[Proof of Theorem \ref{margulis}]
The last step of the proof of Theorem \ref{margulis} is to estimate $\sum_{c\in P(t)}e^{\Per_F(c)}$ via $\sum_{c\in C(t)}e^{\Per_F(c)}$ using a Riemannian sum argument.
Indeed, a verbatim repetition of the proof in \cite[Section 6.2]{CKW2} gives
\begin{equation*}
\begin{aligned}
\sum_{c\in P(t)}e^{\Per_F(c)}&\lesssim C_B(1+\rho_0)^{2}e^{2(Q\e^\a+\d_F\e)}\frac{e^{\d_Ft}}{\d_Ft},\\
\sum_{c\in P(t)}e^{\Per_F(c)}&\gtrsim (1+\rho_0)^{-2}e^{-2(Q\e^\a+\d_F\e)}\frac{e^{\d_Ft}}{\d_Ft}.
\end{aligned}
\end{equation*}
%$$\sum_{c\in P_\L(t)}e^{\Per_F(c)}\sim (1+\rho_0)^{\pm 2}e^{\pm 2(Q\e^\a+\d_F\e)}\frac{e^{\d_Ft}}{\d_Ft}.$$
We emphasize that $\d_F>0$ is crucial here. From the discussion above, we are allowed to set $\e\to 0$ first, and then $\l\to 0$. Note that we can choose $\rho_0\to 0$ as $\l\to 0$. Therefore,
\begin{equation*}
\begin{aligned}
\sum_{c\in P(t)}e^{\Per_F(c)}&\lesssim C\frac{e^{\d_Ft}}{\d_Ft},\\
\sum_{c\in P(t)}e^{\Per_F(c)}&\gtrsim \frac{e^{\d_Ft}}{\d_Ft}.
\end{aligned}
\end{equation*}
In the above, $C=\sup_BC_B$ where the supremum is taken over any sequence of generalized flow boxes described above.
%$$\sum_{c\in P(t)}e^{\Per_F(c)}\gtrsim \frac{e^{\d_Ft}}{\d_Ft}.$$
The proof of Theorem \ref{margulis} is complete.
\end{proof}
\begin{remark}\label{CB1}
Since we can take $\L$ as a compact subset of $SM$, by Theorem \ref{equi0}, 
$$\limsup_{t\to \infty}\nu_{F,t}(\tilde \lB_\L^\e)\le \mu_F(\tilde \lB_\L^\e).$$
Thus $C\ge C_B\ge 1$. It seems possible that $C=+\infty$.
\end{remark}
\begin{proof}[Proof of Corollary E.1]
The ``if'' part follows from Theorem \ref{ps}. For the ``only if'' part, recall from Proposition \ref{asymptotic} that
$$\frac{e^{-L(\e,\l)}}{1+\rho_0} \lesssim \frac{\sum_{\c\in \C^*_{\theta,\L}(t,\a)}e^{\Per_F(\c)-\d_Ft}}{\mu_F(B_\theta^\a)}.$$
For any $\c\in \C^*_{\theta,\L}(t,\a)$, we can show as in  \eqref{ratio2} that
\begin{equation*}
\begin{aligned}
|\Per_F(\c)-\int_{o}^{\c o}F|\le &L_6(\e,\l)
\end{aligned}
\end{equation*}
for some $L_6(\e,\l)>0$ independent of $t$. Moreover, $|d(o,\c o)-t|\le 8\e$. Then we have
\begin{equation*}
\begin{aligned}
&\sum_{\c\in \C, n-1< d(o, \c o)\le n}e^{\int_o^{\c o}F}\ge \sum_{\c\in \C^*_{\theta,\L}(n-8\e,\a)}e^{\int_{o}^{\c o}F}\\
\ge &e^{-L_6(\e,\l)}\sum_{\c\in \C^*_{\theta,\L}(n-8\e,\a)}e^{\Per_F(\c)} \gtrsim C'e^{n\d_F}
\end{aligned}
\end{equation*}
where $C':=e^{-L_6(\e,\l)-8\e \d_F-L(\e,\l)}\cdot\frac{\mu_F(B_\theta^\a)}{1+\rho_0}$. The proof of Corollary E.1 is complete.
\end{proof}

\ \
\\[-2mm]
\textbf{Acknowledgement.}
This work is supported by National Key R\&D Program of China No. 2022YFA1007800 and NSFC No. 12071474.

\section{Appendix: Proofs of some technical lemmas}
\begin{proof}[Proof of Lemma \ref{uniformrec}]
It is clear from definition that $\UR$ is $g^t$-invariant and flip invariant. Now we consider a $g^t$-invariant measure $\nu$ on $SM$.
Let $\{U_n\}_{n\in\mathbb{N}}$ be a countable base consisting of open sets on $SM$. By Birkhoff ergodic theorem, there exists a set $X \subset SM$ of full $\nu$-measure such that for all $x \in X$ and all $n\in\mathbb{N}$, the limit
$$f_n(x):=\lim_{T \to \pm\infty} \frac{1}{T}\int_0^T \chi_{U_n}(g^tx)dt$$
exists, and
$$\int_{SM}f_n(x)d\nu(x)= \nu(U_n).$$

Assume the contrary that $\nu(\UR^c)>0$ where $\UR^c$ is the complement of $\UR$ in $SM$. Then $\UR^c \cap X$ is non-empty. For each $y \in\UR^c \cap X$, which is not uniformly recurrent, there exists a neighborhood $U$ of $y$ in $SM$ such that
\begin{equation*}\label{e:not}
 \liminf_{T \to +\infty} \frac{1}{T}\int_0^T \chi_U(g^ty)dt =0 \text{\ \ or\ \ } \liminf_{T \to -\infty} \frac{1}{T}\int_0^T \chi_U(g^ty)dt =0.
\end{equation*}
Then there exists an $n(y)$ such that $y\in U_{n(y)} \subset U$, and
\begin{equation}\label{e:f}
\begin{aligned}
 &f_{n(y)}(y)=\lim_{T \to \pm\infty} \frac{1}{T}\int_0^T \chi_{U_n}(g^ty)dt\\
 \leq &\min\Big\{\liminf_{T \to +\infty} \frac{1}{T}\int_0^T \chi_U(g^ty)dt,\  \liminf_{T \to -\infty} \frac{1}{T}\int_0^T \chi_U(g^ty)dt\Big\} =0.
\end{aligned}
\end{equation}
Denote $\UR^c(N):=\{y\in \UR^c: n(y)=N\}$ which is a subset of $U_N$.
Then we can find some $N$ such that $\nu (U_N \cap \UR^c(N) \cap X)>0$ and moreover, $f_N(y)=0$ for any $y \in U_N \cap \UR^c(N)\cap X$ by \eqref{e:f}.

On the other hand, Birkhoff ergodic theorem implies that for $\nu \ae y \in X$, one has
$$g(y):=\lim_{T \to \infty} \frac{1}{T}\int_0^T \chi_{(U_N \cap \UR^c(N)\cap X)}(g^ty)dt$$
exists with
\begin{equation}\label{e:gamma}
\int_{X}g(y)d\nu(y)=\nu(U_N \cap \UR^c(N)\cap X) >0.
\end{equation}
However, by \eqref{e:f}, we have $g(y) \leq f_N(y)=0$ for all $y \in U_N \cap \UR^c(N)\cap X$. Taking into account that $g$ is $g^t$-invariant, we have $g(y)=0$ for $\nu \ae y$, a contradiction to \eqref{e:gamma}. This proves the lemma.
\end{proof}

\begin{proof}[Proof of Lemma \ref{recrate}]
If $v\in \UR$, then for any neighborhood $U$ of $v$, there exists $\tau>0$ and $T_0>0$ such that,
\begin{equation}\label{e:birk}
\frac{1}{T}\int_0^T\chi_U(g^tv)dt>\tau, \quad \forall T>T_0.
\end{equation}
Now we lift $v$ and $U$ to the universal cover $X$. Given $T>0$, pick $\sigma>0$ small enough such that $\tau/\sigma>2T$ and $1/\sigma>T_0$.
We prove the lemma by induction. Suppose the conclusion holds for $n$, i.e., there exist $t_1<t_2 \cdots<t_n<\frac{n}{\sigma}$ and $\phi_i\in I(X)$ such that $d \phi_i g^{t_i}v\in U, i=1,\cdots n$. By \eqref{e:birk}, on $SM$ we have
$$\int_0^{(n+1)/\sigma}\chi_U(g^tv)dt>\tau(n+1)/\sigma.$$
Recall that $\tau(n+1)/\sigma>2nT$ by the choice of $\sigma$. Thus, we can find $t_{n+1}\in [0, (n+1)/\sigma)$ such that $t_{n+1}\notin [t_i-T,t_i+T], i=1, \cdots, n$, and such that
$d \phi_{n+1} g^{t_{n+1}}v\in U$
for some $\phi_{n+1}\in I(X)$. It is possible that $t_{n+1}<t_n$. But then just by reordering the points  $t_1,t_2, \cdots t_n, t_{n+1}$ we prove the conclusion for $n+1$. The lemma follows.
\end{proof}

\begin{proof}[Proof of Proposition \ref{rate}]
At first, we claim that if $v\in \text{Reg}$ is recurrent, then $J^s(v)$ contains no Jacobi field with constant length on $[0, \infty)$.
Otherwise, assume that there is a Jacobi field $Y\in J^s(v)$ along $c_v$ with $\|Y(t)\|\equiv C$ for any $t\ge 0$.
For any $t>0$, since $g^{-t}v$ is recurrent, there exists $t_n>t$ and $\gamma_n\in \Gamma$ such that $w_n:=d\c_ng^{t_n}g^{-t}v\to g^{-t}v$ as $n\to \infty$. Then along the geodesic $c_{w_n}$ there is a Jacobi field $Y_n$ with $(Y_n(0), Y'_n(0))=d\c_ndg^{t_n-t}(Y(0), Y'(0))$. Hence $\|Y_n(s)\|=C$ for any $s\ge 0$. Taking $n\to \infty$, $Y_n$ converges to a Jacobi field $\bar Y_{-t}$ along $c_{g^{-t}v}$ with $\|\bar Y_{-t}(s)\|=C$ for any $s\ge 0$. Then a subsequence of $\bar Y_{-t}$ converges to a Jacobi field $\bar Y$ along $c_v$ with $\|\bar{Y}(s)\|\equiv C$ for any $s\in \RR$. By \cite[Proposition 2.4]{Wat}, we have $\bar Y$ is a parallel Jacobi field, a contradiction since $v\in \text{Reg}$.

Now we claim that there exist $T>0$ and $\d>0$ such that
$$\log (\|Y(T)\|/\|Y(0)\|)<-\d$$
for any $Y\in J^s(v)$. Assume the contrary, then there exist $T_n\to \infty$, $Y_n\in J^s(v)$ with $Y_n(0)=1$ such that $\log \|Y_n(T_n)\|\ge -1/n$. Then a subsequence of $Y_n$ converges to a Jacobi field $Y\in J^s(v)$ such that $\|Y(0)\|=1$ and
$$\|Y(T)\|=\lim \|Y_{n_k}(T)\|\ge \lim \|Y_{n_k}(T_{n_k})\|\ge 1$$
for any $T\ge 0$. Since $t\mapsto \|Y(t)\|$ is nonincreasing, we have $\|Y(t)\|=1$ for any $t\ge 0$, a contradiction to the previous claim.

The remaining proof is a slight modification of the one of \cite[Lemma 3.4]{BBE}. Choose $\d>0$ as in the above claim. By continuity, there exists an open neighborhood $O$ of $v^*=dp(v)$ in $SM$ such that
$$\log (\|Y(T)\|/\|Y(0)\|)<-\d$$
for all $w'\in O$ and all $Y\in J^s(w')$. Choose a compact neighborhood $V\subset O$ of $v^*$ and let $t_n\to \infty$ satisfy $g^{t_n}v^*\in V, t_{n+1}-t_n> T$ and $t_n < n/\sigma$ for all $n$ and some $\sigma> 0$ by Lemma \ref{recrate}. Choose a neighborhood $U$ of $v$ in $W^s(v)$ so small that $g^{t_n}dp(w)\in O$ for all $n$ if $w\in U$. By the choice of $O$ we obtain $\|Y(t_n)\|< e^{-\d n}\|Y(0)\|$, where $Y\in J^s(w), w\in U$. Note that $\|Y(t)\|$ is a nonincreasing function of $t$. Given $t>0$ choose $n$ such that $t_n \le t < t_{n+1}$. Then
$$\|Y(t)\| \le \|Y(t_n)\|< e^{-\d n}\|Y(0)\| < e^\d e^{-\d\sigma t}\|Y(0)\|$$
since $t < t_{n+1} <(n+1)/\sigma$. Then find a constant $C>e^\d$ and set $\lambda:=\d\sigma$. This finishes the proof of the proposition.
%Suppose that $t_n\le t<t_{n+1}$ satisfies that $g^{t_n}dp(w)\in O$ for all $0<i\le n+1$ if $w\in \tilde U\subset W^u(v)\cap \mathcal{O}$. Then analogous argument shows that for any $\tilde Y\in J^s(w)$,
%$\tilde Y(t)\ge e^{-\d}e^{\d\sigma t}\|\tilde Y(0)\|.$
\end{proof}

\begin{proof}[Proof of Proposition \ref{contracting}]
The proof is given in \cite[Proposition 3.10]{BBE}. We emphasize that in that proof, all the estimates are explicit. Indeed, let $w\in W^s(v)$, and $\b:[0,L]\to W^s(v)$ be a $C^1$ curve parameterized by arclength such that $\b(0)=v$, $\b(L)=w$, $L\le 2d^s(v,w)$, and $\a_t(s)=c_{\b(s)}(t)$. Based on Proposition \ref{rate}, the proof of \cite[Proposition 3.10]{BBE} gives
$$\text{length}(\a_t)\le \frac{2^{m+1}\cdot L}{2^{\sigma t}}$$
for all $t\ge 0$. Here $m=\lfloor \frac{2L}{\e}\rfloor$ where $\e$ is a small positive number such that the ball of $2\e\sqrt{1+a^2}$ centered at $v$ is contained in $U$ obtained in Proposition \ref{rate}, and $\sigma$ is from Lemma \ref{recrate}.
Setting
$$\lambda:=\min\left\{\frac{\e}{4\log 2}, \sigma \log 2\right\}$$
and $C=8$, we get \eqref{e:imp}.
%the constant $C$ is computed explicitly. It depends on $L$, which in turn is chosen according to the distance between $v$ and $w$. Therefore, $C$ can be chosen to be a continuous function in $d^s(v,w)$, rather than in $w$.
\end{proof}

\begin{proof}[Proof of Lemma \ref{contracting1}]
Let $w\in W^s(v)$ be such that $\pi w=\pi w_T$. We have by \eqref{e:imp}
\begin{equation*}
\begin{aligned}
&d(\pi g^Tw_T, \pi g^Tv)=|S-T|\\
=&|d(\pi w_T, \pi g^Tv)- d(\pi w, \pi g^Tw)|\le d(\pi g^Tw, \pi g^Tv)\\
\le &Cd^s(v,w)e^{d^s(v,w)/\lambda}e^{-\lambda T}.
\end{aligned}
\end{equation*}
Then we have
\begin{equation*}
\begin{aligned}
&d_K(g^{T-t}v,g^{S-t}w_T) \le d_K( g^{T-t}v,g^{T-t}w)+d_K( g^{T-t}w,g^{S-t}w_T)\\
 \le &d_K( g^{T-t}v, g^{T-t}w)+d_K( g^{T-t}w, g^{T-t}w_T)+|S-T|\\
 \le &d_K(g^{T-t}v, g^{T-t}w)+d_K( g^{T}w, g^{T}v)+|S-T|\\
 \le &3Cd^s(v,w)e^{d^s(v,w)/\lambda}e^{-\lambda (T-t)}.
\end{aligned}
\end{equation*}
Setting $C_0=3C$, the proof of Lemma \ref{contracting1} is complete.
\end{proof}

\end{document}